%% file: main.tex
\documentclass{article}
\usepackage{blindtext}
\usepackage{graphicx}
\usepackage{arxiv}
\usepackage{personalstyle}
\usepackage{authblk}        
\usepackage{natbib}

\usepackage{lastpage}
\begin{document}
\sloppy

\title{
    High Probability and Risk-Averse Guarantees\\
    for a Stochastic 
    Accelerated Primal-Dual Method}

\author{\textbf{Yassine Laguel}\thanks{Department of Management Science and Information Systems, Rutgers Business School, Rutgers University. \\ \indent\;\; Piscataway, NJ, USA.}
\hspace{1cm}
\textbf{Necdet Serhat Aybat}\thanks{Department of Industrial and Manufacturing Engineering. Pennsylvania State University
\\
\indent\;\; University Park, PA, USA. \\
\indent\; Corresponding author: yassine.laguel@rutgers.edu}
\hspace{1cm}
\textbf{Mert Gürbüzbalaban}$^{*}$
}

\maketitle

\theoremstyle{definition}
\newtheorem{conj}{Conjecture}[section]
\newtheorem{definition}{Definition}[section]
\newtheorem{example}{Example}[section]

\newtheorem*{note}{\textbf{Note}}
\newtheorem*{exmp*}{\textbf{Examples}}

\numberwithin{equation}{section}
\newtheorem{theorem}{Theorem}
\newtheorem{thm}{\textbf{Theorem}}
\newtheorem{lem}[thm]{\textbf{Lemma}}
\newtheorem{prop}[thm]{\textbf{Proposition}}
\newtheorem{cor}[thm]{\textbf{Corollary}}
\newtheorem{corollary}[theorem]{Corollary}
\newtheorem{rema}[thm]{\textbf{Remark}}
\newtheorem{assumption}{Assumption}
\newcommand{\mtodo}[1]{}  
\newcommand{\mg}[1]{{\color{black} #1}}  

\newtheorem{innercustomassumption}{Assumption}

\begin{abstract}
We consider {stochastic} strongly-convex-strongly-concave (SCSC) saddle point (SP) problems which frequently arise in applications {ranging} from distributionally robust learning to game theory and 
fairness in machine learning.
We focus on the recently developed stochastic accelerated primal-dual algorithm (SAPD), which admits optimal complexity in several settings as an accelerated algorithm. We provide high probability guarantees for convergence to a neighborhood of the saddle point that reflects accelerated convergence behavior. We also provide an analytical formula for the limiting covariance matrix of the iterates for a class of {stochastic} SCSC quadratic problems where the gradient noise is additive and Gaussian. This allows us to develop lower bounds for this class of quadratic problems which show that our analysis is tight in terms of 
\saa{the high probability bound} dependency to the parameters. We also provide a risk-averse convergence analysis characterizing {the} ``Conditional Value at Risk'', 
the ``Entropic Value at Risk''{, and the $\chi^2$-divergence} of the distance to the saddle point, highlighting the trade-offs between the bias and the risk associated with an approximate solution obtained by terminating the algorithm at any iteration.
\end{abstract}


\section{Introduction}
\input{1_introduction}

\section{\texorpdfstring{\mg{Preliminaries and Background}}{Preliminaries and Background}}
\input{2_technical}

\section{Main Results}\label{rso:sec:high_prob_bound_sapd}
\input{3_main_results}

\section{Analytical solution for quadratics}\label{rso:sec:quadratics}
\input{4_quadratics}

\clearpage
\section{Proof of Main Results}\label{rso:sec:proof_section}

\subsection{Concentration inequalities through recursive control}\label{rso:app:recursive_control}
\input{5.1_recursive_control}

\subsection{Proofs of Theorem~\ref{rso:thm:high_probability_bound_sapd} and Theorem~\ref{thm:risk-nounds}}\label{rso:sec:proof_main_thms}
\input{5.2_proof}

\label{rso:app:sapd_conv}

\section{Numerical Results}
\input{6_numerical_exps}
\label{rso:sec:numerical_exps}

\section{Conclusion}
We consider \saa{a first-order primal-dual method that relies} on stochastic estimates of the gradients for solving SCSC saddle point problems. We focused on the stochastic accelerated primal dual (\sapdname) method~\cite{zhang2021robust}.
We obtained high-probability bounds for the iterates to lie in a given neighborhood of the saddle point that reflects accelerated behavior. For a class of quadratic SCSC problems subject to i.i.d. isotropic Gaussian noise and under a particular parameterization of the \saa{\sapdname~parameters, we were able to 
compute the distribution of the \sapdname~iterates exactly in closed form. We used this result to} show that our 
\saa{high-probability bound} is tight in terms of its dependency to target probability $p$, primal and dual stepsizes and the momentum parameter $\theta$.  We also provide a risk-averse convergence analysis characterizing the ``Conditional Value at Risk", \mgtwo{$\chi^2$-divergences} and the ``Entropic Value at Risk" of the distance to the saddle point, highlighting the trade-offs between the bias and the risk associated with an approximate solution.



\section*{Acknowledgements}
Yassine Laguel and Mert G\"urb\"uzbalaban acknowledge support from the grants Office of Naval Research 
N00014-21-1-2244, National Science Foundation (NSF) CCF-1814888, NSF DMS1723085, NSF DMS-2053485. Necdet Serhat Aybat's work was supported in part by the grant Office of Naval Research Award N00014-21-1-2271.

\vskip 0.2in
\bibliographystyle{plainnat}
{
\bibliography{optim}
}


\appendix

%
%
\section{Elementary proofs for subGaussians and convex risk measures}
\input{A1_basics}\label{rso:app:subgaussians}

%
%
\section{Intermediate results and proofs for the non-quadratic case}\label{sec:sapd_elementary_convex}
\input{A2_sapd_intermediate}

\section{Details and proofs for the quadratic setting}
\input{A3_quadratic}

\section{Symbols and constants used in the paper}\label{rso:sec:constants}
\input{A4_constants}

\clearpage

\input{B_online_material}

\end{document}

%% file: 1_introduction.tex
We consider strongly convex/strongly concave (SCSC) saddle point problems of the form:
\begin{equation}\label{rso:eq:main_problem}
    \min_{x\in \setX} \max_{y \in \setY} \obj(x,y) \triangleq f(x) + \sadobj(x,y) - g(y),
\end{equation} 
where $\mathcal{X}$ and $\mathcal{Y}$ are finite-dimensional Euclidean spaces, $f:\mathcal{X}\to\mathbb{R}\sa{\cup\{+\infty\}}$ and $g:\mathcal{Y}\times\mathbb{R}\sa{\cup\{+\infty\}}$ are closed convex functions, and $\Phi:\mathcal{X}\times\mathcal{Y}\to\mathbb{R}$ is a smooth convex-concave function such that $\cL(x,y)$ is strongly convex in $x$ and strongly concave in $y$, i.e., SCSC -- see Assumption~\ref{rso:assumption:str_cvx_smooth} for details. 
Throughout the paper, we assume that $f$ and $g$ are strongly convex; since $\cL$ is SCSC, we emphasize that this assumption is without loss of generality as strong convexity/concavity can be transferred from $\Phi$ to $f$ and $g$ by adding and subtracting simple quadratics.\looseness=-1

\yassine{Such saddle point}~\sa{(SP)} problems arise in many applications and \saa{different} contexts. In unconstrained and constrained optimization problems, saddle-point formulations arise naturally when the problems are \sa{reformulated as a minimax problem} based on the Lagrangian duality. 
Furthermore, the SP formulation in~\eqref{rso:eq:main_problem} encompasses many key problems such as \emph{robust optimization}~\citep{ben2009robust} -- here 
\sa{$g$ is selected to be the indicator function of an uncertainty set} from which nature (adversary) picks an uncertain model parameter $y$, and the objective is to choose $x\in\cX$ that minimizes the worst-case cost $\max_{y\in\cY}\cL(x,y)$, i.e., a two-player zero-sum game. 
Other applications \saa{involving SCSC problems} include but are not limited to \emph{supervised learning} with non-separable regularizers {(where $\Phi(x,y)$ may not be bilinear)} \citep{palaniappan2016stochastic}, \emph{fairness} in machine learning \citep{liu2022partial}, \emph{unsupervised learning} \citep{palaniappan2016stochastic} and various \emph{image processing} problems, e.g., denoising, \citep{chambolle2011first}.

In this work, we are interested in 
SCSC problems where the partial gradients \sa{$\nabla_x \Phi$ and $\nabla_y \Phi$} are not deterministically available; \sa{but, instead we postulate that their stochastic estimates $\widetilde\grad_x\Phi$ and $\widetilde\grad_y\Phi$ are accessible}. 
Such a setting arises frequently in large-scale optimization and machine learning applications where the gradients are estimated from either streaming data or from random samples of data (see e.g. \citep{zhu2023distributionally,gurbuzbalaban2022stochastic,bottou2018optimization}). 
First-order (FO) methods that rely on stochastic estimates of the gradient information have been the leading computational approach for computing low-to-medium-accuracy solutions \mg{for these problems} because of their cheap iterations and mild dependence on the problem dimension and data size. 
In this paper, our focus will be on first-order primal-dual algorithms that rely on stochastic gradient estimates for solving \eqref{rso:eq:main_problem}.

\paragraph{Existing relevant work.} 
Stochastic primal-dual algorithms for solving SP problems generate a sequence of primal and dual iterate pairs $z_n = (x_n, y_n) \in\cX\times\cY\triangleq\cZ $ starting from an initial point $(x_0, y_0) \in\dom f\times \dom g\triangleq Z$. 
Two popular metrics to \saa{assess} the quality of a random solution $(\hat{x}, \hat{y})$ returned by a stochastic algorithm are the \emph{expected gap} and the \sa{\emph{expected squared distance}} defined as  
\begin{equation}
\label{eq:gap}
\cG(\hat{x},\hat{y})\triangleq\mathbb{E}\big[ \sup_{(x,y)\in \cX\times \cY} \big\{ \mathcal{L}(\hat{x}, y)  - \mathcal{L}(x, \sa{\hat{y}}) \big\} \big],\quad
{\mathcal{D}(\hat{x},\hat{y})\triangleq\mathbb{E}[\norm{\hat{x}-\optx}^2+\norm{\hat{y}-\opty}^2]},
\end{equation}
respectively, where \saa{$(\optx, \opty)$} denotes the \textit{unique} saddle point of \eqref{rso:eq:main_problem}, due to the strong convexity of $f$ and $g$. 
The iteration complexity of \sa{FO-methods} in these two metrics depend naturally on the block Lipschitz constants $L_{xx}$, \saa{$L_{xy}$}, $L_{yy}$ and \saa{$L_{yx}$}, i.e., Lipschitz constants of $\nabla_x\Phi(\cdot, y)$, \saa{$\nabla_x\Phi(x, \cdot)$,} $\nabla_y\Phi(x, \cdot)$ and $\nabla_y\Phi(\cdot, y)$ as well as on the strong convexity constants $\mu_x$ and $\mu_y$ of the functions $f$ and $g$. 
In particular, \cite{fallah2020optimal} show that a multi-stage variant of Stochastic Gradient Descent Ascent (SGDA) algorithm \saa{generates $(x_\epsilon,y_\epsilon)$ such that $\mathcal{D}(x_\epsilon,y_\epsilon)\leq \epsilon$ within $\cO(\kappa^2\ln(1/\epsilon)+\frac{\delta^2}{{\mu^2}}\frac{1}{\epsilon})$ gradient oracle calls,} where $\delta^2 = \max\sa{\{}\delta_x^2,\delta_y^2\sa{\}}$, while $\delta_x^2$ and $\delta_y^2$ are bounds on the variance of the stochastic gradients \sa{$\widetilde\grad_x\Phi$ and $\widetilde\grad_y\Phi$,} respectively; $\mu\triangleq\min\{\mu_x,\mu_y\}$ and $L\triangleq\max\{L_{xx},\saa{L_{xy}, L_{yx}},L_{yy}\}$ are the worst-case strong convexity and Lipschitz constants and $\kappa \triangleq L/\mu$ is defined as the \emph{condition number}. 
SGDA \sa{analyzed in \citep{fallah2020optimal}} consists of Jacobi-\sa{type} updates in the sense that stochastic gradient descent and ascent steps are taken simultaneously. 
Jacobi-\sa{type} updates are easier to analyze than Gauss-Seidel updates in general, and can be viewed as solving a structured variational inequality~\sa{(VI)} problem, \sa{for which there are many existing techniques that} directly apply, \sa{e.g.,} \citep{gidel2018variational,chen2017accelerated}. 
\cite{zhang2022near} consider 
deterministic SCSC problems, when gradient descent ascent (GDA) with Gauss-Seidel-\sa{type} updates \sa{is considered, \saa{i.e.,} the primal and dual variables are updated in an alternating fashion using the most recent information obtained from the previous update step.} 
{Their results show that} an accelerated \sa{asymptotic} convergence rate, i.e., iteration complexity scales linearly with $\kappa$ instead of $\kappa^2$, can be obtained \sa{for the Gauss-Seidel variant of GDA}.
However, as discussed in \citep{zhang2022near}, this comes with the price that Gauss-Seidel style updates greatly complicate the analysis because every iteration of an alternating algorithm is a composition of two half updates. 
\sa{Furthermore, extending the acceleration result to non-asymptotic rates requires using momentum terms in either the primal or the dual updates, and this further complicates the convergence analysis.} 
\sa{\cite{fallah2020optimal} also considered using momentum terms both in the primal and dual updates, and show that the multi-stage \emph{Stochastic Optimistic Gradient Descent Ascent} (OGDA) algorithm using Jacobi-type updates} achieves an iteration complexity of $\cO(\kappa\ln(1/\epsilon)+\frac{\delta^2}{\sa{\mu^2}\epsilon})$ in expected \sa{squared} distance.
\mg{There are also several other algorithms that can achieve the accelerated rate, \saa{i.e., $\log(1/\epsilon)$ has the coefficient} $\kappa$ instead of $\kappa^2$ - see, e.g.~\citep{beznosikov2022smooth}. 
We call this term that depend on the condition number as \emph{initialization bias} since it captures how fast the error due to initial conditions decay and reflects the behavior of the algorithm in the noiseless setting.} 
\sa{Among the algorithms that achieve an accelerated rate, the most closely related work to ours is \citep{zhang2021robust} which} \mg{develops} a stochastic accelerated primal-dual (SAPD) algorithm with Gauss-Seidel type updates. SAPD \sa{using a momentum acceleration only in the dual variable} \saa{can generate $(x_\epsilon,y_\epsilon)$ such that $\mathbb{E}[\mu_x\norm{x_\epsilon-\optx}^2+\mu_y\norm{y_\epsilon-\opty}^2]\leq \epsilon$ within}
$\mathcal{O}\Big( \Big(
\frac{L_{xx}}{\mu_x} + \frac{L_{yx}}{\sqrt{\mu_x\mu_y}} + \frac{L_{yy}}{\mu_y}
+ \Big( \frac{\delta_x^2}{{\mu_x}} +  \frac{\delta_y^2}{{\mu_y}} \Big)\frac{1}{\epsilon}\Big) \log(\frac{1}{\varepsilon}) \Big)
$ \saa{iterations;} \sa{this result implies ${\tilde{\cO}(\kappa\ln(\varepsilon^{-1})+ \sa{\mu^{-2}}\delta^2 \varepsilon^{-1})}$ in expected \sa{squared} distance.} 
This complexity is optimal for bilinear problems. 
To our knowledge, SAPD is also the fastest single-loop algorithm for solving stochastic SCSC problems \sa{in the form of \eqref{rso:eq:main_problem}} that are non-bilinear. 
\sa{Furthermore, using acceleration only in one update, as opposed to in both variables~\citep{fallah2020optimal}, leads to smaller variance accumulation \mg{(see \cite{zhang2021robust} for more details)}. }

\begin{figure}[t]
\centering
\includegraphics[width=0.99\textwidth, 
                trim={4cm 0 1.5cm 0},]{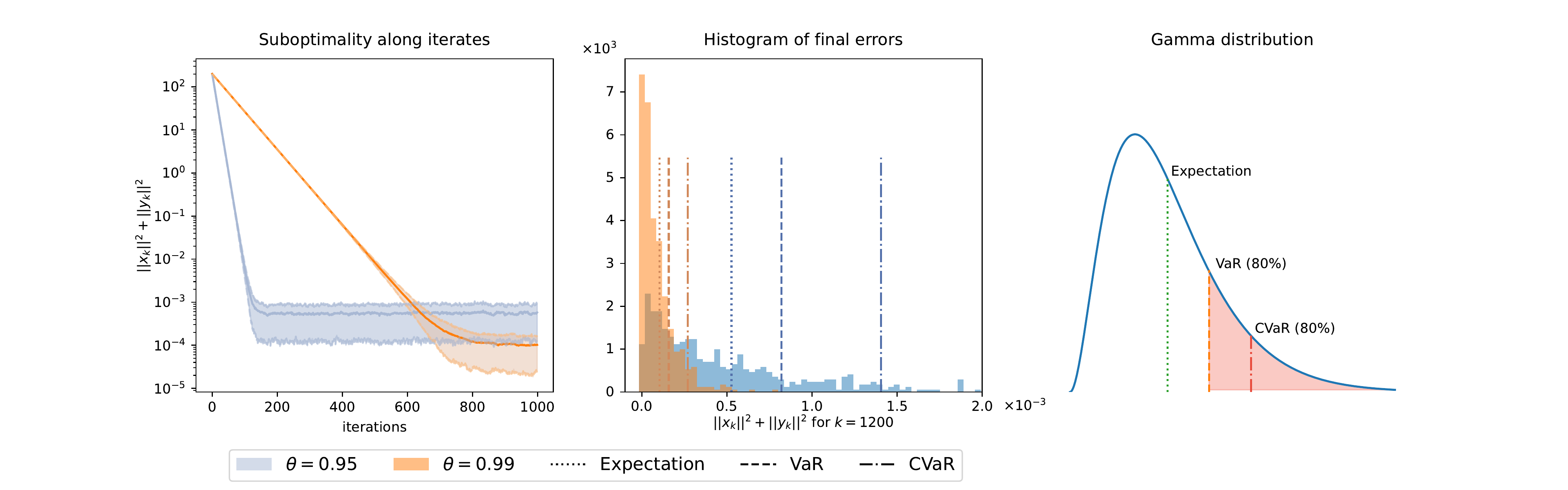}
\caption{
\textbf{\saa{(Left)}} Convergence of SAPD on the saddle point problem $\min_{x\in\mathbb{R}} \max_{y\in\mathbb{R}} x^2/2 + xy + y^2/2$, initialized at $x_0=y_0=10$ with momentum parameters $\theta=0.95$ and $\theta=0.99$. 
\textbf{\saa{(Middle)}} Histogram of the distribution of the SAPD iterates $(x_n, y_n)$ after $n=1000$ iterations for 500 runs, with corresponding momentum parameters $\theta=0.95$ and $\theta=0.99$.
\textbf{\saa{(Right)}} Illustration of the expectation $\mathbb{E}(X)$, $p$-th quantile ($\mbox{VaR}_{1-p}(X)$) and $\mbox{CVaR}_p(X)$ of a gamma distributed random variable $X$ with shape parameter $3$ and scale parameter $5$ for $p=80\%$.}
\label{fig:gamma-density} 
\end{figure}
\vspace*{-1mm}

\mg{While the aforementioned results provide performance guarantees in expectation based on the metrics \saa{defined in}~\eqref{eq:gap} and their variants, unfortunately} \sa{having guarantees in these} metrics 
do not allow us to control tail events, i.e., the expected gap and distance can be smaller than a given target threshold $\varepsilon>0$; but, the iterates can still be arbitrarily far away from the saddle point with a non-zero probability. 
In this context, high probability guarantees are key in the sense that they allow us to control tail probabilities and quantify how many iterations are needed for the iterates to be in a neighborhood of the saddle point with a given probability level $p\in(0,1)$. 
\mg{This is illustrated on the right panel of Figure \ref{fig:gamma-density}} where we run the SAPD algorithm for two different values of the momentum parameter $\theta=0.95$ and $\theta=0.99$ for a toy problem $\min_{x\in\mathbb{R}} \max_{y\in\mathbb{R}} \frac{1}{2}x^2 + xy + \frac{1}{2}y^2$  with strong convexity parameters $\mu_x=\mu_y=1$ admitting a saddle point at $\optx=\opty=0$. 
SAPD is initialized at $(x_0,y_0)=(10,10)$, primal and dual stepsizes are chosen according to the Chambolle-Pock parametrization as suggested in \citep{zhang2021robust}. 
For simplicity, the stochastic gradients $\widetilde\nabla_x\Phi$ and $\widetilde\nabla_y\Phi$ are \saa{set to $\nabla_x\Phi$ and $\nabla_y\Phi$ perturbed with} additive i.i.d. Gaussian $\cN(0, 0.1)$ noise, \saa{and we assume that $\widetilde\nabla_x\Phi$ and $\widetilde\nabla_y\Phi$ are independent from each other and also} from the past history of the algorithm.
For each parameter choice, we run SAPD \saa{for 500 runs, and for $n=1000$ iterations for every run --see} \saa{Figure~\ref{fig:gamma-density} (\textbf{Left}),} and plot the distribution of the squared distance \saa{of $(x_n,y_n)$ to the unique saddle point $(\optx,\opty)$, i.e., $E_n \triangleq \saa{\|x_n - \optx\|^2 + \|y_n-\opty\|^2}$, in Figure~\ref{fig:gamma-density} (\textbf{Middle})}. 
As we can see, the random error $E_n$ can take significantly larger values than its expectation $\mathcal{D}(x_n, y_n) = \mathbb{E}[E_n]$. 
This motivates estimating the $p$-th quantile of the error $E_n$, which is also called the \emph{value at risk} at \saa{level $p$, {traditionally abbreviated as $\var_{p}(E_n)$ in the financial literature.}
While $\var_{p}(E_n)$ represents a worst-case error $E_n$ associated with a probability $p$, it does not capture the behavior if that worst-case threshold is ever crossed. 
Conditional value at risk (CVaR) at level $p$, on the other hand, is an alternative risk measure that can be used characterizing the expected error if that worst-case threshold is ever crossed. 
CVaR is in fact a coherent risk measure with some desirable properties~\citep{rockafellar2013superquantiles}. 
In \saa{Fig.~\ref{fig:gamma-density}(\textbf{Middle}),} we report $\var_{p}(E_n)$ and $\cvar_{p}(E_n)$ for $p=80\%$ and $n=1000$. We can see that $\var$ and $\cvar$ capture the tail behavior better compared to expectation. 
Similar behavior can be seen on other distributions, \saa{e.g., in Figure~\ref{fig:gamma-density}(\textbf{Right})}, we illustrate the expectation, \saa{$\var_p$, i.e., $p$-th quantile,} and $\cvar_{p}$ of a gamma-distributed random variable with shape parameter $3$, and scale parameter $5$ corresponding to $p=\%80$. 
We see that $\cvar$ estimates the average of the tails after the $p$-th quantile; therefore, it is useful for capturing the average risk associated to tail events beyond the $p$-th quantile. 
In addition to CVaR, there are also other coherent risk measures such as entropic value at risk (EVaR) and $\chi^2$-divergence which have been of interest in the study of stochastic optimization algorithms as they can provide risk-averse guarantees capturing the worst-case tail behavior and deviations from the mean performance \citep{can2022entropic}}.\looseness=-1

While high-probability guarantees (VaR guarantees) \mg{and risk guarantees in terms of risk measures such as CVaR and EVaR} are available in the optimization setting for the iterates of stochastic gradient descent-like methods \citep{harvey2019tight,rakhlin2012making,davis2020high,can2022entropic}, \sa{results in a similar nature} are \mg{considerably} more limited in the SP setting. 
Among existing results, \mg{\cite{juditsky2011solving} obtained high-probability guarantees for the stochastic mirror-prox algorithm for solving stochastic VIs with Lipschitz and monotone operators. 
This algorithm can be used to solve smooth stochastic convex/concave SP problems (which corresponds to the case $f=g=0$ with \saa{$\Phi(x,y)$ being smooth, convex in $x$ and concave in $y$}) and implies that with probability $p\in(0,1)$, after $n$ iterations, the gap metric for the VI will be bounded by $\mathcal{O}(\frac{1}{\sqrt{n}} + \log(\frac{1}{1-p})\frac{1}{n})$ assuming that the domain is bounded and the stochastic gradient noise is light-tailed with a sub-Gaussian distribution. 
In \citep{gorbunov2022clipped}, it is shown  that the same high-probability results as in \citep{juditsky2011solving} can be attained by clipping the gradients properly without the sub-Gaussian and bounded domain assumptions.}
In another line of work \citep{yan-high-proba}, it is shown for the \mg{SGDA} algorithm that the expected gap $\mathcal{G}(x_n, y_n)\leq \varepsilon$ with probability at least $p\in(0,1)$ after $n=\mathcal{O}\left(\frac{1}{\varepsilon}\log(\frac{1}{1-p}) + \frac{\delta^2}{\mu\varepsilon}\log(\frac{1}{1-p})\right)$ \sa{oracle calls} for possibly non-smooth SCSC problems. 
Also, in \citep{wood-dallanese}, high probability bounds are given for online algorithms applied to stochastic saddle point problems where the objective is time-varying and is revealed in a sequential manner, and the data distribution over which stochastic gradients are estimated depends on the decision variables. 
However, these high-probability guarantees are obtained for \textit{non-accelerated} algorithms \mg{with Jacobi-style updates}; therefore, the high probability bounds do not \sa{exhibit} accelerated decay of the initialization \sa{bias}, {and scale as $\kappa^2$, i.e., quadratically with the condition number $\kappa$, instead of a linear scaling}. 
\mg{Also, to our knowledge, high-probability bounds for SP algorithms with Gauss-Seidel style updates are not available in the literature even if they do not incorporate momentum, see e.g. the survey by \cite{beznosikov2022smooth}). Similarly, we are not aware of any risk guarantees (in terms of CVaR and EVaR of the performance metric over iterations) for any \saa{primal-dual algorithm for solving stochastic SP problems}.}
\paragraph{Contributions.}
In this paper, we present a risk-averse analysis of the \sapdname~method~\citep{zhang2021robust} to solve saddle point problems of the form~\eqref{rso:eq:main_problem}.
A key novelty of our work lies in providing the first analysis of an accelerated algorithm for SCSC problems with high probability guarantees, where our bounds reflect the accelerated decay of the initialization bias scaling linearly with the condition number $\kappa$. 
More specifically, our high-probability bounds provided in Section \ref{rso:sec:high_prob_bound_sapd} imply that given target accuracy $\varepsilon>0$, \sapdname, with a proper choice of parameters that we \ylfourth{explicit,} can generate a solution $(x_n, y_n)$ that satisfies $\mu_x \|x_n - \optx\|^2 + \mu_y \|y_n - \opty\|^2 \leq \varepsilon$ with probability $p \in (0,1)$ after 
\begin{equation}
    \scalebox{0.9}{$ 
 n=   \bigOh
    \Bigg(
        \bigg[
            \frac{\Lxx}{\mux} + \frac{\Lyx}{\sqrt{\mux\muy}} + \frac{\Lyy}{\muy} + \left(1 + \frac{\Lxy}{\Lyx} + \frac{\Lxy^2}{\Lyx^2}\right) 
            \begin{aligned}[t]
                \max\left(\frac{\proxyx^2}{\mux}, \frac{\proxyy^2}{\muy}\right) 
            \frac{\ylfourth{1 + \log\left(\frac{1}{1-p}\right)}}{\varepsilon}
        \bigg]  
                \log\bigg(
                \frac{(1 + \frac{\Lxy^2}{\Lyx^2}) \mathcal{W}_0
                }{\varepsilon}
            \bigg)
    \Bigg)
            \end{aligned}$}\label{intro-complexity}
\end{equation}
iterations where $\mathcal{W}_0 \triangleq\mux \squarednormtwo{x_0 - \optx} + \muy \squarednormtwo{y_0 - \opty}$. 
{When the partial gradients $\nabla_x \Phi$ and $\nabla_y \Phi$ are continuously differentiable (which is the case for bilinear problems and for many SCSC problems arising in practice \citep{zhang2021robust,palaniappan2016stochastic,chambolle2011first})}, then we can take $\Lxy = \Lyx$ (as discussed in Remark \ref{rema-Lxy-Lyx-equal}) and the complexity \eqref{intro-complexity} simplifies to 
{\small
\[    
\scalebox{1.}{$n= \bigOh
                    \left(
                        \left[\frac{\Lxx}{\mux} + \frac{\Lyx}{\sqrt{\mux\muy}} + \frac{\Lyy}{\muy} +  \max\left(\frac{\proxyx^2}{\mux}, \frac{\proxyy^2}{\muy}\right) \frac{\ylfourth{1 + \log\left(\frac{1}{1-p}\right)}}{\varepsilon}\right]  \log\left(\frac{1}{\varepsilon}\right)
                    \right)
            $},
\]
}%
hiding constants depending on the initialization. 
Simplifying the terms further, this implies $n = \mathcal{O}\left(\kappa \log(\frac{1}{\varepsilon}) + (1 + \log(\frac{1}{1-p})) \frac{\delta^2 \log(1/\varepsilon)}{\mu\varepsilon}\right)$ iterations are sufficient where $\delta^2 = \max(\delta_x^2, \delta_y^2)$. 
To achieve this, under a light-tail (subGaussian) assumption for the norm of the gradient noise, we develop concentration inequalities tailored to the specific \emph{Gauss-Seidel} structure of ~\sapdname. 
In particular, the Gauss-Seidel type updates and the use of a momentum term complicate the analysis significantly where the evolution of the iterates and the performance metric over the iterations need to be studied with respect to a non-standard filtration for having the right measurability properties (as discussed in Section \ref{rso:sec:proof_main_thms} in detail). 
A crucial step for the development of these results is a new Lyapunov function $V_n$ we construct that has favorable contraction properties.
To our knowledge, these are the first high-probability guarantees for a SP algorithm with Gauss-Seidel style updates and first high-probability guarantees that are accelerated in the sense that the dependency of the iteration complexity to the initialization bias scales linearly with the condition number $\kappa$. 
We also provide finite-time risk guarantees, where we measure the risk in terms of the CVaR, EVaR and $\chi^2$-divergence of the distance to the saddle point. 
In addition, we provide an in-depth analysis of the behavior of~\sapdname~on a class of quadratic problems subject to i.i.d. isotropic Gaussian noise where we can characterize the behavior of the distribution of the iterates explicitly. 
In particular, we derive an analytical formula for the limiting covariance matrix of~\sapdname's iterates, which demonstrates the tightness of our high probability bounds with respect to several parameter choices in \sapdname. 
To our knowledge, these are the first risk-averse guarantees that quantify the risk associated \saa{with an \textit{approximate} solution generated by a primal-dual algorithm for SP problems.}\looseness=-1 
\paragraph{Notations.}
Throughout this manuscript, $\cX=\Rn$ and $\cY=\Rm$ denote finite dimensional vector spaces equipped with the Euclidean norm $\norm{u} \defineq \dotproduct{u}{u}^{\frac{1}{2}}$, and $\cZ\triangleq\cX\times\cY$. We adopted $\integers_{++}$ for \mg{positive integers} and $\integers_+=\integers_{++}\cup\{0\}$. 
For $A,B \in  \R^{n\times n}$, we denote $\vectorize(A) \in \R^{n^2}$ the vector composed of the \sa{vertical} concatenation of the columns of $A$, and $A\otimes B$ the \sa{Kronecker} product of $A$ and $B$. 
We let 
$\norm{A}$ denote the 
spectral norm of $A$ and let $\rho(A)$ denote the \textit{spectral radius} of $A$, i.e., \mg{the largest modulus of the eigenvalues of $A$}.
For a finite sequence of reals $x_1, \dots, x_n$ (resp. matrices $X_1, \dots, X_n$), we denote $\diag(x_1 \dots, x_n)$ (resp. $\diag(X_1 \dots, X_n)$) the  associated (block) diagonal matrix. 
If $A$ is diagonalizable, $\spectrum(A)$ denotes the set of the eigenvalues of $A$. 
For any convex set $C$, \sa{$\cvxindicator_C$ denotes the indicator function of $C$, i.e., $\cvxindicator_C(x)=0$ if $x\in C$, and equal to $+\infty$ otherwise.}
For a given proper, closed and convex function  $\varphi \from \setX \to \R \cup\{+\infty\}$, $\prox{\varphi}(\cdot)$ denotes the associated \textit{proximal operator}: $x \mapsto \argmin_{u \in \setX} \varphi(u) + \frac{1}{2} \squarednormtwo{u - x}$. We use the Landau notation to describe the asymptotic behavior of functions. 
That is, for $u\in \R\cup \{\pm \infty\}$, a function $f(x) = \littleOh(g(x))$ in a neighborhood of $u$ if $\frac{f(x)}{g(x)} \to 0$ as $x \to u$, whereas $f(x) = \bigOh(g(x))$ if there exist a positive constant $C$ such that $|f(x)| \leq C|g(x)|$ in some neighborhood of $u$. 
Similarly, we say $f(x)=\Theta(g(x))$, if $f(x) = \bigOh(g(x))$ and $g(x) = \bigOh(f(x))$. 
\mg{Given random vectors $U_n:\Omega \to \R^d$ for $n\geq 0$, we let $U_n \convergesindistributionto U$ if $U_n$ converges in distribution to another random vector $U:\Omega \to \mathbb{R}^d$. \looseness=-1}

%% file: 2_technical.tex
\subsection{Stochastic Accelerated Primal-Dual (SAPD) Method}\label{rso:sec:admissible_params}

SAPD, displayed in Algorithm \ref{rso:algo:sapd_algorithm}, is a stochastic accelerated primal-dual method developed in \citep{zhang2021robust} which uses stochastic estimates $\tilde \nabla _x \Phi$ and $\tilde \nabla _y \Phi$ of the partial gradients $\nabla _x \Phi$ and $\nabla _y \Phi$. 
SAPD extends the accelerated primal-dual method~(APD) proposed in \citep{hamedani2021primal} to the stochastic setting, which itself is an extension of the Chambolle-Pock (CP) method developed for bilinear couplings $\Phi$. 
Given primal and dual stepsizes $\tau$ and $\sigma$ and a number of iterations $n$, SAPD applies momentum averaging to the partial gradients with respect to \sa{the dual variable}, and \sa{updates the primal and the dual variables in an alternating fashion computing proximal-gradient steps.} 
\begin{algorithm}[h!]
	\caption{SAPD Algorithm}
	\label{rso:algo:sapd_algorithm}
{\small
\begin{algorithmic}[1]
        \Require{Parameters $\tau, \sigma, \theta$. Starting point $(x_0, y_0)$. Horizon \saa{$n$}.}
        \Initialize{
        $x_{-1} \leftarrow x_0$,\quad $y_{-1} \leftarrow y_0$,\quad $\approxq_0\gets \mathbf{0}$\\
        }
		\For{$k \geq 0$}
		    \State $\approxs_k \leftarrow \approxgradphiy(x_k, y_k, \omega^y_k) + \theta \approxq_k$ \Comment{Momentum averaging}
		    \State $y_{k+1} \leftarrow \prox{\sigma g}(y_k + \sigma \approxs_k)$
		    \State $x_{k+1} \leftarrow \prox{\tau f}(x_k - \tau \approxgradphix(x_k, y_{k+1}, \omega^x_{k}))$
                \State $\approxq_{k+1} \leftarrow \approxgradphiy(x_{k+1}, y_{k+1}, \omega^y_{k+1}) - \approxgradphiy(x_{k}, y_{k}, \omega^y_{k})$
		\EndFor
		\Return \saa{$(x_n,y_n)$}
\end{algorithmic}}%
\end{algorithm}
The \sa{high-probability convergence guarantees} of SAPD derived in this paper {rely} on several standard assumptions on $f$, $g$, $\Phi$, and the noisy estimates $\tilde \nabla _x \Phi$ and $\tilde \nabla _y \Phi$ of the partial gradients of $\Phi$.
\sa{The} first assumption \sa{on} the smoothness properties of the coupling function $\Phi$ is standard for first-order methods (see e.g..~\citep{mokhtari2020unified,gidel2018variational,ZhangJ_2021}). 
\begin{assumption}\label{rso:assumption:str_cvx_smooth} 
$f \from \cX \to \R \cup\{+\infty\}$ and $g \from \cY \to \R \cup\{+\infty\}$ are strongly convex with convexity modulii $\mu_x,\mu_y>0$, respectively; and $\sadobj \from \Rd \times \Rd \to \R$ is continuously differentiable on an open set containing $\dom f\times \dom g$ such that
    \begin{enumerate}
        \item[(i)] $\sadobj(\cdot, y)$ is convex on $\dom f$, for all $y \in
        \dom g$;
        \item[(ii)] $\sadobj(x, \cdot)$ is concave on $\dom g$, for all $x \in 
        \dom f$;
        \item[(iii)] there exist $\Lxx, \Lyy \geq 0$ and $\Lxy, \Lyx >0$ such that
            \[
            \begin{split}
                \norm{\gradphix(x,y) - \gradphix(\bar x, \bar y)} &\leq \Lxx \norm{x - \bar x} + \Lxy \norm{y - \bar y}, \\
                \norm{\gradphiy(x,y) - \gradphiy(\bar x, \bar y)} &\leq \Lyx \norm{x - \bar x} + \Lyy \norm{y - \bar y},
            \end{split}
            \]
        for all $(x,y), (\bar x, \bar y) \in \dom f\times \dom g$.
    \end{enumerate}
\end{assumption}
\sa{By strong convexity/strong concavity of $\cL$ from Assumption~\ref{rso:assumption:str_cvx_smooth}, the problem in~\eqref{rso:eq:main_problem} admits a unique saddle point $\optz \defineq (\optx, \opty)$ which satisfies
\begin{equation}\label{rso:eq:solution}
 \obj(\optx, y) \leq \obj(\optx, \opty) \leq \obj(x, \opty)\quad \forall (x,y) \in 
    \cX \times \cY.
\end{equation}
}
Following the literature on stochastic saddle-point algorithms~\citep{nemirovski2009robust,juditsky2011solving,chen2017accelerated}, we assume that only (noisy) stochastic estimates $\approxgradphiy(x_k, y_k, \omega_k^y), \approxgradphix(x_k, y_{k+1}, \omega_k^x)$ of the partial gradients $\gradphiy(x_k, y_k), \gradphix(x_k, y_{k+1})$ are available, where $\omega_k^x, \omega_k^y$ are random variables that are being revealed sequentially. 
{Specifically}, we let $\left(\omega^x_k\right)_{k \geq 0}$,  $\left(\omega^y_k\right)_{k\geq 0}$ be two sequences of random variables revealed in the following order \mg{in time which is the natural order for the SAPD updates}: 
\[
    \omega_0^y \rightarrow \omega_0^x \rightarrow \omega_1^y \rightarrow \omega_1^x \rightarrow \omega_2^y \rightarrow \cdots \;,
\]
and we let $(\sigmafield_{k}^{y})_{k\geq 0}$ and $(\sigmafield_{k}^{x})_{k\geq 0}$ denote the associated filtrations, \sa{i.e., $\sigmafield_{0}^y = \sigma(\omega_{0}^y)$ \mg{is the sigma algebra generated by the random variable $\omega_{0}^y$}, $\sigmafield_{0}^x = \sigma(\omega_{0}^y, \omega_{0}^x)$ and}
\[
\begin{split}
    \sigmafield_{k}^y = \sigma(\sigmafield_{k-1}^x, \sigma(\omega_{k}^y)), 
        \quad \sigmafield_{k}^x = \sigma(\sigmafield_{k}^y, \sigma(\omega_{k}^x)),\quad \forall~k\geq 1. 
\end{split}    
\]
For any $k\geq 0$, we introduce the following random variables 
to represent the gradient noise: 
\[
\noisegrady{k} \triangleq \approxgradphiy(x_k,y_{k},\omega_k^y) - \gradphiy(x_k,y_{k}),\quad
\noisegradx{k} \triangleq \approxgradphix(x_k,y_{k+1}, \omega_k^x) - \gradphix(x_k,y_{k+1}).
\]
Often times, stochastic gradients are assumed to be unbiased with a bounded variance conditional on the history of the iterates. 
Such an assumption is standard in the study of stochastic optimization algorithms and stochastic approximation theory \citep{harold1997stochastic} and frequently arises in the context of stochastic gradient methods that estimate the gradients from randomly sampled subsets of data \citep{bottou2018optimization}. 
\begin{assumption}\label{rso:assumption:unbiased}
For any $k \geq 0$, there exists scalars $\ylfourth{\nu_x}, \ylfourth{\nu_y}>0$ such that 
      $$  \expectation\left[\noisegrady{k} | \sigmafield_{k-1}^x \right] = 0, \quad
        \expectation\left[ \noisegradx{k} |\sigmafield_{k}^y \right] = 0, \quad 
        \expectation\left[ \| \noisegrady{k}\|^2 | \sigmafield_{k-1}^{x} \right] \leq \ylfourth{\nu_y}^2, \quad \expectation\left[ \| \noisegradx{k}\|^2 |\sigmafield_{k}^y \right] \leq \ylfourth{\nu_x}^2.$$          
\end{assumption}   
Under Assumptions \ref{rso:assumption:str_cvx_smooth} and \ref{rso:assumption:unbiased}, given a set of parameters $(\tau,\sigma,\theta)$, \sapdname~iterates $(x_k, y_k)$ {were shown to converge to a neighborhood of the solution linearly in expectation where the size of the neighborhood gets smaller when the gradient noise levels $\ylfourth{\nu_x},\ylfourth{\nu_y}\geq 0$ gets smaller~\citep{zhang2021robust}; in particular, in the absence of noise (when $\ylfourth{\nu_x}=\ylfourth{\nu_y}=0$)}, \mg{the iterates $(x_k, y_k)$ converges to \saa{$(\optx, \opty)$}} 
at a linear rate $\rho \in (0,1)$ provided that there exists some $\alpha \in [0, \sigma^{-1})$ for which the following inequality holds: 
{\small
\begin{equation}\label{rso:eq:matrix_inequality_params}
    \left(\begin{array}{ccccc}
        \frac{1}{\tau}+\mu_x-\frac{1}{\rho \tau} & 0 & 0 & 0 & 0 \\
        0 & \frac{1}{\sigma}+\mu_y-\frac{1}{\rho \sigma} & \left(\frac{\theta}{\rho}-1\right) L_{y x} & \left(\frac{\theta}{\rho}-1\right) L_{y y} & 0 \\
        0 & \left(\frac{\theta}{\rho}-1\right) L_{y x} & \frac{1}{\tau}-L_{x x} & 0 & -\frac{\theta}{\rho} L_{y x} \\ 
        0 & \left(\frac{\theta}{\rho}-1\right) L_{y y} & 0 & \frac{1}{\sigma}-\alpha & -\frac{\theta}{\rho} L_{y y} \\
        0 & 0 & -\frac{\theta}{\rho} L_{y x} & -\frac{\theta}{\rho} L_{y y} & \frac{\alpha}{\rho}
    \end{array}\right) \succeq 0.
\end{equation}}%
An important class of solutions to the matrix inequality in \eqref{rso:eq:matrix_inequality_params} takes the following form: \sa{Given \saa{an arbitrary} $c\in(0,1]$, choose}
\begin{equation}\label{rso:eq:cb_params} 
    \tau = \frac{1-\theta}{\theta \mu_x}, \quad \sigma = \frac{1-\theta}{\theta \muy}, \quad \theta \geq \sa{\bar{\theta}_c},
\end{equation}
for some $\sa{\bar{\theta}_c}\in(0,1)$ \yassine{explicitly given in~\citep[Corollary 1]{zhang2021robust}} \sa{-- $(\tau,\sigma,\theta)$ satisfying \eqref{rso:eq:cb_params} solves~\eqref{rso:eq:matrix_inequality_params} with $\rho=\theta$ and $\alpha=\frac{c}{\sigma}-\sqrt{\theta}L_{yy}$} {with $c \in (0,1]$}. 
\sa{SAPD generalizes the primal-dual algorithm~CP proposed in~\citep{chambolle2011first} -- CP algorithm can solve SP problems with a \textit{bilinear} coupling function $\Phi$ when a deterministic first-order oracle for $\Phi$ exists; indeed, for bilinear coupling functions \saa{with deterministic first-order oracles}, SAPD reduces to the CP algorithm.} 
It is shown in~\citep{chambolle2011first} that for a particular value\footnote{see \citep[\sa{Eq.(48)}]{chambolle2011first}.} of $\theta$, \sa{the choice of primal and stepsizes $(\tau,\sigma)$ according to \eqref{rso:eq:cb_params}} achieves acceleration of the CP algorithm proposed in~\citep{chambolle2011first}. 
\sa{For SAPD, \cite{zhang2021robust} {study the squared distance of iterates to the saddle point in expectation} and extends the same acceleration result to the case when $\Phi$ is not bilinear and when one has only access to a stochastic first-order oracle rather than a deterministic one.}

As we focus on SCSC problems, we can rely on the squared distance of the iterates $(x_n, y_n)$ to the solution $(\optx, \opty)$ to quantify sub-optimality. 
Precisely, sub-optimality will be measured in terms of \saa{a} weighted squared distance to the solution, i.e., 
\begin{equation}\label{def-metric}
    \cD_n \defineq \frac{1}{2\tau} \|x_{n} - \optx \|^2 + \frac{1}{2} \Big( 
    \frac{1}{\sigma}-\alpha\Big) \|y_n - \opty\|^2,
\end{equation} 
for some $\alpha \in [0, \sigma^{-1})$. 
\mg{This weighted metric turns out to be more convenient for the convergence analysis of SAPD, but it is clearly equivalent to the unweighted squared distance $E_n = \|x_{n} - \optx \|^2 + \| y_n - \opty \|^2$ up to a constant that depends on the choice of $(\tau, \sigma,\alpha)$}. 
\saa{For the sake of completeness of the paper,} we first recall the convergence of \sapdname~in expected {weighted} squared distance, established in~\citep{zhang2021robust}.
\begin{thm}[\citep{zhang2021robust}, Theorem 1]\label{rso:thm:sapd_expectation_result}
    \mg{Suppose Assumptions \ref{rso:assumption:str_cvx_smooth} and \ref{rso:assumption:unbiased} hold} and let $z_n=(x_n,y_n)_{n\geq 1}$ be the iterates generated by \sapdname, starting from an arbitrary tuple $(x_0, y_0) \in \setX \times \setY$. 
    For all $n \in \N$, $p \in (0,1)$ and $\tau, \sigma > 0, $ and $\theta\geq 0$ satisfying~\eqref{rso:eq:matrix_inequality_params} for some $\rho\in (0,1)$ and $\alpha \in [0, \sigma^{-1})$, it holds that
    \begin{equation}\label{rso:eq:conv_sapd_equation}
        \expectation\left[\cD_n\right] 
            \leq \rho^n\; \startingbias +\frac{\rho}{1-\rho} \left(\frac{\tau}{1+\tau \mu_x} \Xi_{\tau, \sigma, \theta}^x \ylfourth{\nu_x}^2+\frac{\sigma}{1+\sigma \mu_y} \Xi_{\tau, \sigma, \theta}^y \ylfourth{\nu_y}^2\right),
    \end{equation}
    {where $\startingbias \defineq \frac{1}{2\tau} \|x_{0} - \optx \|^2 + \frac{1}{2\sigma} \|y_n - \opty\|^2$ denotes the initial bias, and $\Xi_{\tau, \sigma, \theta}^x \defineq 1 + \frac{\sigma \theta(1+\theta) L_{y x}}{2\left(1+\sigma \mu_y\right)}$ $\Xi_{\tau, \sigma, \theta}^y \triangleq \frac{\tau \theta(1+\theta) L_{y x}}{2\left(1+\tau \mu_x\right)}+\left(1+2 \theta+\frac{\theta+\sigma \theta(1+\theta) L_{y y}}{1+\sigma \mu_y}+\frac{\tau \sigma \theta(1+\theta) L_{y x} L_{x y}}{\left(1+\tau \mu_x\right)\left(1+\sigma \mu_y\right)}\right)(1+2 \theta)$ are noise related constants that depend on the problem and algorithm parameters}.
\end{thm}
\yassine{
As stated above, the convergence of SAPD in expected squared distance presents the classical bias-variance trade-off, \saa{which can be controlled through adjusting the \sapdname~parameter choice}. 
The bias term $\rho^n \cD_{\tau, \sigma}$, captures the rate at which the error due to initialization (bias) decays, ignoring the noise. 
It is shown in~\citep{zhang2021robust} that for certain choice of parameters, \saa{convergence of initialization bias to $0$} occurs at an accelerated rate $\rho= 1-\Theta(\frac{1}{\kappa})$ instead of the non-accelerated rate $\rho= 1-\Theta(\frac{1}{\kappa^2})$ of methods such as (Jacobi-style) SGDA. 
The variance term constitutes the (remaining part) second term at the right-hand side of~\eqref{rso:eq:conv_sapd_equation} and is due to noise accumulation \saa{that scales} with the stepsize and the noise variance.
For a particular choice of SAPD parameters, it is shown that SAPD exhibits an optimal complexity in expectation, up to logarithmic factors, and achieves an \saa{accelerated} decay rate for the bias term; however, in a number of risk-sensitive situations, convergence in expectation can prove to be insufficient. In this paper, we further investigate the properties of \sapdname~for several measures of risks, that we detail in Section~\ref{rso:sec:risk_measures}.\looseness=-1
}

\subsection{\texorpdfstring{\mg{Assumptions on the gradient noise}}{Assumptions on the gradient noise}}

\sa{Although} according to Theorem~\ref{rso:thm:sapd_expectation_result}, \eqref{rso:eq:matrix_inequality_params} describes a general set of parameters for which \sapdname~will admit guarantees in terms of the \textit{expected} weighted distance squared to the solution, risk-sensitive guarantees \sa{for SAPD}, including high-probability \sa{bounds} are not known. 
In the forthcoming \sa{sections,} we study \sapdname~for parameters satisfying~\eqref{rso:eq:matrix_inequality_params}, and we obtain convergence \sa{guarantees} in high probability, in $\cvar$, in $\evar$, and \sa{also} in the $\chi^2$-divergence-based risk \sa{measures}, \saa{which are} properly defined in Section \ref{rso:sec:risk_measures}. 
\saa{In other words,} our focus here is to obtain high probability guarantees as well as bounds on \sa{the risk associated with $\mg{E_n^{1/2}}=\|z_n - \optz\|$.}
To this end, we will \sa{make a ``light-tail" assumption on the magnitude of the gradient noise \saa{adopting a} subGaussian structure. 
Before giving our \saa{assumption on the gradient noise} precisely, we start with introducing the family of norm-subGaussian random variables, and recall their basic properties.}
\begin{definition}
    A random vector $X \from \Omega \to \Rd$ is \textit{norm-subGaussian} with proxy \saa{$\sigma>0$}, denoted by $X \in \normsubgaussian{\sigma}$, if we have
   $ \probability\Big[\norm{X - \expectation[X]}~\geq t\Big] \leq 2 e^{\frac{-t^2}{2\sigma^2}}, \quad \forall t \in \R$.
\end{definition}
Random vectors with Norm-subGaussian distribution were introduced in~\citep{jin2019short}, and encompass a large class of random vectors including subGaussian random vectors. 
First, note that given \sa{an} arbitrary $\alpha>0$ and a random variable $X \from \Omega \to \Rd$ such that $X \in \normsubgaussian{\sigma}$ for some $\sigma>0$, we immediately have the following implication:
\begin{equation}\label{rso:eq:nSG_pos_homogeneity}
    X \in \normsubgaussian{\sigma} \implies \alpha X \in \normsubgaussian{\alpha \sigma}.
\end{equation}
For instance, $X:\Omega\to\reals^d$ is norm-subGaussian when $X$ is subGaussian, or it is bounded, i.e., $\exists B>0$ such that $\norm{X}\leq B$ with probability $1$. 
As discussed in~\citep[Lemma 3]{jin2019short}, the squared norm of a norm-subGaussian vector admits a sub-Exponential distribution, \saa{which is defined next.}
\begin{definition}
    A random variable $\mathcal{U} \from \Omega \to \R$ is subExponential with proxy $K > 0$ if it satisfies
    \[
        \expectation\left[e^{\lambda |\mathcal{U}|}\right] \leq e^{\lambda K},  \quad \forall \lambda \in [0, K^{-1}].
    \]
\end{definition}
\yassine{In particular, if we take $U = \squarednormtwo{X}$ \sa{for} $X \in \normsubgaussian{\sigma}$ \saa{with $\expectation[X] = 0$}, \sa{then}, {the following result shows that} $U$ is subExponential with proxy $K = 8\sigma^2$.}
\begin{lem}\label{rso:lem:nSG_mgf_bounds}
    Let $X\in \normsubgaussian{\sigma}$ be such that $\expectation[X] = 0$. Then, for any $ \lambda \in \left[0, \sa{\frac{1}{4\sigma^2}}\right]$, 
    \begin{equation}\label{rso:eq:nSG_mgf_bound}
        \mgf{\lambda \squarednormtwo{X}} \leq \sa{2 e^{4\lambda\sigma^2}-1}\leq e^{8 \lambda \sigma^2}.
    \end{equation}
\end{lem}
\begin{lem}\label{rso:lem:norm_sg_dot}
    Let $X \in \normsubgaussian{\sigma}$ such that $\expectation[X]=0$. 
    Then, for any $u\in \Rd$ and $\lambda\geq 0$, it holds that
    %
     $  \mgf{\lambda \dotproduct{u}{X}} \leq e^{8\lambda^2 \squarednormtwo{u} \sigma^2}$.
    %
\end{lem}
\sa{For completeness, the proofs of these two elementary results are provided in~Section~\ref{rso:app:prelimaries} of the appendix.}
\mg{Next, we will introduce an assumption which says that gradient noise terms $\noisegradx{k}$ and $\noisegrady{k} $ are light-tailed admitting a norm-subGaussian structure when conditioned on the natural filtration of the past iterates.} 
\begin{assumption}\label{rso:assumption:stochastic}
    For any $k\geq 0$ the random variables $\noisegradx{k} $ and $\noisegrady{k}$ are conditionally unbiased and norm-subGaussians with respective proxy parameters $\proxyx, \proxyy > 0$. 
    More precisely, for all $t\geq 0$, we have $\expectation\left[\noisegrady{k} | \sigmafield_{k-1}^x \right] = 0, \expectation\left[ \noisegradx{k} |\sigmafield_{k}^y \right] = 0,$ and
    \[
        \probability[\|\noisegrady{k}\| \geq t | \saa{\cF_{k-1}^{x}}] \leq 2 e^{\frac{-t^2}{2\proxyy^2}},  \quad
        \probability[\|\noisegradx{k}\| \geq t | \saa{\cF_{k}^{y}}] \leq 2 e^{\frac{-t^2}{2\proxyx^2}}. 
    \]
\end{assumption}
We note that such subGaussian noise assumptions are common in the study of stochastic optimization algorithms~\citep{rakhlin2012making,ghadimi2012optimal,harvey2019tight}. 
In machine learning applications, where stochastic gradients are often estimated on sampled batches, noisy estimates typically behave Gaussian for moderately high sample sizes, as a consequence of the central limit theorem \citep{panigrahi2019non}. 
Furthermore, there are applications in data privacy where i.i.d. subGaussian noise is added to the gradients for privacy reasons~\citep{levy2021learning,varshney2022nearly}. 
In such settings, we expect Assumption \ref{rso:assumption:stochastic} to hold naturally.
\ylfourth{In the rest of the paper, together with Assumption \ref{rso:assumption:str_cvx_smooth}, we will assume that Assumption \ref{rso:assumption:stochastic} holds in lieu of Assumption \ref{rso:assumption:unbiased}.}

\subsection{\texorpdfstring{\mg{$\mbox{VaR}$, $\mbox{CVaR}$, $\mbox{EVaR}$ and $\chi^2$-divergences}}{VaR, CVaR, EVaR, and X2-divergence}}\label{rso:sec:risk_measures}

\saa{For any given $n\geq 0$, to quantify the risk associated with $\saa{\|z_n -\optz\|}$, i.e., the distance to the unique saddle point,} we will resort to $\phi$-divergence-based risk measures borrowed from the risk measure theory~\citep{ben2007old}, including $\cvar$, $\evar$ and $\chi^2$-divergence. 
The first risk measure of interest is \sa{the} quantile function, \saa{also known as value at risk,} defined for any random variable $\sa{\cU} \from \Omega \to \R$ as 
\[
    Q_p(\mathcal{U}) \saa{\triangleq \var_p(\cU)} \triangleq \inf_{t \in \R} \probability[\mathcal{U} \leq t] \geq p.
\]
Quantile upper bounds correspond to high-probability results, which have been already fairly studied to assess the robustness of stochastic algorithms~\citep{ghadimi2012optimal,rakhlin2012making,harvey2019tight}. 
One key contribution of this paper is the derivation of an upper bound on the quantiles of the \saa{weighted} distance metric $\cD_n$, defined in~\eqref{def-metric}, \sa{such that this upper bound exhibits} a tight bias-variance trade-off --see Section~\ref{rso:sec:tightness}.

Furthermore, we investigate the robustness of \sapdname~with respect to three convex risk measures based on $\varphi$-divergences~\citep{ben2007old}. 
Generally speaking, for a given proper convex function $\varphi: \R_+ \to \R$ satisfying $\varphi(1) = 0$ and $\lim_{t \to 0^+} \varphi(t) = \varphi(0)$, the associated $\varphi$-divergence, is defined as 
%
    $D_{\varphi}(\mathbb{Q}\Vert\mathbb{P}) \triangleq \int_{\Omega} \varphi\left(\frac{\measuredWRT \mathbb{Q}}{\measuredWRT \mathbb{P}}\right) \measuredWRT \mathbb{P}$,
%
for any input probability measures $\mathbb{Q}, \mathbb{P}$ such that $\mathbb{Q} \ll \mathbb{P}$, i.e., $\mathbb{Q}$ is absolutely continuous with respect to $\mathbb{P}$. 
Different choices of $\varphi$-divergence result in different risk measures \saa{as discussed next.}
\begin{definition}\label{rso:def:variational_representation_rm}
    For any $r \geq 0$, 
    the $\varphi$-divergence based risk measure at level $r$ \saa{is defined} as 
    \begin{equation}\label{rso:eq:phi_div_risk_measure}
        \riskMeasure_{\varphi, r}(\mathcal{U}) \triangleq \sup_{\substack{\mathbb{Q} \ll \mathbb{P}, ~ D_{\varphi}(\mathbb{Q}\Vert\mathbb{P}) \leq r}} \expectation_{\mathbb{Q}}\left[\mathcal{U}\right],
    \end{equation}
    where $\mathbb{P}$ denotes an arbitrary reference probability measure. 
\end{definition}
\mg{We refer the reader to \citep{ben2007old,shapiro2017distributionally} for more on $\varphi$-divergence based risk measures.} 
In this paper, we investigate the performances of SAPD under three $\varphi$-divergence based risk measures, summarized in \sa{Table~\ref{rso:table:risk_measures}}.

\begin{table}[t] 
{\scriptsize
    \centering 
    \renewcommand{\arraystretch}{2.}
    \begin{tabu} to \textwidth {X[c] X[c] X[c]}
        \hline 
        Risk measure & Formulation & Divergence \\ 
        \hline 
        $\cvar_p$, $p\in [0,1)$ & $\frac{1}{1-p} \int_{p'=p}^1 Q_{p'}(\mathcal{U}) \mathrm{d}p'$ & $\varphi(t) = \cvxindicator_{[0, \frac{1}{1-p}]}(t)$ \\ 
        $\evar_p$, $p\in [0,1)$ & $\inf_{\eta > 0 } \left\{\frac{-\log(1-p)}{\eta} +~\frac{1}{\eta}  \log(\expectation(e^{\eta \mathcal{U}}))\right\}$ & $\varphi(t) = t \log t - t + 1$ \\ 
        $\riskMeasure_{\chi^2, r}$, $r \geq 0$ & $\inf _{\eta \geq 0}\left\{\sqrt{1+2 r} \sqrt{\expectation(\mathcal{U}-\eta)_{+}^2}+\eta\right\}$ & $\varphi(t) = \frac{1}{2}(t-1)^2$ \\ 
        \hline 
    \end{tabu}}
    \renewcommand{\arraystretch}{1}
    \caption{Three examples of $\varphi$-divergence based risk measures studied in this paper.} 
    \label{rso:table:risk_measures} 
\end{table}%
First, given $p\in[0,1)$, we consider the conditional value at risk \saa{at level $p$, i.e., $\cvar_p$,} defined as 
\begin{equation}\label{rso:eq:def_cvar}
    \cvar_p(\mathcal{U}) \defineq \frac{1}{1-p} \int_{p'=p}^1 Q_{p'}(\mathcal{U})\;\mathrm{d}p'\sa{.}
\end{equation}
The \cvarname~\saa{measure} admits the variational representation~\eqref{rso:eq:phi_div_risk_measure} with $\varphi : t \mapsto \cvxindicator_{[0, (1-p)^{-1}]}(t)$ for any $r> 0$. 
As an average of the higher quantiles of $\mathcal{U}$, $\cvar_p(\mathcal{U})$ holds intuitively as a statistical summary of the tail of $\mathcal{U}$, beyond its $p$-quantile. 
While high-probability bounds do not take into account the \textit{price of failure} tied to the event $\mathcal{U} \geq Q_p(\mathcal{U})$, the \cvarname~presents the advantage of averaging the whole tail of the distribution; therefore, it can quantify the risk associated \saa{with} tail events in a robust fashion. 

\medskip
The second convex risk measure we investigate is the Entropic Value at Risk~\citep{ahmadi2012entropic}, denoted $\evar$, and is defined as 
{\small
  $   \evar_{p}(\mathcal{U}) \triangleq \inf_{\eta > 0 } \left\{- \eta^{-1} \log(1-p) + \eta^{-1}  \log(\expectation(\exp(\eta \mathcal{U})))\right\}$.
}
The $\evar$~admits the variational representation~\eqref{rso:eq:phi_div_risk_measure} with $\varphi : t \mapsto t\log(t)-t+1$ and the parameter $r$ is  set to $- \log(1-p)$ for given $p\in[0,1)$ --see e.g. \citep{shapiro2017distributionally}. $\evar$ exhibits a higher tail-sensitivity than $\cvar$, in the sense that $\cvar_{p}(\mathcal{U}) \leq \evar_p(\mathcal{U})$ for all $p \in [0,1)$ whenever $\evar_{p}(\mathcal{U}) <\infty$. 
Finally we will also derive results in terms of the $\chi^2$-divergence based risk measure, defined as~\eqref{rso:eq:phi_div_risk_measure} with $\varphi : t \mapsto \frac{1}{2} (t-1)^2$.

%% file: 3_main_results.tex
{In this section, we present the main results of this paper, which consists of convergence analysis of \sapdname~in high-probability and \sa{provide guarantees in terms of} the three convex risk measures presented in Table~\ref{rso:table:risk_measures}.
\saa{Later in Section~\ref{rso:sec:quadratics}, we derive analytical expressions related to convergence behavior of \sapdname~applied on quadratic SP problems, and in Section~\ref{rso:sec:tightness} we discuss some tight characteristics of our main results provided in this section. 
Finally, in Section \ref{rso:sec:proof_section}, we 
provide the proofs of our main result stated in Theorem~\ref{rso:thm:high_probability_bound_sapd}.} 
%
\begin{thm}\label{rso:thm:high_probability_bound_sapd}
    \sa{{Suppose Assumption \ref{rso:assumption:str_cvx_smooth} and Assumption \ref{rso:assumption:stochastic} hold.} 
    Given $\tau, \sigma > 0, $ and \sa{$\theta\geq 0$} satisfying~\eqref{rso:eq:matrix_inequality_params} for some $\rho\in (0,1)$ and $\alpha \in [0, \sigma^{-1})$, let $(x_n,y_n)_{n\geq 1}$ denote the corresponding \sapdname~iterates,} \yassine{initialized at an arbitrary tuple $(x_0, y_0) \in \setX \times \setY$. 
    For all $n \in \N$, $p \in \sa{[0},1)$, \sa{it holds that}}
    \begin{equation}\label{rso:eq:high_probability_bound_sapd} 
        \begin{split}
            \mathbb{P}\bigg[{\cD_{n+1} + \cD_n} \leq \mg{q_{p,n+1}}
                        \bigg]\geq p, 
        \end{split}
        \quad \hbox{where}
    \end{equation}
    %
    %
    \begin{equation} q_{p,n+1}\triangleq 
            {\left(\frac{1+\rho}{2}\right)^n 
                    \left( \cC_{\tau, \sigma, \theta}\; \startingbias + \Xi^{(1)}_{\tau,\sigma,\theta} \right) 
                    + \Xi^{(2)}_{\tau,\sigma,\theta} 
                    + \Xi^{(3)}_{\tau,\sigma,\theta}\log\left(\frac{1}{1-p}\right)}, 
    \end{equation}
    where \sa{$\cD_n  \mg{=} \frac{1}{2\tau} \squarednormtwo{x_{n} - \optx} + \frac{1-\alpha \sigma}{2 \sigma} \squarednormtwo{y_n - \opty}$, $\startingbias \defineq \frac{1}{2\tau} \squarednormtwo{x_0 - \optx} + \frac{1}{2\sigma} \squarednormtwo{y_0 - \opty}$, $\cC_{\tau, \sigma, \theta}$ and $\Xi_{\tau, \sigma, \theta}^{(i)} \defineq \Xi_{\tau, \sigma, \theta}^{(i, x)} \proxyx^2 + \Xi_{\tau, \sigma, \theta}^{(i, y)} \proxyy^2$ for $i = 1,2,3$
    depend only on {the algorithm parameters ($\tau,\sigma,\theta)$ and the problem parameters $(\mu_x, \mu_y, \Lxx, \Lyy, \Lxy, \Lyx)$}. 
    {Furthermore, all these constants can be made explicit}\footnote{\sa{
    These constants are explicitly given within the proof, provided in Section~\ref{rso:sec:proof_section}.}}} and {in particular,} under the CP parameterization in~\eqref{rso:eq:cb_params}, they satisfy $\cC_{\tau,\sigma,\theta}=\Theta(1)$, $\Xi^{(i, x)}_{\tau,\sigma,\theta} = \Theta(1)$, and $\Xi^{(i, y)}_{\tau,\sigma,\theta} = \Theta(1)$ for all $i$, as $\theta \to 1$, \sa{which implies that 
    {\small
    \[
        \limsup_{n \to \infty} Q_p(\norm{z_n-\optz}^2)
            = \cO\Big(
                        (1-\theta)\delta^2\Big(1+\log\Big(\frac{1}{1-p}\Big)\Big)
                \Big).
    \]}}%
    %
\end{thm}
\begin{proof}
    \sa{The proof is provided in Section~\ref{sec:proof-high-probability}.}
\end{proof}
%
%
\begin{rema}
    Under the premise of Theorem~\ref{rso:thm:high_probability_bound_sapd},  \eqref{rso:eq:high_probability_bound_sapd} implies that for all $p\in[0,1)$ and $n\geq 0$,
    {\small
    \begin{equation}\label{eq-quantile-bound} 
        Q_p\left(\cD_{n+1}^{1/2}\right) \leq Q_p\left((\cD_{n+1}+(1-\rho)\cD_n)^{1/2}\right) = \saa{\var_{p}\left((\cD_{n+1}+(1-\rho)\cD_n)^{1/2}\right)} \leq q_{p,n+1}^{1/2}.
    \end{equation}}%
\end{rema}
\begin{rema} 
    {
    For any given $\rho\in(0,1)$, to check if there exists \sapdname~parameters $\tau,\sigma,\theta$ such that the bias component of $\mathbb{E}[\cD_{n+1}+\cD_n]$ decreases to $0$ linearly with a rate coefficient bounded above by $\rho\in(0,1)$, one needs to solve a \textit{5-dimensional} SDP, i.e., after fixing $\rho$, checking the feasibility of \eqref{rso:eq:matrix_inequality_params} reduces to an SDP problem, see~\citep{zhang2021robust} for details.}
    Below in Corollary~\ref{rso:thm:complexity_sapd_result}, we provide a particular solution to \eqref{rso:eq:matrix_inequality_params}, in the form of \eqref{rso:eq:cb_params}, for which the choice of $\rho$ leads to an accelerated 
    {behavior with a complexity of $ \mathcal{O}(\kappa \log(\varepsilon^{-1}) + \mu^{-1}\delta^2  (1+\log((1-p)^{-1}))  \varepsilon^{-1} \log(\varepsilon^{-1}))$ where  $\delta^2 = \max(\delta_x^2, \delta_y^2)$. 
    Thus, the bias term in Theorem \ref{rso:thm:high_probability_bound_sapd} decays at an \textit{accelerated} rate, which differs from the standard 
    decay of non-accelerated Jacobi-style algorithms where the initialization (bias) error scales proportionally to $\kappa^2$~\citep{fallah2020optimal}}. 
\end{rema} 
\begin{rema} 
    \sa{Our bound for the $p$-th quantile, $q_{p,{n+1}}$, is tight, i.e., under the parameter choice in \eqref{rso:eq:cb_params}, the dependency of $q_{p,{n+1}}$ to $\theta$ and $p$ cannot be improved {when $n$ is large enough}. 
    For details, please see~Theorem~\ref{rso:prop:tightness}.} 
\end{rema}
Next, we provide the oracle complexity of \sapdname~in high probability, which can be derived as a corollary to our main Theorem~\ref{rso:thm:high_probability_bound_sapd}.     
\addtocounter{theorem}{+4}
\begin{corollary}\label{rso:thm:complexity_sapd_result}
    {For $p \in [0,1)$ and $\varepsilon > 0$, set $\tau, \sigma, \theta, \rho$ as}
    {\small
    \begin{equation}\label{rso:eq:cb_param_complexity} 
        \tau = \frac{1-\theta}{\theta \mux}, 
            \quad 
        \sigma = \frac{1-\theta}{\theta \muy}, 
            \quad 
       \rho= \theta = \max \left(1/2, \bar \theta_1, \bar \theta_2 , \bar{\bar{\theta}}_x, \bar{\bar{\theta}}_y  \right),\quad \hbox{where}
    \end{equation}
    {\scriptsize
    {
    \begin{equation*}
        {\bar \theta}_1 
            \defineq 1- \frac{\beta\left(\Lxx+\mu_x\right) \mu_y}{4 \Lyx^2} 
                \left(\sqrt{1+\frac{8 \mu_x \Lyx^2}{\beta \mu_y\left(\Lxx+\mu_x\right)^2}}-1\right), \quad 
            {\bar \theta}_2 
                \defineq
                    \left\{
        \begin{array}{ll}
            1-\frac{(1-\beta)^2}{32} \frac{\mu_y^2}{\Lyy^2}
                \left(\sqrt{1+\frac{64 \Lyy^2}{(1-\beta)^2 \mu_y^2}}-1\right) & \mbox{if } \Lyy > 0 \\
            0 & \mbox{if } \Lyy = 0
        \end{array} \right. ,
    \end{equation*}
    }
    }}
    {\small
    \begin{eqnarray*} 
        \bar{\bar{\theta}}_x 
            &\defineq& 1 - \frac{\mux \varepsilon}{\left(c_x^{(1)} + c_x^{(2)} \frac{\Lxy}{\Lyx} + c_x^{(3)} \frac{\Lxy^2}{\Lyx^2}\right) \proxyx^2 \left(1 + \log(1/(1-p))\right)}, \\
            \bar{\bar{\theta}}_y 
            &\defineq& 1 - \frac{\muy \varepsilon}{\left(c_y^{(1)} + c_y^{(2)} \frac{\Lxy}{\Lyx} + c_y^{(3)} \frac{\Lxy^2}{\Lyx^2}\right) \proxyy^2 \left(1 + \log(1/(1-p))\right)},
    \end{eqnarray*}}%
    with $\beta = \ylsecond{\min\{1/2,\mux/\muy, \muy/\mux\}}$, and for universal positive constants $c_x^{(i)}, c_y^{(i)}$ for $i=1,2,3$ that are large enough\footnote{For simplicity of the presentation, we do not provide these universal constants explicitly here; that said, the constants can be made explicit in a straightforward manner following {the step-by-step computations} in Section~\ref{sec-complexity-proof} of the Appendix.}.
    Then, \eqref{rso:eq:matrix_inequality_params} is satisfied for $\alpha=\frac{1}{2\sigma}-\sqrt{\theta}L_{yy}$ and \sapdname~guarantees that $\probability\left[\mux \squarednormtwo{x_{n} - \optx} +  \muy \squarednormtwo{y_{n} - \opty} \leq \varepsilon \right] \geq p$ for all $n$ satisfying \eqref{intro-complexity}.
\end{corollary}
\begin{proof}
    {The result follows directly from Theorem~\ref{rso:thm:high_probability_bound_sapd} after plugging in our choice of parameters based on tedious but straightforward computations. 
    For the sake of completeness, we provide the details of these computations in Section~\ref{sec-complexity-proof} of the Appendix}\looseness=-1
\end{proof}
%
%
{
\addtocounter{thm}{+1}
\begin{rema}\label{rema-Lxy-Lyx-equal}
    By Assumption \ref{rso:assumption:str_cvx_smooth}, the partial gradients $\nabla_x \Phi$ and $\nabla_y \Phi$ are Lipschitz continuous; therefore, they are almost everywhere differentiable by Rademacher's Theorem. If we assume slightly more, i.e., if 
    $\nabla_x \Phi$ and $\nabla_y \Phi$ are continuously {differentiable}, then the partial derivatives commute and we have $\nabla_y\nabla_x \Phi (x,y) = \nabla_x\nabla_y \Phi (x,y)$, {as a consequence of Schwarz's Theorem}. 
    In this case, we can take $\Lxy=\Lyx$ in Assumption \ref{rso:assumption:str_cvx_smooth} and in Theorem \ref{rso:thm:complexity_sapd_result}.
\end{rema}
}
\sa{Using Theorem~\ref{rso:thm:high_probability_bound_sapd} and} \mg{building on the representation of the CVaR in terms of the quantiles}, we can deduce \sa{a bound on $\cvar_p(\cD_n^{\frac{1}{2}})$ as shown in Theorem~\ref{thm:risk-nounds}, where we also provide bounds on the entropic value at risk and on the $\chi^2$-based risk measure, as defined in Table~\ref{rso:table:risk_measures}.}
\addtocounter{theorem}{-4}
\begin{thm}[Bounds on Risk Measures]
    \label{thm:risk-nounds}
    \sa{Under the premise of Theorem~\ref{rso:thm:high_probability_bound_sapd},} 
    {\scriptsize
    {\begin{equation}\label{rso:eq:cvar_bound}
        \cvar_p \left(\cD_{n+1}^{\frac{1}{2}}\right) 
            \leq 
                    \sqrt{
                        \left(\frac{1+\rho}{2}\right)^{n/2}
                        \left(
                            \cC_{\tau, \sigma, \theta} \startingbias 
                            + \Xi^{(1)}_{\tau,\sigma,\theta} 
                        \right)
                        +\Xi^{(2)}_{\tau,\sigma,\theta}
                        }
                    + \sqrt{
                        \Xi^{(3)}_{\tau,\sigma,\theta} \Big(1 + \log\Big(\frac{1}{1-p}\Big)\Big)
                    },
    \end{equation}}}%
    {\scriptsize
        \begin{equation}\label{rso:eq:evar_result}
            \evar_p(\cD_{n+1}^{\frac{1}{2}}) 
                \leq 
                    \sqrt{
                        \left(\frac{1+\rho}{2}\right)^{n/2}
                        \left(
                            \cC_{\tau, \sigma, \theta} \startingbias 
                            + \Xi^{(1)}_{\tau,\sigma,\theta} 
                        \right)
                        +\Xi^{(2)}_{\tau,\sigma,\theta}
                        }
                    + \sqrt{
                        \Xi^{(3)}_{\tau,\sigma,\theta}
                        } \left(\log\Big(\frac{1}{1-p}\Big)^{1/2} + \sqrt{\pi}\right),
        \end{equation}}
    \saa{hold for all $n \in \N$ and $p \in [0,1)$,} where \sa{$\cD_n$, $\Xi^{(1)}_{\tau,\sigma,\theta}$, $\Xi^{(2)}_{\tau,\sigma,\theta}$ and $\bar\delta$} are as defined in Theorem~\ref{rso:thm:high_probability_bound_sapd}. 
    \sa{Furthermore, for all $n \in \N$ and $r>0$, the right-hand side of \eqref{rso:eq:evar_result} with $p=1-\frac{1}{1+r}$ is an upper bound on $\riskMeasure_{\chi^2, r}(\cD_{n+1}^{1/2})$.}
\end{thm}
\begin{proof}
    \saa{The proof is provided in Section~\ref{sec:risk-proof}.}
\end{proof}
In the next section, we discuss a \sa{{family of} quadratic SP problems for which we can compute the limiting covariance matrix of the iterates in expectation explicitly, assuming additive i.i.d. Gaussian noise on the partial gradients. 
This will allow us to gain insights about the effect of parameter choices and argue about the tightness of our analysis.} 

%% file: 4_quadratics.tex
In this section, we study the behaviour of \sapdname~on quadratic problems subject to isotropic Gaussian noise. 
More specifically, we consider
\begin{equation}\label{rso:eq:quadratic_problem}
    \min_{x \in \Rd} \max_{y \in \Rd} \frac{\mux}{2} \squarednormtwo{x} + \dotproduct{Kx}{y} - \frac{\muy}{2} \squarednormtwo{y},
\end{equation}
where $K\in \R^{d\times d}$ is a symmetric matrix, and $\mux, \muy >0$ are two regularization parameters. 
\saa{The unique saddle point of \eqref{rso:eq:quadratic_problem} is the origin $\optz=(\optx,\opty)=(\mathbf{0},\mathbf{0})$.}
\sa{At each iteration $k\geq 0$, suppose we have access to noisy estimates $\approxgradphiy(x_k,y_k)= Kx_k + \gnoisey{k}$ and $\approxgradphix(x_k,y_{k+1}) = K^\top y_{k+1} + \gnoisex{k}$} of the partial gradients of $\sadobj: (x,y) \mapsto \dotproduct{Kx}{y}$, where the $(\gnoisex{k})_{k\geq 0}$ and $(\gnoisey{k})_{k\geq 0}$ denote i.i.d. centered Gaussian vectors satisfying $\expectation[\gnoisex{k} {\gnoisex{k}}^\top] = \expectation[\gnoisey{k} {\gnoisey{k}}^\top] = d^{-1}\delta^2 I$ \sa{for some $\delta\geq 0$}. 
{In this special case, we have the gradient noise vectors $\noisegradx{n} = \gnoisex{k}$, $\noisegrady{n} = \gnoisey{k}$.} 
\saa{Our main motivation to study this toy problem is to gain some insights into the sample paths \sapdname~generates, and use these insights while studying the tightness of high-probability bounds provided in Section~\ref{rso:sec:high_prob_bound_sapd}.}\looseness=-1

This problem was first studied in~\citep{zhang2021robust} where it was shown that, under certain conditions on $\tau, \sigma, \theta$, the sequence of iterates \sa{$(\tilde z_n)_{n\geq 0}$} generated by \sapdname, \sa{where $\tilde z_n=(x_{n-1}, y_n)$,} converges in distribution to a \sa{zero-mean multi-variate Gaussian random vector} whose covariance matrix \sa{$\widetilde{\Sigma}^\infty$} satisfies a certain Lyapunov equation of dimension $2d\times 2d$. 
The authors of~\citep{zhang2021robust} manage to split this equation into $d$ \sa{many $2\times 2$} Lyapunov equations: 
\begin{equation}\label{rso:eq:reduced_lyap_main_paper}
    \sa{\widetilde{\Sigma}^{\infty, \sa{\lambda}}} = A^{\lambda}~\sa{\widetilde{\Sigma}^{\infty, \sa{\lambda}}}~ (A^{\lambda})^\top + R^{\lambda}\sa{\in\reals^{\ylfourth{2\times 2}},\quad \saa{\forall}\lambda\in\spectrum(K),}
\end{equation}
\sa{i.e., for each eigenvalue $\lambda$ of $K$, there is a $2\times 2$ Lyapunov equation to be solved, where $A^{\lambda}, R^{\lambda}\in\reals^{2\times 2}$ \ylfourth{depend} only} on $\tau, \sigma, \theta$, $\delta^2$ and  \sa{$\lambda\in\spectrum(K)$ --for completeness, we 
provide these steps in detail} in Section~\ref{rso:sec:recalls_quadratics} \mg{of the appendix}.

\sa{Given an arbitrary symmetric matrix $K$, in~\citep{zhang2021robust}, the small-dimensional Lyapunov equation in~\eqref{rso:eq:reduced_lyap_main_paper} is solved numerically. 
On the other hand,} \sa{to establish the tightness of our high-probability bounds for the class of SCSC problems with a non-bilinear $\Phi$} subject to \sa{noisy gradients with subGaussian tails} (see Section~\ref{rso:sec:high_prob_bound_sapd}), we need to analytically solve~\eqref{rso:eq:reduced_lyap_main_paper}. 
However, analytically solving~\eqref{rso:eq:reduced_lyap_main_paper} for general parameters satisfying the matrix inequality \eqref{rso:eq:matrix_inequality_params} is a challenging problem that standard symbolic computation tools were not in a position \sa{to properly address our needs. 
That said, \mg{as we shall discuss next}, we can provide analytical solutions for \eqref{rso:eq:reduced_lyap_main_paper}} under the Chambolle-Pock \ylfourth{(CP)} parameterization~\sa{in}~\eqref{rso:eq:cb_params}, where primal and dual stepsizes are parameterized in $\theta$, i.e., the momentum parameter, and for this choice the rate $\rho=\theta$. {
\saa{We should note that} CP parameterization represents \saa{a rich enough class of admissible SAPD parameters} in the sense that under this parameterization SAPD can achieve accelerated bias decay in the expected squared distance metric \citep{zhang2021robust} and the accelerated high probability results we derived in Theorem \ref{rso:thm:complexity_sapd_result}.}

\subsection{Covariance matrix of the iterates under the CP parameterization}

The main result of this section is \sa{the analytical solution of a Lyapunov equation corresponding to the limiting covariance $\Sigma^\infty$ for $(x_n,y_n)_{n\geq 0}$ as $n\to \infty$, which is similar to the Lyapunov equation in~\eqref{rso:eq:reduced_lyap_main_paper} corresponding to the limiting covariance $\widetilde{\Sigma}^\infty$ for $(x_{n-1},y_n)_{n\geq 0}$,} \sa{for the parameterization in}~\eqref{rso:eq:cb_params}. 
{The proof yields closed-form solutions, useful for understanding the effect of parameters on the solution. Due to lengthy calculations involved, the proof is provided in~Section~\ref{rso:sec:proof_thm_quadratics} of the Appendix.}
\mg{Our proof technique is based on identifying the conditions on the parameters so that the Lyapunov matrix {in \eqref{rso:eq:reduced_lyap_main_paper}} admits \ylsecond{a} unique solution and we solve it as a function of $\lambda$ by diagonalizing $\Alam$ given in \eqref{rso:eq:reduced_lyap_main_paper} with a proper change of basis.}   
\begin{thm}\label{rso:thm:quadratics}
    Let $\kappa_{\text{max}} \defineq \rho(\yassine{K}) / \sqrt{\mux \muy}$. 
    For \yassine{any} \mg{given} $\theta > (\sqrt{1+\kappa_{\text{max}}^2}- 1)/\kappa_{\text{max}}$ \mg{fixed}, set $\tau = (1-\theta)/(\theta \mux)$ and $\sigma = (1-\theta)/(\theta \muy)$. Suppose gradient noise sequences are i.i.d. centered Gaussian satisfying $\expectation[\gnoisex{k}{\gnoisex{k}}^\top] = \expectation[\gnoisey{k}{\gnoisey{k}}^\top] = \frac{\delta^2}{d} I_d$.
    Then, the iterates $z_n = (x_n, y_n)$ generated by \sapdname~\sa{applied to the SCSC problem in}~\eqref{rso:eq:quadratic_problem} with parameters $\tau, \sigma, \theta$, converges in distribution to a centered Gaussian distribution with covariance matrix \sa{$\Sigma^\infty$} satisfying $\sa{\Sigma^\infty} = V^\top \sa{\Sigma^{\infty, \Lambda}} V$ where $V$ is orthogonal, and $\sa{\Sigma^{\infty, \Lambda}}$ is block diagonal with $d$ blocs $\Sigma^{\infty, \lambda_i}\sa{\in\reals^{2\times 2}}$ \sa{for $i=1,\ldots,d$, where $\yassine{(\lambda_i)_{1\leq i\leq d}}$ denote the eigenvalues of $K$.} 
    Specifically, for each $\lambda \in \spectrum(K)$, block $\Sigma^{\infty, \lambda}$ \sa{has the following form: For $\lambda = 0$,}
    \begin{equation}\label{eq:limit_cov_informal_0}
        \sa{\Sigma^{\infty, 0}} = 
            \frac{\delta^2}{d} \frac{(1-\theta)}{\mu_x^2 \mu_y^2(1+\theta)}
                \left[\begin{array}{cc}
                    \theta^2 \mu_y^2 & 0 \\
                    0 & \mu_x^2\left(1+2\left(1-\theta^2\right) \theta\right)
                \end{array}\right],
    \end{equation}    
    \sa{otherwise, for $\lambda\neq 0$,}
    \begin{equation}\label{eq:limit_cov_informal_non_zero}
        \begin{scriptsize}
           \sa{\Sigma^{\infty, \lambda}} = \frac{\delta^2}{d} \frac{1-\theta}{P_c(\theta, \kappa)}
                \left[\begin{array}{cc}
                    \frac{1}{\mux^2}\left(\sa{P_{1,1}^{(\infty, 1)}(\theta, \kappa)} + \frac{\lambda^2}{\mu_y^2} \sa{P_{1,1}^{(\infty, 2)}(\theta, \kappa)} \right)
                    & \frac{1}{\lambda\mu_x} \left(\sa{P_{1,2}^{(\infty, 1)}(\theta, \kappa)} + \frac{\lambda^2}{\mu_y^2} \sa{P_{1,2}^{(\infty, 2)}(\theta, \kappa)}\right) \\
                    \frac{1}{\lambda\mu_x} \left(\sa{P_{1,2}^{(\infty, 1)}(\theta, \kappa)} + \frac{\lambda^2}{\mu_y^2} \sa{P_{1,2}^{(\infty, 2)}(\theta, \kappa)}\right) 
                    & \frac{1}{\lambda^2}\left(\sa{P_{2,2}^{(\infty, 1)}(\theta, \kappa)} + \frac{\lambda^2}{\mu_y^2} \sa{P_{2,2}^{(\infty, 2)}(\theta, \kappa)}\right)
            \end{array}\right],
        \end{scriptsize}    
    \end{equation}
    where \sa{$P_{1,1}^{(\infty, k)}$, $P_{1,2}^{(\infty, k)}$ and $P_{2,2}^{(\infty, k)}$ for $k=1,2$, and $P_c$ are polynomials of $(\theta,\kappa)$, can be made explicit and are 
   provided in Table~\ref{rso:table:polynomials} of Appendix \ref{rso:sec:constants}.} 
    {
    Moreover, for any $\lambda \in \spectrum(K)$, all 
    elements of the matrix $\Sigma^{\infty,\lambda}\in\reals^{4\times 4}$ scale with $(1-\theta)$ as $\theta \to 1$. 
    }\looseness=-1
\end{thm} 
\begin{proof}
\sa{The proof is given in Section~\ref{sec:proof-lim-cov}.}
\end{proof}
%
According to Theorem \ref{rso:thm:quadratics}, the matrix $\Sigma^{\infty,\lambda}$ has the property that it scales with $(1-\theta)$ as $\theta \to 1$; we leverage this fact to 
establish the tightness of our analysis in the next section (Section \ref{rso:sec:tightness}).

\mg{In Figure~\ref{rso:fig:cov_convergence}, we illustrate Theorem \ref{rso:thm:quadratics} on a simple quadratic problem where primal and dual iterates are scalar, i.e., $d=1$ and $K=c$ is a scalar. 
In the three panels of Figure~\ref{rso:fig:cov_convergence}, we consider three problems P1, P2, P3 from left to right where the problem constants, $(c, \mux, \muy, \delta)$, are chosen as \sa{P1}: $(1, 4.4, 1.5, 35)$ - \sa{P2}: $(1, 2, 20, 50)$ - \sa{P3}: $(10^{-3}, 0.205, 0.307, 5)$. SAPD was run $2000$ times for $500$ iterations \saa{using CP parameterization in~\eqref{rso:eq:cb_params}} with $\theta=0.99$. 
For each problem,} we estimate the empirical covariance matrix \sa{$\Sigma_n$ for $n\in\{2^k:~k=0,\ldots, \log_2(500)\}$.}
\sa{The level set $\{z:\ z^\top\Sigma^\infty z = 1\}$} for the theoretical covariance matrix \sa{$\Sigma^\infty$} derived in Theorem \ref{rso:thm:quadratics} is represented by a brown edged ellipse on each plot.
Figure~\ref{rso:fig:cov_convergence} \mg{suggests} the linear convergence of the matrices to the equilibrium matrix~\sa{$\Sigma^\infty$}. 
Subsequently, we observe on these three examples how noise accumulates along iterations, producing covariance matrices that are non-decreasing in the sense of the Loewner ordering. 
\mg{This monotonicity behavior} is intuitively expected as the noise accumulates over the iterations, but can also be proven using the fact that covariance matrix $\Sigma_n$ of $z_n = [x_n^\top,y_n^\top]^\top$ follows a Lyapunov recursion \citep{laub1990sensitivity,hassibi1999indefinite}. 
We \sa{elaborate further on} this property by showing below that convergence of $\Sigma_n$ to $\Sigma^\infty$ happens at a linear rate characterized by the spectral radius of \sa{a particular matrix related to the SAPD iterations}. {The proof builds on the spectral characterizations of the covariance matrix $\Sigma^\infty$ obtained in the proof of Theorem \ref{rso:thm:quadratics}.}
\begin{corollary}\label{coro:cov-rate}
    {In the 
    premise of Theorem~\ref{rso:thm:quadratics}}, for any $\theta \in \Big( \kappa_{\text{max}}^{-1} (\sqrt{1 + \kappa_{\text{max}}^2} - 1), 1\Big)$, the sequence of covariance matrices $\Sigma_n \defineq \expectation[z_n z_n^\top]$ satisfies  
    \begin{equation}\label{rso:eq:conv_speed_Sigma_infty}
        \spectralradius{\Sigma_n - \Sigma^{\infty}} = \mathcal{O}\left(\rho(A)^{2n}\right),
    \end{equation}
    where $z_n \defineq \mg{(x_n^\top, y_n^\top)^\top}$, 
    \sa{$\scalemath{0.9}{A 
        \defineq 
            \left[\begin{array}{cc}
                \frac{1}{1+\tau \mu_x} I_d & \frac{-\tau}{\left(1+\tau \mu_x\right)} K \\
                \frac{1}{1+\sigma \mu_y}\left(\frac{\sigma(1+\theta)}{1+\tau \mu_x}-\sigma \theta\right) K & \frac{1}{1+\sigma \mu_y}\left(I_d-\frac{\tau \sigma(1+\theta)}{1+\tau \mu_x} K^2 \right)
            \end{array}\right]}$ {
            with $\rho(A)<1$.}}
\end{corollary}
\begin{proof}
\sa{The proof is provided in Appendix~\ref{sec:cov-rate-proof}.
}
\end{proof}
\begin{figure}[h]
  \centering
  \includegraphics[width=0.9\textwidth]{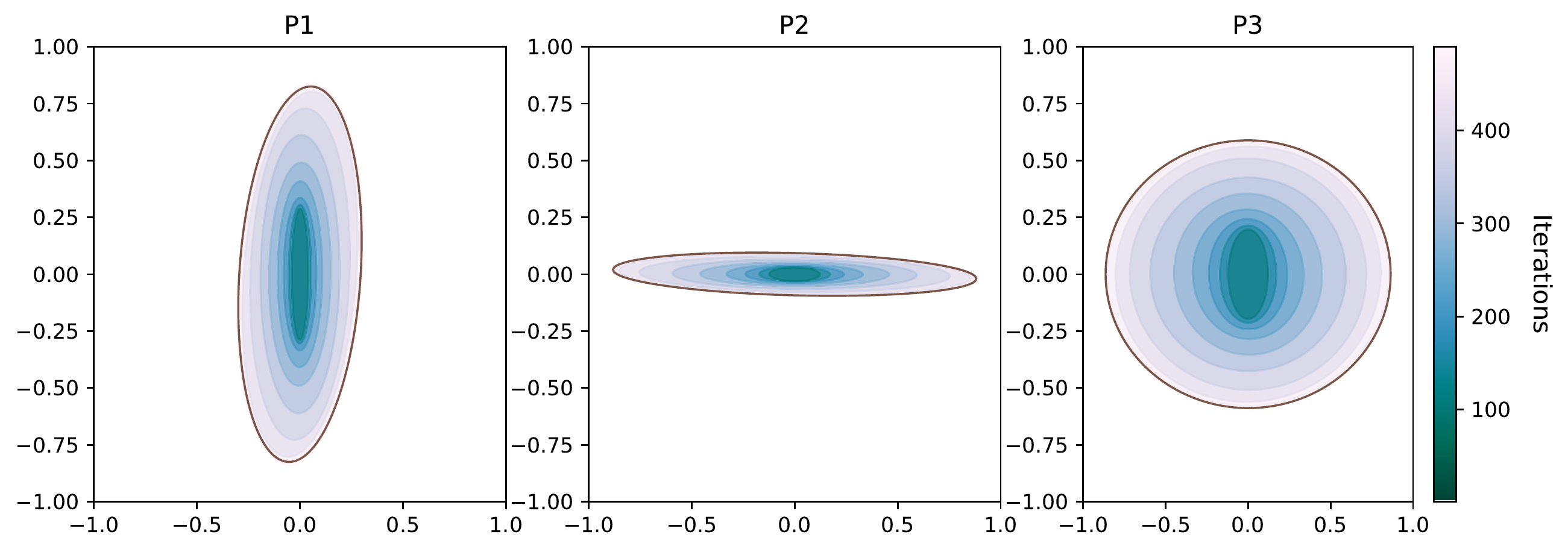}
  \caption{Noise accumulation over SAPD iterates: 
  \sa{$\{\Sigma_n\}$} converges to an equilibrium covariance \sa{$\Sigma^\infty$} due to convergence in distribution $z_n \convergesindistributionto z^\infty$ derived in Thereom~\ref{rso:thm:quadratics}. By \saa{Corollary~\ref{coro:cov-rate},} 
  convergence occurs at a linear rate \sa{which is plotted for the level sets $\{z\in\reals^{2}: z^\top \Sigma_n z = 1\}$ for $d=1$.} The constants $(\sa{c}, \mux, \muy, \delta)$ for each problem \mg{from left to right}: \sa{P1}: $(1, 4.4, 1.5, 35)$ - \sa{P2}: $(1, 2, 20, 50)$ - \sa{P3}: $(\sa{10^{-3}}, 0.205, 0.307, 5)$, \sa{where $K=c\in\mathbb{R}$}.}
  \label{rso:fig:cov_convergence}
\end{figure}
%
\subsection{Tightness analysis}\label{rso:sec:tightness}
We discuss in this section that the constants given in~\sa{Theorem}~\ref{rso:thm:high_probability_bound_sapd} are tight in the sense that under the \saa{CP} parameterization given in~\eqref{rso:eq:cb_params}, which corresponds to a particular solution of the matrix inequality in \eqref{rso:eq:matrix_inequality_params}, the dependency of these constants to $\theta$ and $p$ cannot be improved {when the number of iterations $n$ is sufficiently large}. To this end, we consider quadratic problems subject to \sa{additive} isotropic Gaussian noise \sa{for which} we can do exact computations, \sa{i.e., both $\{\noisegradx{k}\}$ and $\{\noisegrady{k}\}$ are i.i.d \saa{zero-mean} Gaussian random vector sequences with isotropic covariances, and these sequences are independent from each other as well.}

In Section~\ref{rso:eq:quadratic_problem}, \sa{under the isotropic Gaussian noise assumption,} we show that the distribution \mg{$\pi_n$} of the iterates $z_n \defineq (x_n, y_n)$ converges to a Gaussian distribution \mg{$\pi_\infty$} with mean $\optz = (\optx, \opty)$ and a covariance matrix $\Sigma^\star$ for which we provide a formula in~\eqref{eq:limit_cov_informal_non_zero}. 
If we let $z_\infty$ denote a random variable with the stationary distribution $\pi_\infty$, {Theorem \ref{rso:thm:high_probability_bound_sapd} implies that}
\begin{equation}\label{eq-quantile-upper-bound}
    Q_p(\norm{z_\infty-\optz}^2) =\limsup_{n \to \infty} Q_p(\norm{z_n-\optz}^2)
        = \cO\Big(
                    (1-\theta)\delta^2\Big(1+\log\Big(\frac{1}{1-p}\Big)\Big)
            \Big),
\end{equation}
as $\theta\to 1$. 
This upper bound (grows) scales linearly with respect to $1-\theta$ and $\log(\frac{1}{1-p})$, and a natural question is whether this scaling can be improved.
In the next proposition we provide lower bounds on the quantiles of $\squarednormtwo{z_\infty}$ {that also grows linearly with respect to $1-\theta$ and $\log(\frac{1}{1-p})$, matching the upper bound in \eqref{eq-quantile-upper-bound}. 
Therefore, we conclude that our analysis is tight in the sense that we cannot expect to improve our bound in \eqref{eq-quantile-upper-bound} 
in terms of its dependency to $p$ and $1-\theta$}. 
\begin{thm}\label{rso:prop:tightness}
    Let $(z_n)_{n\geq 0}$ be the sequence initialized at \mg{an arbitrary tuple} $z_0=(x_0,y_0)$ generated by~\sapdname~under the parameterization~\eqref{rso:eq:cb_params} on the quadratic problem~\eqref{rso:eq:quadratic_problem} where $\optz = (0, 0)$. 
    Then, the sequence $(z_n)_{n\geq 0}$ converges in distribution to a Gaussian vector $z_\infty$. 
    Furthermore, for any
    $p \in (0,1)$, \saa{$p$-th} quantile $Q_p(\squarednormtwo{z_\infty-\optz})$ \mg{admits the bound}  
    \[
        \psi_1(p,\theta) \leq Q_p(\squarednormtwo{z_\infty-\optz}) \leq \psi_2(p, \theta),
    \]
    where $\psi_1(p, \theta) = (1-\theta) \log(1/(1-p)) \Theta(1)$ and $\psi_2(p, \theta) = (1-\theta) \bigOh(1+\log(1/(1-p)))$, as $\theta \to 1$.
\end{thm}
\begin{proof}
\saa{The proof is provided in Section~\ref{rso:sec:proof_tightness} of the appendix.}
\end{proof}

%% file: 5.1_recursive_control.tex
\yassine{This section presents general concentration inequalities that will be specialized later for the analysis of~\sapdname. 
The first result is a recursive concentration inequality \saa{extending the result provided in~\citep[Proposition 6.7]{cutler2021stochastic}, which is used in the analysis of the stochastic gradient descent (SGD) method for minimization of a smooth strongly convex function in~\citep{cutler2021stochastic}}. 
Our variant of this inequality enables us to analyze saddle point problems with acceleration, providing new insights on their robustness properties.}
\begin{prop}\label{rso:prop:recursive_control}
    Let $(\sigmafield_n)_{n\geq 0}$ be a filtration on $(\Omega, \sigmafield, \mathbb{P})$. Let $(V_n)_{n\geq 0}$, $(T_n)_{n\geq 0}$, and $(R_n)_{n\geq 0}$, be three scalar \sa{stochastic processes adapted to $(\cF_n)_{n\geq 0}$ with following properties: there exist $\sigma_R, \sigma_T > 0$ such that for all $n \geq 0$,}
    \begin{itemize}
        \itemsep0em
        \item $V_n$ is non-negative;
        \item $\mgfc{\lambda \sa{T_{n+1}}}{ \sigmafield_{n}} \leq e^{\lambda^2 \sigma_T^2 V_n}$ for all $\lambda > 0$, i.e., $T_{n+1}$ \saa{conditioned} on $\sigmafield_n$ is subGaussian;
        \item $\mgfc{\lambda \sa{R_{n+1}}}{\sigmafield_{n}}\leq e^{\lambda \sigma^2_R}$ for all $\lambda \in [0,1/\sigma_R^2]$, i.e., $R_{n+1}$ \saa{conditioned} on $\sigmafield_n$ is subExponential.
    \end{itemize}
    \sa{If there exists $\rho\in(0,1)$ such that}
    \begin{equation}\label{rso:eq:general_recursive_bound}
          V_{n+1} \sa{- T_{n+1} - R_{n+1}} \leq \rho\; V_n, \quad \forall n\geq 0,
    \end{equation}
    then for all~$\lambda \in \Big(0,\min\{\frac{1}{2\sigma_R^2},\frac{1-\rho}{4\sigma_T^2} \}\Big)$, it holds that
        $\mgf{\lambda V_{n+1}} \leq e^{\lambda \sigma_R^2} \mgf{\frac{\lambda(1+\rho)}{2} V_n}$, 
        for $n\geq 0$.
\end{prop}
\begin{proof}
    \yassine{Our proof follows closely the arguments of~\citep[Proposition 6.7]{cutler2021stochastic}.
    The main difference \sa{is in} the term $T_{n+1}$ which takes the specific form $T_{n+1} = G_{n+1} \sqrt{V_n}$ in~\citep{cutler2021stochastic}, where $G_{n+1}$ \saa{conditioned on $\cF_{n}$} is assumed to be subGaussian.}
    For any $\lambda \geq 0$, \eqref{rso:eq:general_recursive_bound} together with Cauchy-Schwarz inequality implies that
    \[
    \begin{split}
        \mgfc{\lambda V_{n+1}}{\sigmafield_{n}}
            &\leq e^{\lambda \rho V_n} \mgfc{\lambda \sa{(T_{n+1} + R_{n+1})} }{ \sigmafield_{n}} \leq e^{\lambda \rho V_n} \mgfc{2\lambda \sa{T_{n+1}} }{ \sigmafield_{n}}^{1/2} \mgfc{2\lambda \sa{R_{n+1}} }{ \sigmafield_{n}}^{1/2}. \\
    \end{split}
    \]
    Thus for 
    $\lambda \in\Big(0, \frac{1}{2\sigma_R^2}\Big]$, we have 
    $\mgfc{\lambda V_{n+1} }{ \sigmafield_{n}}
            \leq e^{\lambda \sigma_R^2} 
                 e^{\lambda (\rho + 2\lambda \sigma_T^2) V_n}$.
    Setting $0 \leq \lambda\leq \min\left\{\frac{1}{2\sigma_R^2}, \frac{1-\rho}{4 \sigma_T^2}\right\}$ and taking the non-conditional expectation, we ensure that  $\mgf{\lambda V_{n+1}} \leq e^{\lambda \sigma_R^2} \mgf{\lambda \frac{1+\rho}{2} V_n}$.
    This completes the proof. 
\end{proof}
Unrolling the above recursive property on the moment generating function of $V_n$ provides us with high probability results \sa{on} $(V_n)_{n\geq 0}$, given in the next result.
\begin{prop}\label{rso:prop:high_proba_recursive_control}
    Let $V_n, T_n, R_n$ be defined as in Proposition~\ref{rso:prop:recursive_control}. 
    Then, for all $n \geq 0$ and $\lambda \in \left[0, \min\left\{\frac{1-\rho}{4 \sigma_T^2}, \frac{1}{2\sigma_R^2}\right\} \right]$,
    \begin{equation}\label{rso:eq:rate_mgf_scalar_process}
        \mgf{\lambda V_n} \leq  e^{\frac{2\lambda \sigma_R^2}{1-\rho}} \mgf{\lambda \left(\frac{1+\rho}{2}\right)^n V_0}.
    \end{equation}
    \yassine{Furthermore, if $V_0 = C_0$ is constant, then
    \begin{equation}
        \probability\left[V_n \leq \left(\frac{1+\rho}{2}\right)^n C_0 + \frac{2\sigma_R^2}{1-\rho}\left(1+\max\left\{1, 2~\frac{\sigma_T^2}{\sigma_R^2}\right\} 
        \log\left(\frac{1}{1-p}\right)\right)
        \right] \geq p.
    \end{equation}
    Alternatively, if $V_0$} can be expressed as $V_0 = C_0 + \mathcal{U}$ \saa{such that} $C_0 \geq 0$ is constant and $\mathcal{U}$ satisfies  $\mgf{\lambda \mathcal{U}} \leq e^{\alpha \lambda + \beta \lambda^2}$, for all $\lambda \in \sa{\left[0, \frac{1}{\bar \alpha}\right]}$, for some constants $\alpha, \bar \alpha > 0$ and {$\beta\geq 0$}, then for any $p\in\sa{[0},1)$ and \sa{$\lambda\in[0,\gamma]$} where \sa{$\gamma \defineq \frac{1-\rho}{\max\{\bar \alpha, 2 \sigma_R^2, 4 \sigma_T^2\}}$, we have}
    \begin{equation}\label{rso:eq:high_prob_recursive_control}
        \mathbb{P}\left(V_{n} 
            \leq \left(\frac{1+\rho}{2}\right)^n (C_0 + \alpha) + \sa{\left(\frac{1+\rho}{2}\right)^{2n} \lambda \beta} + \frac{2\sigma_R^2}{1-\rho} + \sa{\frac{1}{\lambda}} \log\left(\frac{1}{1-p}\right)
            \right) \geq p.
    \end{equation}
\end{prop}
\begin{proof}
    Let us first prove by induction on $n$ that for all $\lambda \in \left(0, \min\left\{\frac{1-\rho}{4 \sigma_T^2}, \frac{1}{2\sigma_R^2}\right\} \right)$,
    \begin{equation}
    \label{eq:induction_bound}
        \mgf{\lambda V_{n}}
            \leq    
                \mgf{\lambda \left(\frac{1+\rho}{2}\right)^n V_0 
                    + \lambda \sigma_R^2
                        \sum_{k=0}^{n-1} \left(\frac{1+\rho}{2}\right)^{k}}.
    \end{equation}
    For $n=0$, this property holds trivially with the convention $\Sigma_{k=0}^{n-1} = 0$ when $n=0$. 
    Assuming the inequality holds for some $n\geq 0$, next we show it also holds for $n+1$. According to Proposition~\ref{rso:prop:recursive_control}, 
    \[
    \begin{split}
        \mgf{\lambda V_{n+1}}
            &\leq e^{\lambda \sigma_R^2} \mgf{\lambda \frac{1+\rho}{2} V_n} \\
            &\leq e^{\lambda \sigma_R^2} 
                    \mgf{
                        \lambda \frac{1+\rho}{2} \left(\frac{1+\rho}{2}\right)^n V_0 
                        + \lambda \frac{1+\rho}{2} \sigma_R^2
                            \sum_{k=0}^{n-1} \left(\frac{1+\rho}{2}\right)^{k} 
                    }
            = \mgf{
                        \lambda \left(\frac{1+\rho}{2}\right)^{n+1} V_0 
                        + \lambda \sigma_R^2
                            \sum_{k=0}^{n} \left(\frac{1+\rho}{2}\right)^{k} 
                    },
    \end{split}
    \]
    where the second inequality follows from the induction hypothesis since 
    \[
        0<\lambda (1+\rho)/2 \leq \lambda\leq \min\left\{\frac{1-\rho}{4 \sigma_T^2}, \frac{1}{2\sigma_R^2}\right\},
    \]
    and this completes the induction. 
    Thus, \eqref{rso:eq:rate_mgf_scalar_process} follows from using  $\sum_{k=0}^{n-1}\Big(\frac{1+\rho}{2}\Big)^k\leq \frac{2}{1-\rho}$ within \eqref{eq:induction_bound}. 
    The \yassine{remaining statements follow from a Chernoff bound; indeed, if $V_0 = C_0$ is constant, we obtain
    \[
        \probability\left[V_n \geq \left(\frac{1+\rho}{2}\right)^n \sa{C_0}+ \frac{2 \sigma_R^2}{1-\rho} + t\right] \leq e^{ \sa{- \lambda t}},
    \]
    for \sa{$\lambda = \frac{1-\rho}{2\max\{\sigma_R^2, 2\sigma_T^2\}}$} and $t = \frac{1}{\lambda} \sa{\log(\frac{1}{1-p})}
    $, \sa{which implies the desired result.}}
    \sa{Next, suppose $V_0=C_0+\cU$ for some constant $C_0$ and $\cU$ as in the hypothesis.} 
    First, observe that for $\lambda \in \left(0, \min\left\{\frac{1-\rho}{4 \sigma_T^2}, \frac{1}{2\sigma_R^2}, \frac{1}{\bar \alpha}\right\}\right)$, 
    \begin{equation}
        \begin{split}
            \mgf{\lambda V_{n}}
                &\leq   
                    e^{\frac{2\lambda \sigma_R^2}{1-\rho}} 
                    \mgf{\lambda \left(\frac{1+\rho}{2}\right)^n (C_0 + \mathcal{U})} 
                    \leq  e^{\lambda \left(
                                \left(\frac{1+\rho}{2}\right)^n (C_0 + \alpha)
                                + \frac{2 \sigma_R^2}{1-\rho}                     
                            \right)
                        + \lambda^2 \left(\frac{1+\rho}{2}\right)^{2n} \beta
                    }. \label{eq:subexponential-bound}
        \end{split}
    \end{equation}
    Thus, for all $t \geq 0$, 
    \[
        \probability\left( V_{n} 
            \geq \left(\frac{1+\rho}{2}\right)^n (C_0 + \alpha) + \frac{2 \sigma_R^2}{1-\rho} + t \right) 
                \leq 
                    e^{\lambda^2 \left(\frac{1+\rho}{2}\right)^{2n} \beta - \lambda t}. 
    \]
    \sa{Fixing an arbitrary non-negative $\lambda$ such that $\lambda \leq \frac{1-\rho}{\max\{\bar \alpha, 2\sigma_R^2, 4 \sigma_T^2\}}$}, we have \sa{$\exp(\lambda^2 \left(\frac{1+\rho}{2}\right)^{\sa{2}n} \beta - \lambda t) = 1-p \iff t = \lambda (\frac{1+\rho}{2})^{2n} \beta + \frac{1}{\lambda} \log (1/(1-p))$,} which \sa{proves \eqref{rso:eq:high_prob_recursive_control}.} 
\end{proof}
Based on Proposition \ref{rso:prop:high_proba_recursive_control}, as a corollary, one can derive convergence rates for the~\sa{CVaR and EVaR risk measures of the scalar process $(V_n)_{n\geq 0}$.}
\begin{cor}\label{rso:cor:cvar_bound_vn}
    Let $V_n, T_n, R_n, \gamma$ be defined as in Proposition~\ref{rso:prop:high_proba_recursive_control}. 
    Then, for any $p \in \sa{[0},1)$ and \sa{$\lambda\in[0,\gamma]$}, 
    \begin{equation}\label{rso:eq:cor:cvar}
        \cvar_p(V_n^{\frac{1}{2}}) \leq \left(\frac{1+\rho}{2}\right)^{\sa{\frac{n}{2}}} \sqrt{C_0 + \alpha + \sa{\lambda} \beta} + \sa{\sqrt{\frac{2}{1-\rho}\sigma_R^2 + \frac{1}{\sa{\lambda}}\left(1 + \log\left(\frac{1}{1-p}\right)\right)}}.
    \end{equation}
\end{cor}
\begin{proof}
    \sa{Note that the first and second terms on the right-hand side of \eqref{rso:eq:high_prob_recursive_control} satisfy
    \[
        \left(\frac{1+\rho}{2}\right)^n (C_0 + \alpha) + \sa{\left(\frac{1+\rho}{2}\right)^{2n} \lambda \beta} 
        \leq \left(\frac{1+\rho}{2}\right)^n (C_0 + \alpha + \lambda \beta).
    \]
    Hence, by integrating the resulting looser bound with respect to $p$, and using CVaR's integral formulation in~\eqref{rso:eq:def_cvar},} we obtain
    \[
        \cvar_p(V_n) \leq \left(\frac{1+\rho}{2}\right)^n \left(C_0 + \alpha + {\lambda \beta}\right) + \frac{2\sigma_R^2}{1-\rho} + {\frac{1}{\lambda}} \left(1 +  \log\left(\frac{1}{1-p}\right) \right),
    \]
    \sa{which directly implies \eqref{rso:eq:cor:cvar}, due to} Lemma~\eqref{lem:cvar_pythagoras} and the sub-additivity of $t \mapsto \sqrt{t}$.
\end{proof}
%
\begin{cor}\label{rso:cor:evar_bound}
    Let $V_n, T_n, R_n, \gamma$ be defined as in Proposition~\ref{rso:prop:high_proba_recursive_control}. 
    Then, for any $p \in \sa{[0},1)$, \sa{and $\lambda\in[0,\gamma]$,}
    \begin{equation}\label{eq:rso:cor:evar_vn}
        \evar_p\left(V_n^{\frac{1}{2}}\right) 
            \leq 
                \left(\frac{1+\rho}{2}\right)^{n / 2}\sa{\sqrt{C_0
                +\alpha
                +\sa{\lambda} \beta}} 
                + \sqrt{\frac{2}{1-\rho}} \sigma_R 
                + \mg{\left( \sqrt{\frac{1}{\sa{\lambda}}\log(\frac{1}{1-p})} + \frac{\sqrt{\pi}}{\sqrt{\sa{\lambda}}}\right)}.                
    \end{equation}
\end{cor}
\begin{proof}
    \sa{The bound in \eqref{rso:eq:high_prob_recursive_control} of Proposition~\ref{rso:prop:high_proba_recursive_control}} ensures that for all $p \in [0, 1)$ \sa{and $\lambda\in[0,\gamma]$}, \mg{the $p$-th quantile of $V_n$ satisfies}
    \[
        Q_p\left(V_n\right) 
            \leq \left(\frac{1+\rho}{2}\right)^n\left(C_0+\alpha+ \sa{\lambda} \beta  \right)
            + \frac{2 \sigma_R^2}{1-\rho}
            + \frac{1}{\sa{\lambda}} \log \left(\frac{1}{1-p}\right);
    \]
    \sa{hence, non-negativity of $V_n$, Lemma~\ref{lem:cvar_pythagoras} and sub-additivity of $t \mapsto \sqrt{t}$ together imply that}
    \begin{equation}\label{rso:eq:qp_sqrt_vn}
        Q_p\left(V_n^{1 / 2}\right) 
            \leq \left(\frac{1+\rho}{2}\right)^{n / 2}\sa{\sqrt{C_0
                +\alpha+\lambda \beta}}
            + \sqrt{\frac{2}{1-\rho}} \sigma_R
            +\sa{\frac{1}{\sqrt{\lambda}}} \log \left(\frac{1}{1-p}\right)^{1 / 2}.
    \end{equation}
    For $n\geq 0$, let
    $U_n 
            \defineq 
                V_n^{1 / 2}-\left(\frac{1+\rho}{2}\right)^{n / 2}\sa{\sqrt{C_0
                +\alpha+\lambda \beta}}
                - \sqrt{\frac{2}{1-\rho}} \sigma_R,$
    and note that~\eqref{rso:eq:qp_sqrt_vn} implies
    \begin{equation}
        \mathbb{P}\left(U_n>t\right) \leqslant e^{-\sa{\lambda} t^2}\quad\sa{\forall~t\geq 0.}
        \label{ineq-Un-tail-bound}
    \end{equation}
    \sa{Therefore,} following standard arguments from~\citep{vershynin2018high}, we have \sa{for any $\eta>0$ that}
  {\small \begin{eqnarray*} \mathbb{E}(e^{\sa{\eta} U_n}) &=& \int_{0}^{\infty} \probability\left[e^{\sa{\eta} U_n}>t\right]dt 
   = \int_{-\infty}^{\infty} \probability\left[e^{\sa{\eta} U_n}>e^u\right]e^u du  \\
    &=& \int_{-\infty}^{0} \probability\left[e^{\sa{\eta} U_n}>e^u\right]e^u du + \int_{0}^{\infty} \probability\left[e^{\sa{\eta} U_n}>e^u\right]e^u du\\
   &\leq&\int_{-\infty}^{0} e^u du  + \int_{0}^{\infty} \sa{e^{-\frac{\lambda }{\eta^2}u^2}} e^u du = 1 + \sa{e^{\frac{\eta^2}{4\lambda}}}\int_{0}^{\infty} \sa{e^{-\frac{\lambda}{\eta^2} (u-\frac{\eta^2}{2\lambda})^2}}  du 
    1 + \sa{e^{\frac{\eta^2}{4\lambda}}\int_{-\frac{\eta^2}{2\lambda}}^{\infty} e^{-\frac{\lambda}{\eta^2} s^2}  ds}\\
    &\leq&  1 + \sa{\eta e^{\frac{\eta^2}{4\lambda}}\sqrt{\frac{\pi}{\lambda}}} \leq \sa{\Big(1+\eta\sqrt{\frac{\pi}{\lambda}}\Big) e^{\frac{\eta^2}{4\lambda}}},
    \end{eqnarray*}}
    where we used \eqref{ineq-Un-tail-bound}. 
    On the other hand, 
    {\small 
    \begin{eqnarray*}
        \evar_p\left[ U_n \right] 
            &=& \inf_{\eta > 0} \Big\{- \eta^{-1} \log(1-p) 
                + \eta^{-1} \mg{\log} \expectation\left[e^{\eta U_n}\right] \Big\} \\
            &\leq& \inf_{\eta > 0} - \eta^{-1} \log(1-p)
                +  \eta^{-1} \left(\eta^2 / (4\lambda) + 
                                           \eta \sqrt{\pi} / \sqrt{\lambda}  \right) 
            = \sqrt{\log(\frac{1}{1-p})} / \sqrt{\lambda}  + \sqrt{\pi} / \sqrt{\lambda},
    \end{eqnarray*}}
    where we used $\log(1+x)\leq x$ for $x\geq 0$.
    Finally, by translation invariance of the $\evar$, we obtain $\evar_p\left[ V_n^{\frac{1}{2}} \right] \leq  
                \left(\frac{1+\rho}{2}\right)^{n / 2}
                \sa{\sqrt{C_0
                    +\alpha
                    +\lambda \beta}} 
                + \sqrt{\frac{2}{1-\rho}} \sigma_R  + \sa{\left( \sqrt{\frac{1}{\lambda}\log\Big(\frac{1}{1-p}\Big)} + \frac{\sqrt{\pi}}{\sqrt{\lambda}}\right)}$.                
\end{proof}
\yassine{
We finish with a bound on the $\chi^2$-based risk measure, as defined in Table~\ref{rso:table:risk_measures}.
\begin{cor}
\label{cor:chi-square}
    Let $V_n, T_n, R_n, \gamma$ be defined as in Proposition~\ref{rso:prop:high_proba_recursive_control}. 
    Then, for any $r > 0$, \sa{and $\lambda\in[0,\gamma]$,}  
    \begin{equation}\label{eq:rso:cor:chi2_vn}
        \riskMeasure_{\chi^2, r}\left(V_n^{\frac{1}{2}}\right) 
            \leq 
                \left(\frac{1+\rho}{2}\right)^{n/2}
                \sqrt{C_0
                    +\alpha
                    +\sa{\lambda} \beta} 
                + \sqrt{\frac{2}{1-\rho}} \sigma_R
                + \left(\sqrt{\frac{1}{\sa{\lambda}} \log\left(1+r\right)} 
                + \frac{\sqrt{\pi}}{\sa{\sqrt{\lambda}}}\right).
    \end{equation}
\end{cor}
\begin{proof}
    By\citep[Theorem 5]{gibbs2002choosing}, for all $\sa{\mathbb{Q}} \ll \probability$, we have
        $D_{\varphi_{\text{KL}}}(\mathbb{Q} \saa{\Vert} \probability) \leq 
            \log \left(1 + D_{\varphi_{\chi^2}}(\mathbb{Q} \saa{\Vert} \probability) \right)$,
    \sa{where $\varphi_{\text{KL}}(t)=t\log(t)-t+1$. 
    Therefore, for any integrable random variable $U : \Omega \to \R$,} 
    \[
        \sup_{\sa{\mathbb{Q}:} D_{\varphi_{\chi^2}}(\mathbb{Q} \saa{\Vert} \probability) \leq r} \expectation_{\mathbb{Q}}[U] 
            \sa{\leq} \sup_{\sa{\mathbb{Q}:} D_{\varphi_{\text{KL}}}(\mathbb{Q} \saa{\Vert} \probability) \leq \log(1 + r)} \expectation_{\mathbb{Q}}[U] = \evar_{1 - 1/(1+r)}(U),
    \]
   whenever $\evar_{1 - 1/(1+r)}(U)<\infty$, where we used the EVaR representation given in Table \ref{rso:table:risk_measures}. 
   The statement follows directly from Corollary~\ref{rso:cor:evar_bound}.
\end{proof}
}
In the next section, we design scalar processes $V_n, T_n, R_n$ which satisfy the above assumptions while dominating the error $\cD_n$ on \sapdname~iterates, so that Proposition~\ref{rso:prop:recursive_control}, Corollaries~\ref{rso:cor:evar_bound} and~\ref{cor:chi-square} will allow us to prove our main results, \saa{stated in} Theorem~\ref{rso:thm:high_probability_bound_sapd} and Theorem~\ref{thm:risk-nounds}.

%% file: 5.2_proof.tex
\begin{figure}[ht]
    \centering
    \includegraphics[width=\linewidth]{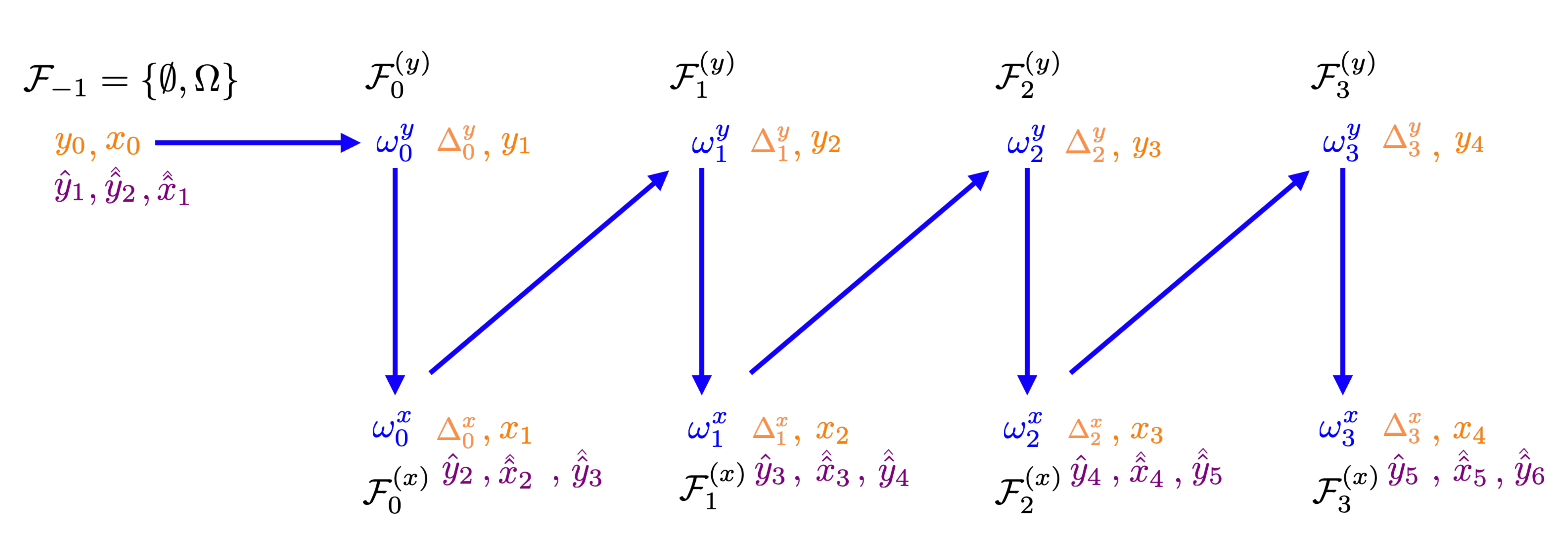}
    \caption{Measurability of 
    \sapdname~sequences. Our analysis is made possible by the introduction of {predictable} counterparts $\noiselessx_{k}, \doublenoiselessx_{k}, \noiselessy_{k}, \doublenoiselessy_{k}$ to the iterates $x_k,y_k$ as defined in~\eqref{rso:eq:def_noiseless_variables}.}
    \label{rso:fig:measurability_fig}
\end{figure}
%
For proving the main results of this paper, namely Theorem~\ref{rso:thm:high_probability_bound_sapd} and Theorem~\ref{thm:risk-nounds}, the application of the recursive control inequality from Section~\ref{rso:app:recursive_control} is not straightforward. In particular, Gauss-Seidel type updates within \sapdname~significantly complicate 
the measurability properties of \sapdname~iterate sequence, as illustrated in Figure~\ref{rso:fig:measurability_fig}: the iterates $x_k$ and $y_k$ are measurable with respect to different filtrations $\mathcal{F}_{k-1}^{(x)}$ and $\mathcal{F}_{k-1}^{(y)}$. We circumvent this issue by introducing a stochastic process {$\{V_k\}_{k\geq 0}$} that almost surely upper bounds the distance to the saddle point while exhibiting simpler measurability characteristics {as discussed next}. We note that even though algorithms with Gauss-Seidel type updates, such as \sapdname, are significantly more complicated to analyze than their Jacobi counterparts, such an analysis is rewarding in the sense that algorithms with Gauss-Seidel type updates can often be faster than those using Jacobi type updates, see \citep{zhang2021robust,zhang2022near}. 
Indeed, our analysis for \sapdname~allowed us to obtain high-probability bounds that demonstrate an accelerated behavior for a stochastic primal-dual algorithm for SP problems. 

\subsubsection{Proof of Theorem~\ref{rso:thm:high_probability_bound_sapd}}
\label{sec:proof-high-probability}

\yassine{Our proof combines several ingredients.
Let $(\tau,\sigma,\theta,\rho,\sa{\alpha})$ be a solution to the matrix inequality in \eqref{rso:eq:matrix_inequality_params}. 
Recall the weighted distance square metric $\cD_n  {=} \frac{1}{2\tau} \|x_{n} - \optx\| + \frac{1-\alpha \sigma}{2 \sigma} \|y_n - \opty\|^2$ we {introduced} in~\eqref{def-metric}}. 
In the proof, we use a scaled version $\cE_n\triangleq \cD_n/\rho$ for $n\geq 0$ {that simplifies the analysis}. 
We first introduce the following auxiliary iterates $\noiselessy_k, \doublenoiselessx_k, \doublenoiselessy_k$ which can be interpreted as the ``noise-free counterparts" to the actual iterates $x_k, y_k$ in the sense {they represent roughly how the algorithm would behave if the gradients were deterministic in lieu of being stochastic at step $k$:}
\begin{align}\label{rso:eq:def_noiseless_variables}
        & \doublenoiselessx_{0} \defineq x_0, & &\doublenoiselessx_{k+1} \defineq \prox{\tau f} \left(x_k - \tau \gradphix(x_k, \noiselessy_{k+1}) \right), \\
        & \noiselessy_{0} \defineq \doublenoiselessy_{0} \defineq y_0,
        & & \noiselessy_{k+1}  \defineq \prox{\sigma g}\Big(y_k + \sigma (1+\theta) \gradphiy(x_k, y_k) - \sa{\sigma}\theta \gradphiy(x_{k-1}, y_{k-1})\Big), \\
        &  & & \doublenoiselessy_{k+1} \defineq \prox{\sigma g} \left(\noiselessy_k  + \sigma (1+\theta) \gradphiy( \doublenoiselessx_k, \noiselessy_{k}) - \sigma \theta \gradphiy( x_{k-1}, y_{k-1}) \right). 
\end{align}%
{where we recall that $x_{-1}=x_0$ and $y_{-1}=y_0$ (see Algorithm \ref{rso:algo:sapd_algorithm}).}
These auxiliary iterates whose measurability properties are illustrated in Figure~\ref{rso:fig:measurability_fig}, will be key for being able to apply Proposition~\ref{rso:prop:high_proba_recursive_control} to obtain high-probability results for \sapdname. 
\saa{Our proof is based on} establishing an almost sure upper bound of the quantity $\distancegap{n+1} + \distancegap{n}$ by a scalar process $V_n$, {
and then showing that our choice of $V_n$} satisfies the assumptions of Proposition~\ref{rso:prop:recursive_control}. 
This will then directly yield the desired high-probability estimates for \sapdname.
{We start with a proposition that provides an almost sure bound to the scaled squared distance metric $\cE_n$.} 
Although this bound is already present in substance in~\citep{zhang2021robust}, it does not appear explicitly. 
For completeness, in Appendix~\ref{sec:as-dom}, we provide 
its proof based on various arguments developed in~\citep{zhang2021robust}.
\begin{prop}\label{rso:prop:almost_sure_bound}
Let $(x_n, y_n)$ be the sequence generated by \sapdname, intialized at an arbitrary tuple ${(x_{-1}, y_{-1})} = (x_0, y_0) \in \setX \times \setY$. Provided that there exists $\tau, \sigma > 0$, and $\theta \geq 0$ that satisfy~\eqref{rso:eq:matrix_inequality_params} for some $\rho \in (0, 1)$ and $\alpha \in {[0, \sigma^{-1})}$, {the following almost sure bound on $\cE_n$,} 
\begin{equation}
\label{rso:eq:almost_sure_upper_bound}
\distancegap{n} 
        \leq \rho^{n-1} \startingbias 
                + \sum_{k=0}^{n-1} \rho^{n-1-k} 
                    \Big( 
                        \dotproduct{\noisegradx{k}}{\optx - x_{k+1}} 
                        + \dotproduct{(1+\theta) \noisegrady{k} - \theta \noisegrady{k-1}}{y_{k+1} - \opty} 
                    \Big),
\end{equation}
{holds for all $n\geq 1$,} where $\distancegap{n} \ylsecond{=} \frac{1}{2\rho\tau} \squarednormtwo{x_{n} - \optx} + \frac{1-\alpha \sigma}{2\rho \sigma} \squarednormtwo{y_n - \opty}$, and $\startingbias \ylsecond{=} \frac{1}{2\tau} \squarednormtwo{x_0 - \optx} + \frac{1}{2\sigma} \squarednormtwo{y_0 - \opty}$.
\end{prop}
\begin{proof}
    The proof is provided in Appendix~\ref{sec:as-dom}.
\end{proof}
{Now, equipped with Proposition~\ref{rso:prop:almost_sure_bound}}, {we can write}
{\small
\begin{align}
\label{eq:SAPD-bound}
\begin{split}
    \distancegap{n} 
        &\leq \rho^{n-1} \startingbias 
                + \sum_{k=0}^{n-1} \rho^{n-1-k} 
                    \Big( 
                        \dotproduct{\noisegradx{k}}{\optx - x_{k+1}} 
                        + \dotproduct{(1+\theta) \noisegrady{k} - \theta \noisegrady{k-1}}{y_{k+1} - \opty} 
                    \Big) \\
        &= \rho^{n-1} \startingbias 
            \begin{aligned}[t]
                &+ \sum_{k=0}^{n-1} \rho^{n-1-k} \left( 
                        \dotproduct{\noisegradx{k}}{\optx - \doublenoiselessx_{k+1}} 
                        + (1+\theta) \dotproduct{\noisegrady{k}}{\noiselessy_{k+1} - \opty}
                        - \theta \dotproduct{\noisegrady{k-1}}{\doublenoiselessy_{k+1} - \opty}
                    \right) \\
                &+ \sum_{k=0}^{n-1} \rho^{n-1-k} \left( 
                        \dotproduct{\noisegradx{k}}{\doublenoiselessx_{k+1} - x_{k+1}} 
                        + (1+\theta) \dotproduct{\noisegrady{k}}{y_{k+1} - \noiselessy_{k+1}}
                        - \theta \dotproduct{\noisegrady{k-1}}{y_{k+1} - \doublenoiselessy_{k+1}}
                    \right).\\
            \end{aligned}
\end{split}
\end{align}
}
\sa{For $k\geq 0$,} {introducing the scalar quantities}
\begin{subequations}\label{eq:P1P2Q-sequence}
    \begin{align}
        & \varidx{P}{k}{1} \defineq \dotproduct{\noisegradx{k}}{\optx - \doublenoiselessx_{k+1}} 
                            + (1+\theta) \dotproduct{\noisegrady{k}}{\noiselessy_{k+1} - \opty},\qquad \varidx{P}{k}{2} \defineq \frac{-\theta}{\rho} \dotproduct{\noisegrady{k}}{\doublenoiselessy_{k+2} - \opty}, \\
        & Q_k \defineq \dotproduct{\noisegradx{k}}{\doublenoiselessx_{k+1} - x_{k+1}}
                            + (1+\theta) \dotproduct{\noisegrady{k}}{y_{k+1} - \noiselessy_{k+1}}
                            - \theta \dotproduct{\noisegrady{k-1}}{y_{k+1} - \doublenoiselessy_{k+1}},
    \end{align}
\end{subequations}
rearranging the sums \sa{in~\eqref{eq:SAPD-bound} and \sa{using $\noisegrady{-1}=\mathbf{0}$}, we may write \eqref{eq:SAPD-bound} equivalently as follows:}
\[
    \distancegap{n} 
        \leq  \rho^{n-1} \startingbias  
                + \sum_{k=0}^{n-1} \rho^{n-1-k} \varidx{P}{k}{1}  
                + \sum_{k=\sa{0}}^{n-2} \rho^{n-1-k} \varidx{P}{k}{2}
                + \sum_{k=0}^{n-1} \rho^{n-1-k} Q_k. 
\]
Now notice that for $n\geq 0$, 
{{\small
\begin{equation}
\label{eq:consecutive_E}
\begin{split}
    \distancegap{n+1} &+ \distancegap{n} 
        \leq
            \begin{aligned}[t]
                 \left(1 + \frac{1}{\rho}\right) 
                 & \left(
                    \rho^{n} \startingbias 
                    + \sum_{k=0}^{n-1} \rho^{n-k} \varidx{P}{k}{1} 
                    + \sum_{k=0}^{n-2} \rho^{n-k} \varidx{P}{k}{2} 
                    + \sum_{k=0}^{n-1} \rho^{n-k} Q_{k} 
                    \right) + \varidx{P}{n}{1} + \rho \varidx{P}{n-1}{2} + Q_n \\
            \end{aligned} \\
        &\qquad~= {\begin{aligned}[t]       
                 \left(1 + \frac{1}{\rho}\right) 
                 & \left(
                    \rho^{n} \startingbias 
                    + \sum_{k=0}^{n} \rho^{n-k} \varidx{P}{k}{1} 
                    + \sum_{k=0}^{n} \rho^{n-k} \varidx{P}{k}{2} 
                    \right) \\
                 & + \left(1 + \frac{1}{\rho}\right)
                  \left(\frac{-1}{1+\rho}
          \varidx{P}{n}{1} 
            - \varidx{P}{n}{2}
            - \frac{\rho}{1+\rho} \varidx{P}{n-1}{2} 
             { - \frac{1}{1+\rho} Q_n 
            + \sum_{k=0}^{n} \rho^{n-k} Q_k}\right).\\        
            \end{aligned}} \\    
\end{split}    
\end{equation}}
}
{We next present a lemma which bounds the terms on the right-hand side of the above equality.}
\begin{lem}\label{rso:prop:barbers_lemma}
\saa{Let $\varidx{P}{n}{1},\varidx{P}{n}{2}$ and $Q_n$ be defined as in~\eqref{eq:P1P2Q-sequence}. Then,} for any $n\geq \sa{0}$, 
\[
\begin{split}
    {\frac{-1}{1+\rho}
          \varidx{P}{n}{1} 
            - \varidx{P}{n}{2}
            - \frac{\rho}{1+\rho} \varidx{P}{n-1}{2} }
            & { - \frac{1}{1+\rho} Q_n 
            + \sum_{k=0}^{n} \rho^{n-k} Q_k}\\
        &\leq {\frac{\rho}{2(1+\rho)} (\distancegap{n+1} + \distancegap{n}) 
            + \sum_{k=0}^{n} \rho^{n-k} \left(\cQ_x \squarednormtwo{\noisegradx{k}} + \cQ_y \squarednormtwo{\noisegrady{k}}\right)},
\end{split}
\]
%
{for some positive constants $\cQ_x$ and $\cQ_y$, 
which depend only on the algorithm and problem parameters, and are provided explicitly in Table~\ref{rso:table:constants_C} of Appendix \ref{rso:sec:constants}}.
\end{lem}
\begin{proof}
The proof is provided in Appendix \ref{sec:proof_barbers_lemma}.
\end{proof}
{Applying Lemma~\ref{rso:prop:barbers_lemma} to the inequality \eqref{eq:consecutive_E}}, 
we obtain
\begin{equation} \label{rso:eq:upperbound_lyap}
    {\frac{\rho}{2(1+\rho)} \left(\distancegap{n+1} + \distancegap{n}\right)}
        \leq \rho^{n} \startingbias 
                \begin{aligned}[t]
                    &+ \sum_{k=0}^{n} \rho^{n-k} \varidx{P}{k}{1} 
                    + \sum_{k=0}^{n} \rho^{n-k} \varidx{P}{k}{2} 
                    + \sum_{k=0}^{n} \rho^{n-k} \Big(\cQ_x\squarednormtwo{\noisegradx{k}} 
                    + \cQ_y 
                    \squarednormtwo{\noisegrady{k}}\Big), \\
                \end{aligned} \\
\end{equation}
For $n \in \N$, we define $V_n$, \sa{$T_{n+1}$ and $R_{n+1}$ as follows:}
\begin{equation}\label{rso:eq:def_process_sapd_proof}
\begin{split}
    V_n &\defineq \rho^{n} \startingbias 
        + \sum_{k=0}^{n} \rho^{n-k} \Big(\varidx{P}{k}{1} 
        + 
        \varidx{P}{k}{2} \Big)
        + \sum_{k=0}^{n} \rho^{n-k} \Big(\cQ_x\squarednormtwo{\noisegradx{k}} 
                    + \cQ_y 
                    \squarednormtwo{\noisegrady{k}}\Big), \\
    \sa{T_{n+1}} &\defineq \varidx{P}{n+1}{1} + \varidx{P}{n+1}{2}, \qquad
    \sa{R_{n+1}} \defineq {\;\cQ_x \squarednormtwo{\noisegradx{n+1}} + \cQ_y \squarednormtwo{\noisegrady{n+1}}} ;\\
\end{split}
\end{equation}
\sa{therefore, \eqref{rso:eq:upperbound_lyap} implies that
\begin{align}
    \label{eq:E-V-bound}
    {\frac{\rho}{2(1+\rho)} \left(\distancegap{n+1} + \distancegap{n}\right)} \leq V_n,\quad\forall~n\geq 0,\quad\text{a.s.}
\end{align}}
Next, we argue that $V_n$ satisfies the assumptions of the recursive control inequality in~\eqref{rso:prop:recursive_control}. To achieve this goal, we will use the following {lemma}.
\begin{lem}\label{rso:lem:upperbound_hat_variables}
{For any \sa{$k \in \N$} and \mg{$\rho\in (0,1)$}, the following inequalities,}
\begin{equation}\label{rso:eq:upperbound_hat_variables}
\begin{split}
    8 \squarednormtwo{\doublenoiselessx_{k+1} - \optx} & \leq  \squarednormtwo{A_1} \; \frac{\rho}{2(1 + \rho)} \left(\distancegap{k} + \distancegap{k-1}\right), \\
    16 {(1+\theta)^2} \squarednormtwo{\noiselessy_{k+1} - \opty} & \leq  \squarednormtwo{A_2} \; \frac{\rho}{2(1 + \rho)} \left(\distancegap{k} + \distancegap{k-1}\right), \\
    16 {\frac{\theta^2}{\rho^2}} \squarednormtwo{\doublenoiselessy_{k+2} - \opty} & \leq  \squarednormtwo{A_3} \; \frac{\rho}{2(1 + \rho)} \left(\distancegap{k} + \distancegap{k-1}\right), \\
\end{split}
\end{equation}
{hold almost surely with the convention that $\distancegap{-1}\triangleq\distancegap{0}$, for some vectors $A_1, A_2, A_3 \in \mathbb{R}^4$ which are explicitly provided in Table~\ref{rso:table:constants_A} of Appendix \ref{rso:sec:constants}.}

\begin{proof}
    The proof is provided in Appendix \ref{proof:rso:lem:upperbound_hat_variables}.
\end{proof}
\end{lem}


Let us now show that $V_n$ satisfies the assumptions of the recursive control inequality in~\eqref{rso:prop:recursive_control}. Indeed, for any $n \geq 0$, $V_{n+1} - V_n = 
    \sa{(\rho-1) V_n} 
    + \varidx{P}{n+1}{1} + \varidx{P}{n+1}{2} 
        + {\cQ_x \squarednormtwo{\noisegradx{n+1}} 
        + \cQ_y \squarednormtwo{\noisegrady{n+1}}}
$,
which \sa{is equivalent to} $V_{n+1} {=} \rho V_n \sa{+ T_{n+1} + R_{n+1}}$. \sa{Let} $\left(\sigmafield_n\right)_{n\geq \sa{-1}}$ be the filtration defined as \sa{$\sigmafield_{-1} \defineq \{\emptyset, \Omega\}$,} 
and $\sigmafield_{n} \triangleq {\sigmafield_{n}^x} =  \sigma\left(\sigmafield_{n-1} \cup\;\sigma(\noisegrady{n}) \cup \sigma(\noisegradx{n}) \right)$, for all $n \geq 0$. 
We first observe that for all $n \in \N$, \sa{$V_n$, $T_n$ and $R_n$ are $\sigmafield_n$-measurable; moreover, $V_n$ is non-negative due to~\eqref{eq:E-V-bound}.} 
Second, for any \sa{$n\geq 0$}, 
\sa{since $\noisegradx{n}$ and $\noisegrady{n}$ are norm-\sa{subGaussian} conditioned {respectively on $\sigmafield_{n}^y$ and $\sigmafield_{n-1}^x$}, for any $\lambda \geq 0$, we get that}
%
%
\[
\begin{split}
    \mgfc{\lambda \sa{T_{n+1}}}{\sigmafield_{n}}    &= {\mathbb{E}\Bigg[e^{\lambda\Big\langle\noisegrady{n+1},~(1+\theta)\left(\noiselessy_{n+2}-\opty\right) - \theta\rho^{-1} (\doublenoiselessy_{n+3} - \opty)\Big\rangle} \expectation \Bigg[e^{\lambda 
                \dotproduct{\noisegradx{n+1}}{\optx - \doublenoiselessx_{n+2}}} \Big | \sigmafield_{n+1}^y  \Bigg] \Big | \sigmafield_{n}\Bigg]} \\
        &\leq {
            e^{8\lambda^2 \left(\squarednormtwo{\doublenoiselessx_{n+2} - \optx} \proxyx^2 + \squarednormtwo{(1+\theta)\left(\noiselessy_{n+2}-\opty \right) - \theta\rho^{-1} \left(\doublenoiselessy_{n+3} - \opty\right)} \proxyy^2  \right)}
        } \\
        & \leq {
            e^{8\lambda^2 \left(\squarednormtwo{\doublenoiselessx_{n+2} - \optx} \proxyx^2 
            + 2 (1+\theta)^2 \squarednormtwo{\noiselessy_{n+2}-\opty} 
            + 2 \theta^{{2}}\rho^{-{2}}  \squarednormtwo{\doublenoiselessy_{n+3} - \opty} \proxyy^2  \right)}
        }, \\
\end{split} 
\]
{where we used {Lemma~\ref{rso:lem:norm_sg_dot} and the inequality $(a+b)^2\leq 2a^2 + 2b^2$ for scalars $a,b$ in the last step}, noting that $\doublenoiselessx_{n+2}$, $\noiselessy_{n+2}$, $\doublenoiselessy_{n+3}$ are all $\cF_n$-measurable.}
Hence, in view of Lemma~\ref{rso:lem:upperbound_hat_variables} {provided above} 
{and the bound in \eqref{eq:E-V-bound},} we have 
%
%
\begin{equation}
{\small
    \mgfc{\lambda \sa{T_{n+1}} }{ \sigmafield_{n}} 
        \leq {e^{\lambda^2 \left(\squarednormtwo{A_1} \proxyx^2  + \left(\squarednormtwo{A_2} + \squarednormtwo{A_3} \right) \proxyy^2  \right) \frac{\rho}{2(1+\rho)}(\distancegap{n+1} + \distancegap{n})}} 
        \leq {e^{\lambda^2 \Big(\squarednormtwo{A_1} \proxyx^2  + \left(\squarednormtwo{A_2} + \squarednormtwo{A_3} \right) \proxyy^2  \Big) V_n},} 
}
\end{equation}
{where we used \eqref{eq:E-V-bound} to obtain the second inequality}. 
Third, for all $n \geq 0$ and $\lambda \in \left(0, {\frac{1}{4 \max\{\cQ_x \proxyx^2, \cQ_y \proxyy^2\}}}\right)$, we have in view of Lemma~\ref{rso:lem:nSG_mgf_bounds}
\begin{equation}
    \begin{split}
        \mgfc{\lambda \sa{R_{n+1}}}{\sigmafield_{n}}
       & 
            ={ \expectation \left[ e^{\lambda \cQ_y \squarednormtwo{\noisegrady{n+1}}} \expectation \left[  e^{\lambda \cQ_x \squarednormtwo{\ylfourth{\noisegradx{n+1}}}} | \sigmafield_{n+1}^y \right]  | \sigmafield_n \right] }
       \leq { \exp
                \left(8 \lambda 
                        \left( \cQ_x \proxyx^2 
                            + \cQ_y  \proxyy^2
                        \right)
                \right). } \\
\end{split}
\end{equation}
Finally, we next argue that $V_0$ can be expressed as $V_0 = \startingbias + \mathcal{U}$ for some $\mathcal{U}$ satisfying $\mgf{\lambda \mathcal{U}} \leq e^{\alpha \lambda + \beta \lambda^2}, \quad \forall \lambda \in \sa{\left[0, \frac{1}{\bar \alpha}\right]}$, for some constants $\alpha, \bar \alpha > 0$ and 
{$\beta\geq 0$}.
First, note that $\noiselessy_{1}, \doublenoiselessx{_1}$ and $\doublenoiselessy_{2}$ are all deterministic quantities {as they depend only on the initialization}; hence, using the inequality $u^\top v \leq \frac{\ylprime{\gamma}}{2(1-\rho)} \squarednormtwo{u} + \frac{(1-\rho)}{2\ylprime{\gamma}} \squarednormtwo{v} $ {for any $\gamma>0$}, we observe that for all $\lambda \in \left[0, {\Big(4 \max\{\left(\frac{1-\rho}{2{\mux}} + \cQ_x \right) \proxyx^2, \left({\frac{(1-\rho)}{{\muy}}} + \cQ_y \right) \proxyy^2\}\Big)^{-1}}\right]$, we have
\[
\begin{split}
    &\mgf{\lambda (V_0 - \startingbias)} = 
        \mgf{\lambda \sa{\left(\varidx{P}{0}{1} + \varidx{P}{0}{2} 
        +   
              {  
               \cQ_x \squarednormtwo{\noisegradx{0}} 
               + \cQ_y \squarednormtwo{\noisegrady{0}} 
                }
            \right)}} \\
    &= \begin{alignedat}[t]{1}
        &  \expectation \left[ 
            e^{ 
            \lambda \dotproduct{\noisegradx{0}}{\optx - \doublenoiselessx_{1}} 
            + \lambda \dotproduct{\noisegrady{0}}{(1+\theta)(\noiselessy_{1} - \opty) 
            - \frac{\theta}{\rho} (\doublenoiselessy_{2} - \opty)} 
                { 
                    + \lambda \cQ_x \squarednormtwo{\noisegradx{0}} 
                    + \lambda \cQ_y \squarednormtwo{\noisegrady{0}} 
                    }} \right] \\
    \end{alignedat} \\
    & \leq 
        {
        \expectation 
            \begin{alignedat}[t]{1}
                & \left[e^{\frac{\lambda}{2(1-\rho)} \left(\ylprime{\mux} \squarednormtwo{\doublenoiselessx_{1} - \optx} + (1+\theta)^{\ylthird{2}} \ylprime{\muy} \squarednormtwo{\noiselessy_1 - \opty} + \theta^{\ylthird{2}} \rho^{-\ylthird{2}} \ylprime{\muy} \squarednormtwo{\doublenoiselessy_2 - \opty}\right)} \right. \\
                & \left.  e^{\frac{\lambda(1-\rho)}{2} \left(\ylsecond{\frac{1}{\mux}} \squarednormtwo{\noisegradx{0}} + \ylsecond{\frac{1}{\muy}}  \ylthird{\squarednormtwo{\noisegrady{0}}} + \ylsecond{\frac{1}{\muy}} \ylthird{\squarednormtwo{\noisegrady{0}}}\right)  + \lambda \cQ_x \squarednormtwo{\noisegradx{0}} 
                + \lambda \cQ_y \squarednormtwo{\noisegrady{0}}} \right] \\
            \end{alignedat}    
            }
            \\
    & = 
        { 
            \begin{alignedat}[t]{1}
                & e^{\frac{\lambda}{2(1-\rho)} \left(\ylprime{\mux} \squarednormtwo{\doublenoiselessx_{1} - \optx} + (1+\theta)^{\ylthird{2}} \ylthird{\muy} \squarednormtwo{\noiselessy_1 - \opty} + \theta^{\ylthird{2}} \rho^{-\ylthird{2}} \ylthird{\muy} \squarednormtwo{\doublenoiselessy_2 - \opty}\right)} \expectation 
                    \left[ 
                        e^{\lambda 
                            \left( 
                                \left(\frac{1-\rho}{2 \ylprime{\mux}} + \cQ_x\right) \squarednormtwo{\noisegradx{0}} +
                                \left(\ylthird{\frac{1-\rho}{\muy}}+ \cQ_y\right) \squarednormtwo{\noisegrady{0}}
                            \right)}
                    \right] \\
            \end{alignedat}
        } \\
    \end{split}
\]
Thus,
\[
    \begin{split}    
    &\mgf{\lambda (V_0 - \startingbias)} \\
    &\leq 
        {
            e^{\frac{\lambda}{2(1-\rho)} \left(\ylprime{\mux} \squarednormtwo{\doublenoiselessx_{1} - \optx} + (1+\theta)^{\ylthird{2}} \ylprime{\muy} \squarednormtwo{\noiselessy_1 - \opty} + \theta^{\ylthird{2}} \rho^{-\ylthird{2}} \ylprime{\muy} \squarednormtwo{\doublenoiselessy_2 - \opty}\right)} 
            e^{8 \lambda 
                            \left( 
                                \left(\frac{1-\rho}{2 \ylprime{\mux}} + \cQ_x\right) \proxyx^2 +
                                \left(\ylthird{\frac{1-\rho}{\ylprime{\muy}}} + \cQ_y\right) \proxyy^2
                            \right)}            
        } \\
    &\leq 
        {
            e^{\frac{\lambda}{2(1-\rho)} \left(\ylsecond{\frac{\mux}{8}} \squarednormtwo{A_1} + \ylsecond{\frac{\muy}{16}} (\squarednormtwo{A_2} + \squarednormtwo{A_3})\right)\; \ylsecond{\frac{\rho}{2(1+\rho)}\; \frac{2}{\rho}} \startingbias} 
            e^{8\lambda 
                            \left( 
                                \left(\frac{1-\rho}{2 \ylprime{\mux}} + \cQ_x\right) \proxyx^2 +
                                \left(\ylthird{\frac{1-\rho}{\ylprime{\muy}}} + \cQ_y\right) \proxyy^2
                            \right)}            
        } \\
    &\leq 
        {
            e^{\frac{\lambda}{2(1-\rho)} \left(\ylsecond{\frac{\mux}{8}} \squarednormtwo{A_1} + \ylsecond{\frac{\muy}{16}} (\squarednormtwo{A_2} + \squarednormtwo{A_3})\right) \startingbias} 
            e^{8\lambda 
                            \left( 
                                \left(\frac{1-\rho}{2 \ylprime{\mux}} + \cQ_x\right) \proxyx^2 +
                                \left(\ylthird{\frac{1-\rho}{\ylprime{\muy}}} + \cQ_y\right) \proxyy^2
                            \right)}            
        }, \\
\end{split}
\]
where the {first} inequality follows from Lemma~\ref{rso:lem:nSG_mgf_bounds}, in the {second} inequality we used Lemma~\ref{rso:lem:upperbound_hat_variables} {given above,} and the relation $\cE_0+\cE_{-1}=2\cE_0 {=} 2\cD_0/\rho\leq 2\startingbias/\rho$, which follows from $1-\alpha\sigma<1$ and the relations $x_{0}=x_{-1}$, $y_{0}=y_{-1}$. 
Hence, we can {apply} Proposition~\ref{rso:prop:high_proba_recursive_control} to {the $V_n, R_n, T_n$ sequence defined in \eqref{rso:eq:def_process_sapd_proof}} with the following choice of parameter values, {\small 
\begin{equation}\label{eq:rso:constants_vn_sapd}
    \begin{split}
          C_0 &= \startingbias, \quad \cU = \varidx{P}{0}{1} + \varidx{P}{0}{2} + \cQ_x \squarednormtwo{\noisegradx{0}} + \cQ_y \squarednormtwo{\noisegrady{0}},\\ 
        \sigma_T^2 &=  {\squarednormtwo{A_1} \proxyx^2  + \left(\squarednormtwo{A_2} + \squarednormtwo{A_3} \right) \proxyy^2}, \quad
        \sigma_R^2 = {8} \; (\cQ_x \proxyx^2 + \cQ_y \proxyy^2), \\
        \alpha &= 
            \begin{aligned}[t]
                \frac{1}{16(1-\rho)} 
                    \bigg(\mux\squarednormtwo{A_1} 
                    &+ \frac{\muy}{2} (\squarednormtwo{A_2} + \squarednormtwo{A_3})\bigg)\; \startingbias  \\
                    &   + \left(4\frac{(1-\rho)}{\mux} + 8 \cQ_x \right) \proxyx^2 
                        + \left(\ylthird{8\frac{(1-\rho)}{\muy}} + 8 \cQ_y \right) \proxyy^2,    
            \end{aligned} \\
        \bar \alpha &= {4 \max\left( \left( \frac{1-\rho}{2\ylprime{\mux}} + \cQ_x \right) \proxyx^2, 
             \left(\ylthird{\frac{(1-\rho)}{\ylprime{\muy}} + \cQ_y} \right)\proxyy^2 \right)}, \quad
        \beta = 0, \end{split} \end{equation}}%
where $\cD_{\tau,\sigma}$ is defined in the statement of Theorem~\ref{rso:thm:high_probability_bound_sapd}.
When we invoke Proposition~\ref{rso:prop:high_proba_recursive_control}, we set $\lambda=\tilde\gamma$ within \eqref{rso:eq:high_prob_recursive_control} for some particular $\tilde\gamma>0$ such that $\tilde\gamma\leq \gamma$ as required by the proposition. Thus, for any $p\in (0,1)$ \sa{and $n\geq 0$, the following inequality
{\small
\[
    \begin{split}
        V_n \leq
            &\left(\frac{1+\rho}{2}\right)^n \bigg(
                            \left(1 + \frac{1}{\ylsecond{16}(1-\rho)} \left(\ylsecond{\mux}\squarednormtwo{A_1} + \ylsecond{\frac{\muy}{2}} (\squarednormtwo{A_2} + \squarednormtwo{A_3})\right) \right) \startingbias \\
            &\quad+ \left(\frac{\ylsecond{4}(1-\rho)}{\ylprime{\mux}} + \ylsecond{8} \cQ_x \right) \proxyx^2 
                            + \left(\ylthird{8}\frac{(1-\rho)}{\ylprime{\muy}}  \ylsecond{+ 8 \cQ_y} \right)\proxyy^2  \bigg)\\
            &+ \frac{\ylsecond{16}}{1-\rho} \left( \cQ_x \proxyx^2 + \cQ_y \proxyy^2\right)
                +\frac{1}{\tilde\gamma}\log\left(\frac{1}{1-p}\right),
    \end{split}
\]}%
holds} with probability at least $p$, with the choice of
{
\begin{align}\label{eq:gamma}
    {\tilde\gamma\triangleq\frac{1-\rho}{\gamma_x \proxyx^2 + \gamma_y \proxyx^2}\leq\gamma\triangleq \frac{1-\rho}{\max\{\bar \alpha, 2\sigma_R^2, 4 \sigma_T^2\}}},
\end{align}}%
where \begin{equation} \gamma_x = 2 \frac{(1-\rho)}{\ylprime{\mux}}  + 16 \cQ_x  + 4 \squarednormtwo{A_1}, \quad \gamma_y = \ylthird{4} \frac{(1-\rho)}{\ylprime{\muy}} + 16 \cQ_y  + 4 (\squarednormtwo{A_2
} + \squarednormtwo{A_3}).\label{eqn-gammas}
\end{equation} 
{In view of~\eqref{eq:E-V-bound}, and noting that $\cD_n = \rho \distancegap{n}$, we obtain $\ylthird{\cD_{n+1} + \cD_n} \leq 2 (1+\rho) V_n \leq 4 V_n$. 
Therefore, 
\[
    \cC_{\tau, \sigma, \theta} = \left(4 + \frac{1}{4(1-\rho)} \left(\ylprime{\mux}\squarednormtwo{A_1} + \ylprime{\frac{\muy}{2}} (\squarednormtwo{A_2} + \squarednormtwo{A_3})\right) \right), \\
\]
{\small
\begin{equation}
\begin{aligned}
    \Xi_{\tau, \sigma, \theta}^{(x,1)} &=  16 \frac{(1-\rho)}{\mux} + 32 \cQ_x, &\quad
    \Xi_{\tau, \sigma, \theta}^{(y,1)} &= \ylthird{32} \frac{(1-\rho)}{\muy} + 32 \cQ_y, \\
    \Xi_{\tau, \sigma, \theta}^{(x,2)} &= \frac{64\cQ_x}{(1-\rho)}, \Xi_{\tau, \sigma, \theta}^{(x,3)} = \frac{4 \gamma_x}{1-\rho}, &\quad
    \Xi_{\tau, \sigma, \theta}^{(y,2)} &= \frac{64\cQ_y}{(1-\rho)}, \quad \Xi_{\tau, \sigma, \theta}^{(y,3)} = \frac{4 \gamma_y}{1-\rho}, \\
\end{aligned}
\label{def-constants-main-thm}
\end{equation}}%
completes the proof of \eqref{rso:eq:high_probability_bound_sapd}.} 
The remaining items to prove regarding the asymptotic properties of $\Xi_{\tau, \sigma, \theta}^{(1)}$ and $\Xi_{\tau, \sigma, \theta}^{(2)}$ \ylsecond{as $\theta \to 1$} {follows from straightforward but tedious computations; for completeness, we provide the details in the separate Lemma~\ref{rso:lem:asymptotic_xi}, provided in Section~\ref{sec:asymp} of the Appendix)}.\looseness=-1

\subsubsection{Proof of Theorem~\ref{thm:risk-nounds}}
\label{sec:risk-proof}
We can deduce Theorem~\ref{thm:risk-nounds} from the above analysis.
\sa{Indeed, the CVaR bound in \eqref{rso:eq:cvar_bound} directly follows from Corollary~\ref{rso:cor:cvar_bound_vn} applied to the process $V_n$ introduced in~\eqref{rso:eq:def_process_sapd_proof}, with the associated constants defined in~\eqref{eq:rso:constants_vn_sapd}.  
Furthermore, the EVaR bound in \eqref{rso:eq:evar_result} follows from Corollary~\ref{rso:cor:evar_bound} applied to the same $(V_n)_{n\geq 0}$. 
Finally, the bound on $\riskMeasure_{\chi^2, r}(\cD_n^{1/2})$ follows from Corollary~\ref{cor:chi-square}.}

%% file: 6_numerical_exps.tex
In this section, we illustrate the robustness properties of \sapdname~when solving bilinear games and distributionally robust learning problems involving both synthetic and real data. 
First we consider the regularized bilinear game presented in~\eqref{rso:eq:quadratic_problem}, 
\[
    \min_{x \in \Rd} \max_{y \in \Rd} \frac{\mux}{2} \norm{x}^2 + x^\top K y - \frac{\muy}{2} \norm{y}^2, \quad \text{for } K \triangleq 10 \tilde K / \lVert \tilde{K} \rVert, \quad \tilde{K} \defineq (M + M^\top)/2,
\]
where $M = (M_{i,j})$ is a $30 \times 30$ matrix with entries sampled from i.i.d standard normal variables. 
We set the regularization variables as $\mux = \muy = 1$. 
We explore two values of the momentum parameter $\theta$ as $\bar{\theta}$ and $1 - (1-\bar{\theta})^2$, with $\bar{\theta} \displaystyle \defineq (1 + \kappa_{\text{max}}^2)^{1/2} -1$ computed based on the threshold value from Theorem~\ref{rso:thm:quadratics}.
We then determine the stepsizes $\tau, \sigma$ according to the CP parameterization~\eqref{rso:eq:cb_params} \mgtwo{where $\rho=\theta$}. 
Finally, \sapdname~is initialized at a random tuple $(x_0, y_0) = 50 (\tilde x_0, \tilde y_0)$, where $\tilde x_0, \tilde y_0 \in \R^{30}$ have entries sampled from i.i.d. standard normal distributions. 
{In Figure \ref{rso:fig:num_exps}, we report the histogram of the distance squared $E_k =\|x_k-\optx\|^2 + \|y_k - \opty\|^2$ to the saddle point $\optz=\mathbf{0}$ after $k=2000$ (top, middle panel) and $k=5000$ iterations (top, right panel) based on 500 sample paths and for both choice of (momentum) parameter values. 
The expected distance $\mathbb{E}[E_k]$ over iterations is also reported on the top, left panel along with the error bars around it. The continuous vertical line in the convergence plots represents the sample average (estimating the expectation $\mathbb{E}[E_k]$), while the dashed vertical line represents $\var_{p}(E_k)$ with $p=0.90$, i.e., the $90^{th}$ percentile of the error $E_k$. 
We observe that the performance is sensitive to the choice of parameters and there are bias/risk trade-offs in the choice of parameters; indeed, when the number of steps is smaller (for $k=2000$), the noise accumulation is not dominant and a smaller rate parameter $\rho=\theta$ allows faster decay of the initialization bias, resulting in better guarantees for the value at risk with $p=0.90$ or equivalently for the 90-th quantile. 
On the other hand when the number of steps is larger (for $k=5000$), there is more risk associated to accumulation of noise and a larger choice of $\rho=\theta$ close to $1$ is preferable, as  this results in smaller primal and dual stepsizes which allows to control the tail risk at the expense of a slower decay of the initialization bias.} 

Next, we aim to solve the following distributionally robust logistic regression problem introduced in~\citep{zhang2021robust}: $\min_{x \in \Rd} \max_{y \in \cP_r} \frac{\mux}{2} \norm{x}^2 + \sum_{i=1}^{n} y_i \phi_i(x) - \frac{\muy}{2} \norm{y}^2$, where $\phi_i(x) \defineq \log(1 + \exp(-b_i a_i^\top x))$, and $\cP_r \defineq \{y \in \R_n^+, \boldsymbol{1}^\top y = 1, \norm{y - \boldsymbol{1}/n}^2 \leq \frac{r}{n^2}\}$, with $r=2\sqrt{n}$. 
We consider two datasets from the UCI Repository\footnote{\url{https://archive.ics.uci.edu/ml/index.php}}, \texttt{DryBean}, and \texttt{Arcene}, and follow the preprocessing protocol outlined in~\citep{zhang2021robust}. 
For each dataset, we run ~\sapdname~with two values $\theta_1, \theta_2$ that are greater than the threshold value $\bar \theta$ given in~\citep[Corollary 1]{zhang2021robust}. \sapdname~is initialized for both datasets at $x_0 = [2, \ldots, 2]$ and $y_0 = \boldsymbol{1}/n$. {In the middle and bottom panels of Figure~\ref{rso:fig:num_exps}, we display the average of the error $E_k$ over the course of the iterations as well as the error histogram for \sapdname~over $500$ runs as we did in the previous experiment.
Our numerical findings are similar to the bilinear case, i.e., to obtain the best risk guarantees, one needs to choose the algorithm parameters in a careful fashion --which 
is inline with our theoretical results, where obtaining the accelerated iteration complexity in Theorem \ref{rso:thm:complexity_sapd_result} requires choosing the parameters in an optimized fashion over the class of admissible CP parameters.} 
\begin{figure}[t]
    \centering
    \begin{minipage}[t]{\textwidth}
        \centering
        \includegraphics[width=\textwidth]{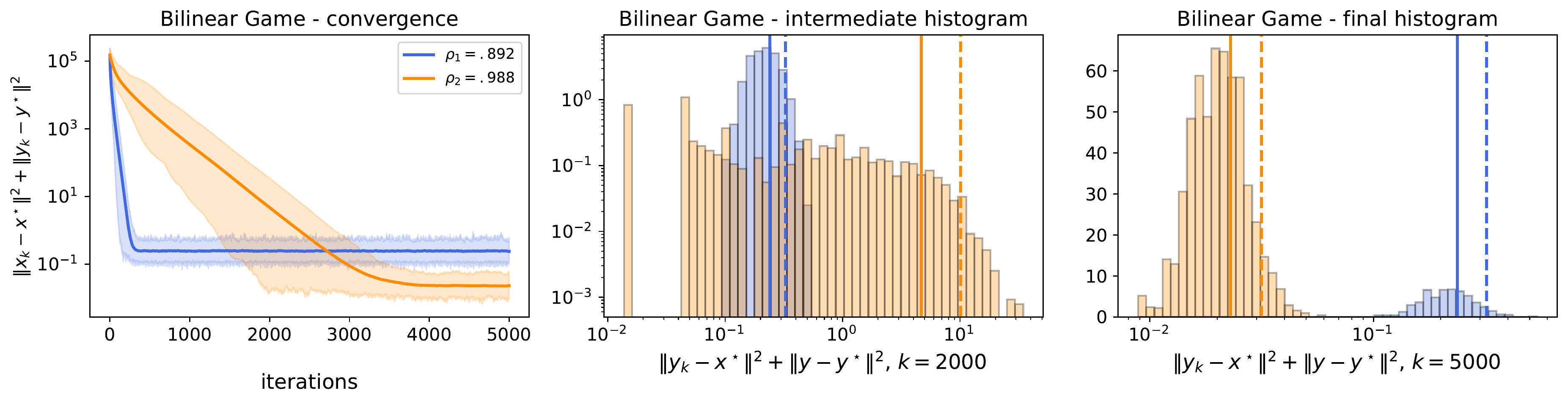}
    \end{minipage}
    \vspace{-0.2cm}
    \begin{minipage}[t]{\textwidth}
        \centering
        \includegraphics[width=\textwidth]{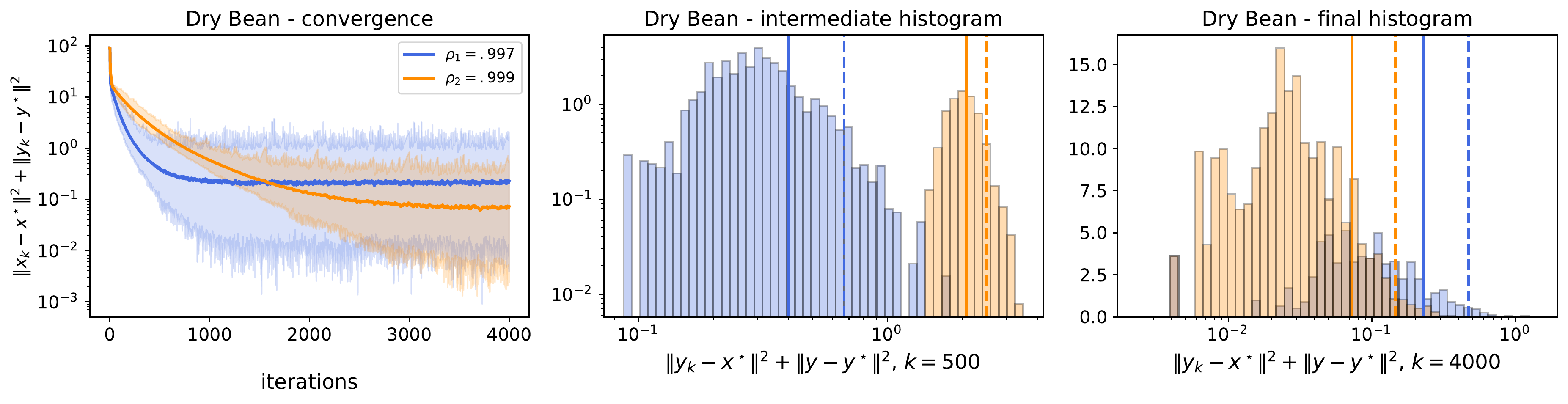}
    \end{minipage}
    \vspace{-0.5cm}
    \begin{minipage}[t]{\textwidth}
        \centering
        \includegraphics[width=\textwidth]{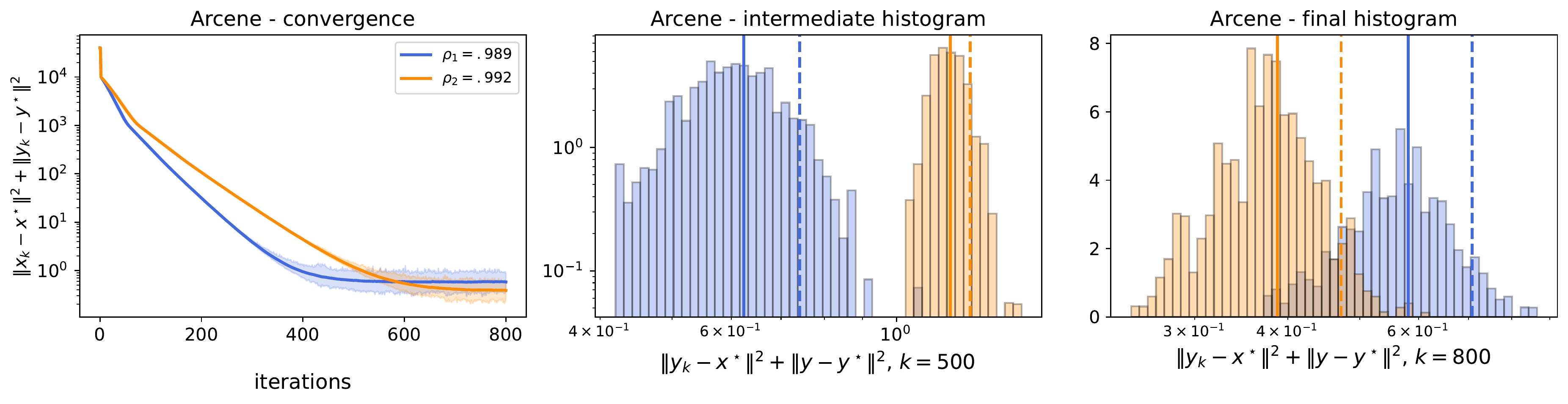}
    \end{minipage}
    \caption{The figure shows the convergence behavior and distribution of performance scores for \sapdname~across three datasets. The left column displays the expected distance squared $E_k$ of \sapdname~iterates to the solution over iterations, while the middle and right columns show histograms of $E_k$ at fixed iterations. The continuous line in the convergence plots represents the average score $\mathcal{E}(E_k)$, while the dashed line represents the $90^{th}$ percentile. The datasets include a synthetically generated bilinear game, and Dry Bean and Arcene from the UCI repository.}
    \label{rso:fig:num_exps}
\end{figure}

%% file: A1_basics.tex
We provide in this section proofs of elementary properties of subGaussian vectors and convex risk measures. 

\subsection{Elementary Properties of Norm-subGaussian Vectors}\label{rso:app:prelimaries}

In this section, we provide elementary proof of Lemma~\ref{rso:lem:nSG_mgf_bounds} and Lemma~\ref{rso:lem:norm_sg_dot}. The proofs follow from standard arguments that can be found in textbooks such as [8, 6].

\subsubsection{Proof of Lemma~\ref{rso:lem:nSG_mgf_bounds}}
    We follow standard arguments from~\citep{vershynin2018high}. First note that, for any $k>0$, we have
    \[
    \begin{split}
        \expectation[\norm{X}^{k}] &= \int_{t=0}^{+\infty} \probability[\norm{X}^{k} \geq t] \measuredWRT t = \int_{t=0}^{+\infty} \probability[\norm{X} \geq t^{1/k}] \measuredWRT t \\
        &\leq 2 \int_{t=0}^\infty e^{-t^{\frac{2}{k}}/(2\sigma^2)} \measuredWRT t 
        =  k (2\sigma^2)^{\frac{k}{2}} \int_{u=0}^\infty e^{-u} u^{\frac{k}{2} - 1} \measuredWRT u = k (2\sigma^2)^{\frac{k}{2}} \Gamma\left(\frac{k}{2}\right),
    \end{split}
    \]
    where $\Gamma$ denotes the gamma function. Hence, noting that $\Gamma(k) = (k-1)!$, by the monotone convergence theorem,
     \[
    \begin{split}
        \expectation[e^{\lambda \squarednormtwo{X}}] 
        &= 1 + \sum_{k=1}^{\infty} \frac{\lambda^k}{k!} \expectation[\norm{X}^{2k}] 
        \leq 1 + \sum_{k=1}^\infty \frac{\lambda^k}{k!}(2k) (2\sigma^2)^{k} \Gamma(k) \\
        &\leq 1 + 2\sum_{k=1}^\infty \left(2 \lambda \sigma^2 \right)^k
        = \frac{2}{1-2\lambda \sigma^2}-1,
    \end{split}
    \]
    the last equality being valid for any $\lambda\in[0, \frac{1}{2 \sigma^2})$. Since for any $u \in [0, \frac{1}{2}]$, $\frac{1}{1-u} \leq e^{2u}$, we obtain $\mgf{\lambda \squarednormtwo{X}} \leq 2e^{4 \lambda \sigma^2}-1$ for any $\lambda\in[0, \frac{1}{4 \sigma^2}]$. Finally, last inequality follows from $2e^{2u}-1 \leq e^{4u}$, where we chose $u=2\lambda \sigma^2$. \QEDB\looseness=-1

\subsubsection{Proof of Lemma~\ref{rso:lem:norm_sg_dot}}
    For $u=0$, the inequality to prove is trivial. Assume $u \neq 0$.
    From Lemma~\ref{rso:lem:nSG_mgf_bounds} and Cauchy-Schwarz inequality, we have
    \begin{equation}
    \label{eq:inner-sq-bound}
        \expectation\left[e^{\lambda^2 \dotproduct{u}{X}^2}\right] \leq \expectation[e^{\lambda^2 \squarednormtwo{u} \squarednormtwo{X}}] \leq e^{8 \lambda^2 \squarednormtwo{u} \sigma^2},
    \end{equation}
    for all $\lambda \in [0, \frac{1}{2\sqrt{2} \sigma \norm{u}}]$. Thus, for any such $\lambda$, noticing that $e^t \leq t + e^{t^2}$ for $t \in \R$, we obtain $\expectation\left[e^{\lambda \dotproduct{u}{X}}\right] \leq \expectation\left[\lambda \dotproduct{u}{X} + e^{\lambda^2 \dotproduct{u}{X}^2}\right] \leq e^{8\lambda^2 \squarednormtwo{u} \sigma^2}$, where the second inequality follows from \eqref{eq:inner-sq-bound} and the assumption that $\expectation[X]=0$. Moreover, for $\lambda \geq \frac{1}{2 \sqrt{2} \norm{u} \sigma}$, we have by Cauchy Schwarz's inequality and Lemma~\ref{rso:lem:nSG_mgf_bounds} that 
        $\expectation[e^{\lambda \dotproduct{u}{X}}] \leq \expectation\Big[e^{\frac{8\lambda^2 \sigma^2 \squarednormtwo{u}}{2} + \frac{\squarednormtwo{X}}{16 \sigma^2}}\Big] \leq e^{\frac{1}{2}\left(1 + 8 \lambda^2 \sigma^2 \squarednormtwo{u}\right)} \leq e^{8\lambda^2 \squarednormtwo{u} \sigma^2}$,
    where the last inequality is due to $e^{\frac{1+t}{2}}\leq e^t$ for $t\geq 1$. \QEDB

%
%
\subsection{Elementary properties of Convex Risk Measures}

The following lemma is used in the derivation of $\cvar$ and $\evar$ bounds.
\begin{lem}\label{lem:cvar_pythagoras}
    For any \yassine{non-negative} random variable $U \from \Omega \to \R_+$, we have for all $p\in [0, 1)$:
    \[
    \begin{split}
            \yassine{Q_p(U^2)^{\frac{1}{2}}} \yassine{= Q_p(U)},\qquad
            \cvar_p(U^2)^{\frac{1}{2}} \geq \cvar_p(U).
    \end{split}
    \]
\end{lem}
\begin{proof}
    We first show that $Q_p(X^2) = Q_p(X)^2$ for any $p \in (0,1)$. Indeed, for any $0 \leq t < Q_p(U)^2$, we have $\probability[U^2 \leq t] = \probability[U \leq \sqrt{t}] < p$ \saa{which follows from non-negativity of $U$ and} definition of $Q_p(U)$. 
    This implies $t < Q_p(U^2)$\sa{; thus, $Q_p(U)^2\leq Q_p(U^2)$.}
    Conversely, we note that $ p \leq \probability[ U \leq Q_p(U)] = \probability[ U^2 \leq Q_p(U)^2]$, \sa{which implies $Q_p(U)^2\geq Q_p(U^2)$;} hence, $Q_p(X^2) = Q_p(X)^2$. 
    \sa{Using this result,} 
    %
    \[
    \begin{split}
        \cvar_p(U^2) = \left(\frac{1}{1-p} \int_{p'=p}^1 Q_{p'}(U^2) \mathrm{d}p' \right) 
        &= \expectation_{p'\sim \mathcal{U}[p,1]} [Q_{\sa{p'}}(U)^2] \\ 
        &\geq \expectation_{p'\sim \mathcal{U}[p,1]} [Q_{p'}(U)]^2 = \cvar_p(U)^2,
    \end{split}
    \]
    where $\mathcal{U}[p,1]$ denotes the uniform distribution on $[p,1]$, and \sa{the} last inequality follows from the identity $\expectation[X^2] = \expectation[X]^2 + \expectation[(X - \expectation[X])^2]$.
\end{proof} 

%% file: A2_sapd_intermediate.tex
\saa{To start with, for the sake of completeness,} we cite two results from~\citep{zhang2021robust}. 
The first lemma is used to derive the almost sure bound \saa{result of Proposition~\ref{rso:prop:almost_sure_bound}, which is provided below in Appendix~\ref{sec:as-dom}, while the second lemma is used for deriving the convex inequalities provided in Appendix~\ref{sec:cvx-ineq}.}
\begin{lem}[See~{\citep[Lemma~1]{zhang2021robust}}]\label{rso:lem:descent_lemma}
    The iterates $(x_k, y_k)$ of \sapdname~satisfy
    \[
        \begin{aligned}
            \mathcal{L}\left(x_{k+1}, \opty\right)-\mathcal{L}\left(\optx, y_{k+1}\right) \leq & -\left\langle q_{k+1}, y_{k+1}-\sa{\opty}\right\rangle+\theta\left\langle q_k, y_k-\sa{\opty}\right\rangle + \Lambda_k - \Sigma_{k+1}+\Gamma_{k+1} \\
                & + \dotproduct{\noisegradx{k}}{\sa{\optx-x_{k+1}}} + \dotproduct{(1+\theta)\noisegrady{k} - \theta \noisegrady{\sa{k-1}}}{y_{k+1} - \opty},
        \end{aligned}
    \]
    \saa{for all $k\geq 0$,} where 
    \[ {\small
        \begin{aligned}
            q_k \defineq & \gradphiy\left(x_k, y_k\right)- \gradphiy\left(x_{k-1}, y_{k-1}\right), \quad
            \Lambda_k \defineq \frac{1}{2 \tau}\left\|\optx-x_k\right\|^2+\frac{1}{2 \sigma}\left\|\opty-y_k\right\|^2, \\
            \Sigma_{k+1} \defineq & \left(\frac{1}{2 \tau}+\frac{\mu_x}{2}\right)\left\|\optx-x_{k+1}\right\|^2+\left(\frac{1}{2 \sigma}+\frac{\mu_y}{2}\right)\left\|\opty-y_{k+1}\right\|^2, \\
            \Gamma_{k+1} \defineq & \left(\frac{\Lxx}{2}-\frac{1}{2 \tau}\right)\left\|x_{k+1}-x_k\right\|^2-\frac{1}{2 \sigma}\left\|y_{k+1}-y_k\right\|^2 +\theta \Lyx\left\|x_k-x_{k-1}\right\|\left\|y_{k+1}-y_k\right\| \\ 
            &+\theta \Lyy\left\|y_k-y_{k-1}\right\|\left\|y_{k+1}-y_k\right\|.
        \end{aligned}
     }
   \]
\end{lem}
\begin{lem}[See~{\citep[Lemma~3]{zhang2021robust}}]\label{rso:lem:sapd_lemma_A3}
    \saa{Let $(x_n, y_n)_{n\geq 0}$ denote the \sapdname~iterate sequence. 
    Then, the following inequalities hold for all $n \in \N$,} 
    \[
    {\small
    \begin{aligned}[t]        
        \norm{\noiselessy_{n+1} - \sa{y_{n+1}}}
            &\leq \frac{\sigma}{1+\sigma \mu_y} \left((1+\theta) \norm{\noisegrady{n}} + \theta \norm{\noisegrady{n-1}} \right),  \\
        \norm{\doublenoiselessx_{n+1} - x_{n+1}} 
            &\leq \frac{\tau}{1+\tau \mu_x} \left(\norm{\noisegradx{n}} + \Lxy \frac{\sigma}{1+\sigma \mu_y} \left((1+\theta) \norm{\noisegrady{n}} + \theta \norm{\noisegrady{n-1}} \right)\right), \\
        \norm{\doublenoiselessy_{n+1} - y_{n+1}} 
            &\leq \frac{\sigma}{1+\sigma \mu_y} 
                \begin{aligned}[t]
                     &\left( \frac{\tau(1+\theta) \Lyx}{1+\tau \mu_x} \norm{\noisegradx{n-1}}\right. 
                        + (1+\theta)\norm{\noisegrady{n}}  \\
                        &\quad + \left(\theta + (1+\theta) \left(\frac{1+\sigma(1+\theta)\Lyy}{1+\sigma \mu_y} + \frac{\tau \sigma (1+\theta) \Lyx \Lxy}{(1+\tau \mu_x)(1+\sigma \mu_y)}\right)\right) \norm{\noisegrady{n-1}} \\
                        & \quad+  \left. \theta \left(\frac{1+\sigma(1+\theta)\Lyy}{1+\sigma \mu_y} + \frac{\tau \sigma (1+\theta) \Lyx \Lxy}{(1+\tau \mu_x)(1+\sigma \mu_y)}\right) \norm{\noisegrady{n-2}} \right).\\
                \end{aligned}
    \end{aligned}
    }
    \]
\end{lem}

\subsection{Proof of Proposition~\ref{rso:prop:almost_sure_bound} (Almost sure domination of \sapdname~iterates)}\label{sec:as-dom}

Letting \sa{$\bar x_n \defineq K_n(\rho)^{-1} \sum_{k=0}^{n-1} \rho^{-k} x_{k+1}$, and $\bar y_n \defineq K_n(\rho)^{-1} \sum_{k=0}^{n-1} \rho^{-k} y_{k+1}$}, with $K_n(\rho) \defineq \sa{\sum_{k=0}^{n-1}} \rho^{-k} = \frac{1}{\rho^{n-1}}\times\frac{1-\rho^n}{1-\rho}$, by Jensen's inequality , we have for all $\rho \in(0,1]$,
\[
        K_n(\rho)\left(\mathcal{L}\left(\bar{x}_n, \sa{\opty}\right)-\mathcal{L}\left(\sa{\optx}, \bar{y}_n\right)\right) \leq \sum_{k=0}^{n-1} \rho^{-k}\left(\mathcal{L}\left(x_{k+1}, \sa{\opty}\right)-\mathcal{L}\left(\sa{\optx}, y_{k+1}\right)\right).
\]
Hence, in view of Lemma~\ref{rso:lem:descent_lemma}, 
\begin{equation}\label{rso:eq:as_jensen_application}
        \begin{aligned}
            & K_n(\rho)\left(\mathcal{L}\left(\bar{x}_n, \sa{\opty}\right)-\mathcal{L}\left(\sa{\optx}, \bar{y}_n\right)\right) \\
            & \leq \sum_{k=0}^{n-1} \rho^{-k}\sa{\Big(}-\left\langle q_{k+1}, y_{k+1}-\sa{\opty}\right\rangle+\theta\left\langle q_k, y_k-\sa{\opty}\right\rangle+\Lambda_k-\Sigma_{k+1}+\Gamma_{k+1} \\
            & \quad + \dotproduct{\noisegradx{k}}{\optx - x_{k+1}} + \dotproduct{(1+\theta)\noisegrady{k} - \theta \noisegrady{k-1}}{y_{k+1} - \opty}\sa{\Big)},
    \end{aligned}
\end{equation}
where $q_k \defineq \gradphiy(x_k, y_k) - \gradphiy(x_{k-1}, y_{k-1})$.
By Cauchy-Schwarz inequality, observe that 
\[
        \left|\left\langle q_{k+1}, y_{k+1}-\sa{\opty}\right\rangle\right| \leq S_{k+1} \defineq \Lyx\left\|x_{k+1}-x_k\right\|\left\|y_{k+1}-y\right\|+\Lyy\left\|y_{k+1}-y_k\right\|\left\|y_{k+1}-\sa{\opty}\right\|, \quad \forall k \geq 0.
\]
Hence, \sa{using $q_0=\mathbf{0}$ due to our initialization of $(x_{-1},y_{-1})=(x_0,y_0)$, we have}
\[
{\small
        \begin{aligned}
            & \sum_{k=0}^{n-1} \rho^{-k}\left(-\left\langle q_{k+1}, y_{k+1}-\sa{\opty}\right\rangle+\theta\left\langle q_k, y_k-\sa{\opty}\right\rangle\right) =\sum_{k=0}^{n-2} \rho^{-k}\left(\frac{\theta}{\rho}-1\right)\left\langle q_{k+1}, y_{k+1}-\sa{\opty}\right\rangle-\rho^{-n+1}\left\langle q_n, y_n-\sa{\opty}\right\rangle \\
            & \leq \sum_{k=0}^{n-2} \rho^{-k}\left|1-\frac{\theta}{\rho}\right| S_{k+1}+\rho^{-n+1} S_n \sa{\leq} \sum_{k=0}^{n-1} \rho^{-k}\left|1-\frac{\theta}{\rho}\right| S_{k+1}+\rho^{-n+1} \frac{\theta}{\rho} S_n.
        \end{aligned}
}        
\]
\yassine{From~\eqref{rso:eq:as_jensen_application}, it follows that}
\[
{
    \begin{aligned}
    & K_n(\rho)\left(\mathcal{L}\left(\bar{x}_n, \sa{\opty}\right)-\mathcal{L}\left(\sa{\optx}, \bar{y}_n\right)\right)+\rho^{-n+1} \distancegap{n}\\
    &    \leq U_n 
            + \sum_{k=0}^{n-1} \rho^{-k}\sa{\Big(}\dotproduct{\noisegradx{k}}{\optx - x_{k+1}} + \dotproduct{(1+\theta)\noisegrady{k} - \theta \noisegrady{k-1}}{y_{k+1} - \opty}\sa{\Big)},
    \end{aligned}
}
\]
where
%
    $U_n \triangleq \sum_{k=0}^{n-1} \rho^{-k}\left(\Gamma_{k+1}+\Lambda_k-\Sigma_{k+1}+\left|1-\frac{\theta}{\rho}\right| S_{k+1}\right)-\rho^{-n+1}\left(-\distancegap{n} -\frac{\theta}{\rho} S_n\right)$.
%
Now, observe that for all $n\geq 1$,
\[
    \begin{aligned}
        U_n & =\frac{1}{2} \sum_{k=0}^{n-1} \rho^{-k}\left(\xi_k^{\top} A \xi_k-\xi_{k+1}^{\top} B \xi_{k+1}\right)-\rho^{-n+1}\left(-\distancegap{n}-\frac{\theta}{\rho} S_n\right) \\
        & =\frac{1}{2} \xi_0^{\top} A \xi_0-\frac{1}{2} \sum_{k=1}^{n-1} \rho^{-k+1}\left[\xi_k^{\top}\left(B-\frac{1}{\rho} A\right) \xi_k\right]-\rho^{-n+1}\left(\frac{1}{2} \xi_n^{\top} B \xi_n-\distancegap{n}-\frac{\theta}{\rho} S_n\right),
    \end{aligned}
\]
where $A, B \in \mathbb{R}^{5 \times 5}$ and $\xi_k \in \mathbb{R}^5$ are defined for $k\geq 0$ as
{\footnotesize
\[
    \begin{aligned}
        & A \triangleq\left(\begin{array}{ccccc}
        \frac{1}{\tau} & 0 & 0 & 0 & 0 \\
        0 & \frac{1}{\sigma} & 0 & 0 & 0 \\
        0 & 0 & 0 & 0 & \theta \Lyx \\
        0 & 0 & 0 & 0 & \theta \Lyy \\
        0 & 0 & \theta \Lyx & \theta \Lyy & -\alpha
        \end{array}\right),~ 
        B \triangleq\left(\begin{array}{ccccc}
        \frac{1}{\tau}+\mu_x & 0 & 0 & 0 & 0 \\
        0 & \frac{1}{\sigma}+\mu_y & -\left|1-\frac{\theta}{\rho}\right| \Lyx & -\left|1-\frac{\theta}{\rho}\right| \Lyy & 0 \\
        0 & -\left|1-\frac{\theta}{\rho}\right| \Lyx & \frac{1}{\tau}-\Lxx & 0 & 0 \\
        0 & -\left|1-\frac{\theta}{\rho}\right| \Lyy & 0 & \frac{1}{\sigma}-\alpha & 0 \\
        0 & 0 & 0 & 0 & 0
        \end{array}\right),
    \end{aligned}
\]
and {$\xi_k \triangleq \left(\begin{array}{ccccc}
        \left\|x_k-\sa{\optx}\right\|, &
        \left\|y_k-\sa{\opty}\right\|. &
        \left\|x_k-x_{k-1}\right\|, &
        \left\|y_k-y_{k-1}\right\|, &
        \left\|y_{k+1}-y_k\right\|
        \end{array}\right)^\top \in\reals^5
        $}.
}
By~\citep[\sa{Lemma~5}]{zhang2021robust}, the matrix inequality condition~\eqref{rso:eq:matrix_inequality_params} is equivalent to having $B - \rho^{-1}A \succeq 0$. 
In this case, we almost surely have
\begin{align}
     U_n \leq \frac{1}{2} \xi_0^{\top} A \xi_0-\rho^{-n+1}\left(\frac{1}{2} \xi_n^{\top} B \xi_n- \distancegap{n} -\frac{\theta}{\rho} S_n\right).\label{eq:Un-bound}
\end{align}
Finally, denoting 
{\small
\[
        G^{\prime \prime} 
            \defineq \left(\begin{array}{ccc}
                    \frac{1}{\sigma}\left(1-\frac{1}{\rho}\right)+\mu_y+\frac{\alpha}{\rho} & \left(-\left|1-\frac{\theta}{\rho}\right|-\frac{\theta}{\rho}\right) \Lyx & \left(-\left|1-\frac{\theta}{\rho}\right|-\frac{\theta}{\rho}\right) \Lyy \\
                    \left(-\left|1-\frac{\theta}{\rho}\right|-\frac{\theta}{\rho}\right) \Lyx & \frac{1}{\tau}-\Lxx & 0 \\
                    \left(-\left|1-\frac{\theta}{\rho}\right|-\frac{\theta}{\rho}\right) \Lyy & 0 & \frac{1}{\sigma}-\alpha
                \end{array}\right),
\]}%
we have \sa{$G^{\prime \prime}\succeq 0$ in view of~\citep[\sa{Lemma~6}]{zhang2021robust}; thus,}
{\footnotesize
\begin{eqnarray*}
    \frac{1}{2} \xi_n^{\top} B \xi_n-\frac{\theta}{\rho} S_n 
    &= &
    \frac{1}{2 \rho \tau}\left\|x_n-x\right\|^2+\frac{1}{2}\left(\frac{1}{\rho \sigma}-\frac{\alpha}{\rho}\right)\left\|y_n-y\right\|^2+\frac{1}{2} \xi_n^{\top}\left(\begin{array}{ccc}
    \frac{1}{\tau}\left(1-\frac{1}{\rho}\right)+\mu_x & \mathbf{0}_{1 \times 3} & 0 \\
    \mathbf{0}_{3 \times 1} & G^{\prime \prime} & \mathbf{0}_{3 \times 1} \\
    0 & \mathbf{0}_{1 \times 3} & 0
    \end{array}\right) \xi_n \\
    &\geq & \frac{1}{2 \rho \tau}\left\|x_n-x\right\|^2+\frac{1}{2 \rho \sigma}(1-\alpha \sigma)\left\|y_n-y\right\|^2= \distancegap{n}.
\end{eqnarray*}
}
\sa{Therefore, using~\eqref{eq:Un-bound}, we can conclude that} $U_n \leq \frac{1}{2} \xi_0^{\top} A \xi_0 \sa{\leq} \frac{1}{2\tau} \squarednormtwo{\sa{x_0} - \optx} + \frac{1}{2\sigma} \squarednormtwo{\sa{y_0} - \opty} = \startingbias$. 
\sa{Finally,} by non-negativity of $\mathcal{L}\left(\bar{x}_n, \sa{\opty}\right)-\mathcal{L}\left(\sa{\optx}, \bar{y}_n\right)$, we obtain~\eqref{rso:eq:almost_sure_upper_bound}.\QEDB

\subsection{Convex inequalities}\label{sec:cvx-ineq}
\subsubsection{Proof of Lemma~\ref{rso:prop:barbers_lemma}}\label{sec:proof_barbers_lemma}
We first start with a technical result we will use in the proof of  Lemma~\ref{rso:prop:barbers_lemma}.
\begin{lem}\label{rso:lem:upperbound_hathaty}
    For any \sa{$n \geq 1$}, 
    \[
        \norm{\doublenoiselessy_{n+1} - \opty} 
            \leq {\norm{A_0} (\distancegap{n} + \distancegap{n-1})^{1/2} }
                + \frac{1}{1+\sigma \mu_y} \left(\left(1 + \sigma(1+\theta)\Lyy \right) \norm{y_n - \noiselessy_n}
                + \sigma(1+\theta)\Lyx \norm{x_n - \doublenoiselessx_n}\right),
    \]
    where $\doublenoiselessy_{n+1},\noiselessy_n,\doublenoiselessx_n$ are defined in \eqref{rso:eq:def_noiseless_variables}, and $A_0 \in \mathbb{R}^4$ 
    is defined as
    {\small
    \[
        A_0 \defineq
        \sa{\tfrac{1}{1+\sigma \muy}}
            \begin{bmatrix}
                \sqrt{2\rho \tau}\sigma(1+\theta) \Lyx \\
                \frac{\sqrt{2\rho \sigma}}{\sqrt{1-\alpha \sigma}}(1 + \sigma(1+\theta)\Lyy) \\
                {\sqrt{2\rho \tau}\cdot \sigma \theta \Lyx} \\
                {\frac{\sqrt{2\rho \sigma}}{\sqrt{(1-\alpha \sigma)}}\cdot\sigma \theta \Lyy} 
            \end{bmatrix} \in \R^4.
    \]
    }
\end{lem}
\begin{proof}
    {
    Since $(\optx, \opty)$ is a solution of~\eqref{rso:eq:main_problem}, $\optx$ and $\opty$ are fixed points of two deterministic proximal gradient maps, i.e.,
\begin{equation}\label{rso:eq:fixed_point_eq}
        \optx = \prox{\tau f}\left(\optx - \tau \gradphix(\optx, \opty)\right), \quad \opty = \prox{\sigma g}\left(\opty + \sigma \gradphiy(\optx, \opty)\right).
    \end{equation}
    Thus, by the contraction properties of the prox for strongly convex functions, and convexity of the squared norm, we have}
    \[
        \begin{split}
            \norm{\doublenoiselessy_{n+1} - \opty} 
                \leq \frac{1}{1+\sigma \mu_y} \left\| \noiselessy_n + \sigma (1+\theta) \gradphiy(\doublenoiselessx_{n}, \noiselessy_{n}) - \sigma \theta \gradphiy(x_{n-1}, \sa{y_{n-1}}) - \opty - \sigma \gradphiy(\optx, \opty) \right\|.
        \end{split}
    \]
    By the triangular inequality and smoothness assumptions on $\gradphiy$, we deduce
    \[
        \scalemath{0.8}{
        \begin{alignedat}[t]{1}
            \norm{\doublenoiselessy_{n+1} - \opty} 
                &\leq \frac{1}{1+\sigma \mu_y}
                    \left(
                        (1+\sigma(1+\theta)\Lyy) \norm{\noiselessy_n - \opty} 
                        + \sigma (1+\theta) \Lyx \norm{\doublenoiselessx_n - \optx} 
                        + \sigma \theta \Lyx \norm{x_{n-1} - \optx} 
                        + \sigma \theta \Lyy \norm{y_{n-1} - \opty} 
                    \right) \\
                &\leq \begin{alignedat}[t]{1}
                     & \frac{1}{1+\sigma \mu_y}
                    \left(
                        (1+\sigma(1+\theta)\Lyy) \norm{y_n - \opty} 
                        + \sigma (1+\theta) \Lyx \norm{x_n - \optx} 
                        + \sigma \theta \Lyx \norm{x_{n-1} - \optx} 
                        + \sigma \theta \Lyy \norm{y_{n-1} - \opty} 
                    \right) \\
                    & + \frac{1}{1+\sigma \mu_y} 
                        \left((1+\sigma(1+\theta)\Lyy) \norm{\noiselessy_n - y_n} 
                        + \sigma (1+\theta) \Lyx \norm{\doublenoiselessx_n - x_n} \right). \\
                \end{alignedat}
        \end{alignedat}
        }
    \]
    The statement finally follows from Cauchy-Schwarz inequality.
\end{proof}

Now we are ready to prove Lemma~\ref{rso:prop:barbers_lemma}. 
By Young's inequality, for any $\gamma_x, \gamma_y > 0$,
\[
    \begin{split}
        {\frac{-1}{1+\rho}} 
            & {\left(
                \varidx{P}{n}{1} 
                + \dotproduct{\noisegradx{n}}{\doublenoiselessx_{n+1} - x_{n+1}} 
                + (1+\theta)\; \dotproduct{\noisegrady{n}}{y_{n+1} - \noiselessy_{n+1}}
                \right)}\\
            &= \frac{1}{1+\rho} 
                    \begin{aligned}[t]
                        & \left(\dotproduct{\noisegradx{n}}{\doublenoiselessx_{n+1} - \optx} 
                            + (1+\theta)\; \dotproduct{\noisegrady{n}}{\opty - \noiselessy_{n+1}} \right. \\
                        & \left. - \dotproduct{\noisegradx{n}}{\doublenoiselessx_{n+1} - x_{n+1}} 
                            - (1+\theta)\; \dotproduct{\noisegrady{n}}{y_{n+1} - \noiselessy_{n+1}} \right) \\
                    \end{aligned} \\
            &= \frac{1}{1+\rho} \left(\dotproduct{\noisegradx{n}}{x_{n+1} - \optx} 
                + (1+\theta) \dotproduct{\noisegrady{n}}{\opty - y_{n+1}} \right) \\        
            &\leq \frac{\gamma_x}{2(1+\rho)} \squarednormtwo{\noisegradx{n}} 
                    + \frac{1}{2 \gamma_x (1+\rho)} \squarednormtwo{\optx - x_{n+1}} 
                    + \frac{(1+\theta)\gamma_y}{2(1+\rho)}  \squarednormtwo{\noisegrady{n}} 
                    + \frac{(1+\theta)}{2 \gamma_y (1+\rho)} \squarednormtwo{y_{n+1} - \opty}.
    \end{split}
\]
Setting $\gamma_x \defineq 8 \tau$ and $\gamma_y \defineq \frac{8\sigma (1+\theta)}{1-\alpha \sigma}$, we ensure that 
%
%
\begin{equation}\label{rso:eq:bound_Pn_1_barber_lemma}
    \begin{split}
        \frac{-1}{1+\rho} \left(
                \varidx{P}{n}{1} 
                + \dotproduct{\noisegradx{n}}{\doublenoiselessx_{n+1} - x_{n+1}} \right.
                &+ \left. (1+\theta)\; \dotproduct{\noisegrady{n}}{y_{n+1} - \noiselessy_{n+1}}
                \right) \\
            &\leq \frac{\rho}{8(1+\rho)} \distancegap{n+1} + \frac{4\tau}{1+\rho} \squarednormtwo{\noisegradx{n}} 
                            + \frac{4\sigma(1+\theta)^2}{(1+\rho)(1-\alpha \sigma)} \squarednormtwo{\noisegrady{n}}, \\            
    \end{split}
\end{equation}
where $\cE_n=\cD_n/\rho$ and $\cD_n$ is defined in \eqref{def-metric}.
\sa{Moreover, we also have} $\frac{-\rho}{1+\rho} \varidx{P}{n-1}{2} + \frac{\theta}{1+\rho} \dotproduct{\noisegrady{n-1}}{y_{n+1} - \doublenoiselessy_{n+1}} = \frac{\theta}{1+\rho} \dotproduct{\noisegrady{n-1}}{y_{n+1} - \opty} \leq \frac{\theta}{1+\rho} \left(\frac{\gamma'_y}{2}\norm{\noisegrady{n-1}}^2+\frac{1}{2\gamma'_y}\norm{\opty - y_{n+1}}^2\right)$
for any $\gamma'_y > 0$. 
Hence, \sa{setting $\gamma'_y=\frac{8\theta\sigma}{1-\alpha\sigma}$ leads to}
\begin{equation}\label{rso:eq:bound_Pn_2_minus_barber_lemma}
    \begin{split}
        {\frac{-\rho}{1+\rho}} \varidx{P}{n-1}{2} + \frac{\theta}{1+\rho} \dotproduct{\noisegrady{n-1}}{y_{n+1} - \doublenoiselessy_{n+1}} 
            \leq \frac{\rho}{8(1+\rho)} \distancegap{n+1} 
                + \frac{4 \sigma\theta^2}{(1+\rho)(1-\alpha \sigma)} \squarednormtwo{\noisegrady{n-1}}.
    \end{split}
\end{equation}
Finally, observe that for \sa{any} $\gamma > 0$,
\[
\begin{split}
    - \varidx{P}{n}{2} 
        &\leq \frac{\theta \gamma}{2\rho} {\squarednormtwo{\noisegrady{n}}} + \frac{\theta}{2\gamma \rho} \squarednormtwo{\doublenoiselessy_{n+2} - \opty} \\ 
        &\leq \frac{\theta \gamma}{2\rho} {\squarednormtwo{\noisegrady{n}}} 
            + \frac{\theta}{2\gamma \rho} \bigg(3 \squarednormtwo{A_0} (\distancegap{n+1} + \distancegap{n}) 
        \begin{aligned}[t]
             &+ 3 \Big(\frac{1 + \sigma(1+\theta)\Lyy}{1+\sigma \mu_y}\Big)^2 
             \squarednormtwo{\noiselessy_{n+1} - y_{n+1}} \\
             &+ 3 \Big(\frac{\sigma (1+\theta) \Lyx}{1+\sigma \mu_y}\Big)^2 \squarednormtwo{\doublenoiselessx_{n+1} - x_{n+1}} \bigg),
        \end{aligned}\\
\end{split}
\]
where \sa{the} last inequality \sa{follows} from Lemma~\ref{rso:lem:upperbound_hathaty} and the 
simple inequality $(a+b+c)^2\leq 3a^2+3b^2+3c^2$ for any $a,b,c\in\reals$.
Setting $\gamma \defineq \frac{6 \theta \squarednormtwo{A_0}(1+\rho)}{\sa{\rho^2}}$ ensures that 
\begin{equation}\label{rso:eq:bound_Pn_2_barber_lemma}
    - \varidx{P}{n}{2} 
            \leq \frac{\rho}{4(1+\rho)}(\distancegap{n+1} + \distancegap{n}) 
                \begin{aligned}[t]
                    &+ \frac{3 \sa{\theta^2}\squarednormtwo{A_0}(1+\rho)}{\sa{\rho^3}} \squarednormtwo{\noisegrady{n}} \\
                    &+ \frac{\rho}{4 \squarednormtwo{A_0} (1+\rho)}  \left(\frac{1 + \sigma(1+\theta)\Lyy}{1+\sigma \mu_y}\right)^2 
                 \squarednormtwo{\noiselessy_{n+1} - y_{n+1}} \\     
                    &+ \frac{\rho}{4 \squarednormtwo{A_0} (1+\rho)} \left(\frac{\sigma (1+\theta) \Lyx}{1+\sigma \mu_y}\right)^2 \squarednormtwo{\doublenoiselessx_{n+1} - x_{n+1}}.
                \end{aligned}
\end{equation}
Hence, {using the trivial {upper bound} $\cE_{n+1}\leq (\cE_{n+1} + \cE_n)$} and {combining the bounds~\cref{rso:eq:bound_Pn_1_barber_lemma,rso:eq:bound_Pn_2_minus_barber_lemma,rso:eq:bound_Pn_2_barber_lemma} we obtained above, we get}
{\small
\[
\begin{split}
    \frac{-1}{1+\rho}
        &  \varidx{P}{n}{1} 
            - \varidx{P}{n}{2}
            - \frac{\rho}{1+\rho} \varidx{P}{n-1}{2} 
            - \frac{1}{1+\rho} Q_n \\
        &\leq 
            \begin{aligned}[t]
                 & \frac{\rho}{2(1+\rho)} (\distancegap{n+1} + \distancegap{n})  \\
                 &  + \frac{4\tau}{1+\rho} \squarednormtwo{\noisegradx{n}} 
                    + \frac{4\sigma(1+\theta)^2}{(1+\rho)(1-\alpha \sigma)} \squarednormtwo{\noisegrady{n}} 
                    + \frac{4\sigma\theta^2}{(1+\rho)(1-\alpha \sigma)} \squarednormtwo{\noisegrady{n-1}} 
                    + \frac{3 \theta^2 (1+\rho) \squarednormtwo{A_0}}{\rho^3} {\squarednormtwo{\noisegrady{n}}}  \\
                 &  + \frac{\rho}{4 \squarednormtwo{A_0}(1+\rho)} \left(\frac{1 + \sigma(1+\theta)\Lyy}{1+\sigma \mu_y}\right)^2 
             \squarednormtwo{\noiselessy_{n+1} - y_{n+1}} \\
             &
             + \frac{\rho}{4 \squarednormtwo{A_0} (1+\rho)} \left(\frac{\sigma (1+\theta) \Lyx}{1+\sigma \mu_y}\right)^2 \squarednormtwo{\doublenoiselessx_{n+1} - x_{n+1}}.
            \end{aligned}
\end{split}
\]}
Let us now introduce $\zeta_k \defineq \Big[\norm{\noisegradx{k}},\; \norm{\noisegrady{k}},\; \rho^{1/2} \norm{\noisegrady{k-1}}\Big]^\top\sa{\in\reals^3}$ \sa{for $k\geq 0$; 
then, by similar computations the following bounds follow from} Lemma~\ref{rso:lem:sapd_lemma_A3}:
\[
\begin{split}
    \squarednormtwo{\noiselessy_{n+1} - y_{n+1}}
        &\leq 
            \zeta_n^\top \diag \left[0, \frac{2\sigma^2 (1+\theta)^2}{(1+\sigma \muy)^2}, \frac{2\sigma^2 {\theta^2 \rho^{-1}}}{(1+\sigma \muy)^2}\right] \zeta_n \\
    \squarednormtwo{\doublenoiselessx_{n+1} - x_{n+1}} 
        & \leq 
            \zeta_n^\top \diag \left[\frac{3 \tau^2}{(1+\tau \mux)^2}, \frac{3 \tau^2 \sigma^2 (1+\theta)^2 \Lxy^2}{(1+\tau \mux)^2 (1+\sigma \muy)^2}, \frac{3 \tau^2 \sigma^2 \theta^2 \rho^{-1} \Lxy^2}{(1+\tau \mux)^2 (1+\sigma \muy)^2}\right]
         \zeta_n,
\end{split}
\]
and we deduce that 
\[
    \frac{-1}{1+\rho} 
        \varidx{P}{n}{1} 
            - \varidx{P}{n}{2}
            - \frac{\rho}{1+\rho} \varidx{P}{n-1}{2} 
            - \frac{1}{1+\rho} Q_n \leq \frac{\rho}{2(1+\rho)} (\distancegap{n+1} + \distancegap{n}) +  \zeta_n^\top \diag \left[B^x, B^y, B_{-1}^{y} \right] \zeta_n, \\
\]
\sa{where $B^x, B^y, B_{-1}^{y}$ are {constants specified} in Table~\ref{rso:table:constants_C} of Appendix \ref{rso:sec:constants}.}

We \ylsecond{now} treat the sum $\sum_{k=0}^n \rho^{n-k} \ylsecond{Q_k}$. 
Observe first that for all $n \in \N$, Lemma~\ref{rso:lem:sapd_lemma_A3}, 
{\small
\[
\begin{split}
    Q_n 
        &\leq 
            \norm{\noisegradx{n}} \; \norm{\doublenoiselessx_{n+1}-x_{n+1}} 
            + (1+\theta) \norm{\noisegrady{n}} \; \norm{y_{n+1}-\noiselessy_{n+1}} 
            + \theta \; \norm{\noisegrady{n-1}} \; \norm{y_{n+1} \ylsecond{-}\doublenoiselessy_{n+1}} \\
        &\leq   \begin{aligned}[t]
                    & \norm{\noisegradx{n}} \; 
                        \left[ 
                            \frac{\tau}{1+\tau \mux}
                                \left(
                                    \left\|\noisegradx{n}\right\|
                                    + \Lxy \frac{\sigma}{1+\sigma \mu_y}
                                        \left(
                                            (1+\theta) \left\|\noisegrady{n}\right\|
                                            + \theta\left\|\noisegrady{n-1}\right\|
                                        \right)
                                \right) 
                            \right] \\
                    & + (1+\theta) \norm{\noisegrady{n}} \; 
                        {\left(
                            \frac{\sigma}{1+\sigma \muy} 
                                \left(
                                    (1+\theta) \norm{\noisegrady{n}} + \theta \norm{\noisegrady{n-1}}
                                \right)
                            \right)} \\
                    & + \theta \norm{\noisegrady{n-1}} 
                        \frac{\sigma}{1+\sigma \mu_y} 
                            \begin{aligned}[t]
                                 &\left( \frac{\tau(1+\theta) \Lyx}{1+\tau \mu_x} \norm{\noisegradx{n-1}}\right. \\
                                    &+ (1+\theta)\norm{\noisegrady{n}}  \\
                                    &+ \left(\theta + (1+\theta) \left(\frac{1+\sigma(1+\theta)\Lyy}{1+\sigma \mu_y} + \frac{\tau \sigma (1+\theta) \Lyx \Lxy}{(1+\tau \mu_x)(1+\sigma \mu_y)}\right)\right) \norm{\noisegrady{n-1}} \\
                                    & +  \left. \theta \left(\frac{1+\sigma(1+\theta)\Lyy}{1+\sigma \mu_y} + \frac{\tau \sigma (1+\theta) \Lyx \Lxy}{(1+\tau \mu_x)(1+\sigma \mu_y)}\right) \norm{\noisegrady{n-2}} \right),\\
                            \end{aligned}                        
                \end{aligned} \\
\end{split}
\]}%
{which, after organizing the terms and using $ab \leq a^2/2 + b^2/2$ for any scalars $a, b$, becomes
{\small
\[
\begin{split}
    Q_n &\leq   \begin{aligned}[t]
                    & \frac{\tau}{1+\tau \mux}\squarednormtwo{\noisegradx{n}} 
                        + \frac{\sigma(1+\theta)^2}{1+\sigma \mu_y} \squarednormtwo{\noisegrady{n}} \\
                    & + \frac{\sigma \theta}{1+\sigma \mu_{\mathbf{y}}}\left(\theta+(1+\theta)\left(\frac{1+\sigma(1+\theta) \mathrm{L}_{\mathrm{yy}}}{1+\sigma \mu_y}+\frac{\tau \sigma(1+\theta) \mathrm{L}_{\mathrm{yx}} \mathrm{L}_{\mathrm{xy}}}{\left(1+\tau \mu_x\right)\left(1+\sigma \mu_y\right)}\right)\right) \squarednormtwo{\noisegrady{n-1}} \\
                    & + \frac{\sigma \theta}{\ylthird{2}(1+\sigma \mu_y)} \quad \frac{\tau(1+\theta) \Lyx}{1+\tau \mu x} \squarednormtwo{\noisegrady{n-1}} 
                        + \frac{\sigma \theta}{\ylthird{2}(1+\sigma \mu_y)} \frac{\tau(1+\theta) \Lyx}{1+\tau \mu_x} \squarednormtwo{\noisegradx{n-1}} \\
                    & + \frac{\sigma \theta}{1+\sigma \mu y}(1+\theta) \squarednormtwo{\noisegrady{n}} + \frac{\sigma \theta}{1+\sigma \muy} (1+\theta) \squarednormtwo{\noisegrady{n-1}} \\
                    & + \frac{\sigma \theta}{1+\sigma \mu_y} \ylthird{\frac{\theta}{2}}\left(\frac{1+\sigma(1+\theta) \mathrm{L}_{\mathrm{yy}}}{1+\sigma \mu_y}+\frac{\tau \sigma(1+\theta) \mathrm{L}_{\mathrm{yx}} \mathrm{L}_{\mathrm{xy}}}{\left(1+\tau \mu_x\right)\left(1+\sigma \mu_y\right)}\right) \squarednormtwo{\noisegrady{n-1}}  \\
                    & \quad + \frac{\sigma \theta}{1+\sigma \mu_y} \ylthird{\frac{\theta}{2}} \left(\frac{1+\sigma(1+\theta) \mathrm{L}_{\mathrm{yy}}}{1+\sigma \mu_y}+\frac{\tau \sigma(1+\theta) \mathrm{L}_{\mathrm{yx}} \mathrm{L}_{\mathrm{xy}}}{\left(1+\tau \mu_x\right)\left(1+\sigma \mu_y\right)}\right) \squarednormtwo{\noisegrady{n-2}} \\
                    & + \frac{\tau \sigma(1+\theta) \Lxy}{\ylthird{2}\left(1+\tau \mu_x\right)\left(1+\tau \mu_y\right)} \squarednormtwo{\noisegradx{n}}
                        + \frac{\tau \sigma(1+\theta) \Lxy}{\ylthird{2} \left(1+\tau \mu_x\right)\left(1+\tau \mu_y\right)} \squarednormtwo{\noisegrady{n}} \\
                    & + \frac{\tau \sigma \theta \Lxy}{\ylthird{2} \left(1+\tau \mu_x\right)\left(1+\tau \mu_y\right)} \squarednormtwo{\noisegradx{n}}
                        + \frac{\tau \sigma \theta \Lxy}{\ylthird{2} \left(1+\tau \mu_x\right) \left(1+\tau \mu_y\right)} \squarednormtwo{\noisegrady{n-1}} \\
                \end{aligned} \\
\end{split} 
\]}
}
Thus, we obtain $Q_n \leq C^x\; \squarednormtwo{\noisegradx{n}} \;+\;  C_{-1}^x\; \rho \squarednormtwo{\noisegradx{n-1}} \;+\; C^y\; \squarednormtwo{\noisegrady{n}} \;+\;  C_{-1}^y\; \rho \squarednormtwo{\noisegrady{n-1}} \;+\; C_{-2}^y\; \rho^2 \squarednormtwo{\noisegrady{n-2}}$ for some constants $C^x, C_{-1}^x, C^y, C_{-1}^y, C_{-2}^{y}$ (that are explicitly given in Table~\ref{rso:table:constants_C} of Appendix \ref{rso:sec:constants}). 
Hence, setting $\noisegradx{-1} = \noisegrady{-1} = \noisegradx{-2} = 0$, we obtain,
\[
\begin{split}
    \sum_{k=0}^{n} \rho^{n-k} Q_k
        &- \frac{1}{1+\rho} \varidx{P}{n}{1} 
        - \varidx{P}{n}{2}
        - \frac{\rho}{1+\rho} \varidx{P}{n-1}{2} 
        - \frac{1}{1+\rho} Q_n \\
        &\leq
            \begin{aligned}[t]
                \frac{\rho}{2(\ylsecond{1}+\rho)} (\distancegap{n+1} + \distancegap{n}) 
                    &+ \sum_{k=0}^{n} \rho^{n-k} 
                        \begin{aligned}[t]
                        ( & C^x \squarednormtwo{\noisegradx{k}} 
                          + C_{-1}^x \rho \squarednormtwo{\noisegradx{k-1}} \\
                        & + C^y \squarednormtwo{\noisegrady{k}}
                        + C_{-1}^y \rho \squarednormtwo{\noisegrady{k-1}}
                        + C_{-2}^y \rho^2 \squarednormtwo{\noisegrady{k-2}})\\  
                        \end{aligned} \\        
                    &+ B^x \squarednormtwo{\noisegradx{n}} + B^y \squarednormtwo{\noisegrady{n}} + B_{-1}^y \rho \squarednormtwo{\noisegrady{n-1}};
            \end{aligned} \\
\end{split}
\]
therefore, rearranging the terms together we get
{\small
\[
\begin{split}
    \sum_{k=0}^{n} \rho^{n-k} Q_k
        &- \frac{1}{1+\rho} \varidx{P}{n}{1} 
        - \varidx{P}{n}{2}
        - \frac{\rho}{1+\rho} \varidx{P}{n-1}{2} 
        - \frac{1}{1+\rho} Q_n \\
        &
        \leq 
            \begin{aligned}[t]
                &\frac{\rho}{2(\ylsecond{1}+\rho)} (\distancegap{n+1} + \distancegap{n}) \\
                &+ C^x  \sum_{k=0}^{n} \rho^{n-k} \squarednormtwo{\noisegradx{k}}
                    + C_{-1}^x \sum_{k=0}^{n-1} \rho^{n-k}  \squarednormtwo{\noisegradx{k}} 
                    + C^y  \sum_{k=0}^{n} \rho^{n-k} \squarednormtwo{\noisegrady{k}} \\
                &+ C_{-1}^y \sum_{k=0}^{n-1} \rho^{n-k} \squarednormtwo{\noisegrady{k}}
                    + C_{-2}^y \sum_{k=0}^{n-2} \rho^{n-k} \squarednormtwo{\noisegrady{k}} 
                + B^x \squarednormtwo{\noisegradx{n}} + B^y \squarednormtwo{\noisegrady{n}} + B_{-1}^y \rho \squarednormtwo{\noisegrady{n-1}}
            \end{aligned} \\
        &\leq  \frac{\rho}{2(\ylsecond{1}+\rho)} (\distancegap{n+1} + \distancegap{n})
            + \cQ_x \sum_{k=0}^{n} \rho^{n-k} \squarednormtwo{\noisegradx{k}} 
            + \cQ_y \sum_{k=0}^{n} \rho^{n-k} \squarednormtwo{\noisegrady{k}},  \\
\end{split}
\]
}
where $\cQ_x \defineq B^x + C^x + C_{-1}^x$ and $\cQ_y \defineq B^y + B_{-1}^y + C^y + C_{-1}^y + C_{-2}^y$. 
{This completes the proof.}
\QEDB

\subsubsection{Proof of Lemma~\ref{rso:lem:upperbound_hat_variables}}\label{proof:rso:lem:upperbound_hat_variables}

Let \sa{$k \in \N$} be fixed. 
In view of~\eqref{rso:eq:fixed_point_eq}, we have
{\small\begin{align*}
    \sa{\norm{\noiselessy_{k+1} - \opty}} 
    &\leq \frac{1}{1+\sigma \mu_y} \norm{y_k + \sigma(1+\theta) \gradphiy(x_{k}, y_k) - \sigma\theta \gradphiy(x_{k-1}, y_{k-1}) - \opty - \sigma \gradphiy(\optx, \opty)} \\
    &\leq \frac{1}{1+\sigma \mu_y} 
        \Big(
            \begin{aligned}[t]
                & \norm{y_k - \opty} + \sigma (1+\theta) \norm{\gradphiy(x_k, y_k) - \gradphiy(\optx, \opty)} \\
                & + \theta \sigma \norm{\gradphiy(x_{k-1}, y_{k-1}) - \gradphiy(\optx, \opty)} \Big)
            \end{aligned} \\
    &\leq \frac{1}{1+\sigma \mu_y}  
        \begin{aligned}[t]  
            \Big( \sigma (1+\theta) \Lyx \norm{x_k - \optx} 
            &+ (1 + \sigma (1+\theta) \Lyy) \norm{y_k - \opty} \\
            + \sigma\theta \Lyx \norm{x_{k-1} - \optx} 
            &+ \sigma\theta \Lyy \norm{y_{k-1} - \opty} \Big),\\
        \end{aligned} \\
\end{align*}}%
where the third inequality \sa{follows} from the smoothness assumptions on $\gradphix$ and $\gradphiy$, {and for the $k=0$ case, we have $x_{-1}=x_0$ and $y_{-1}=y_0$.} 
Using similar arguments, we also obtain 
\[
{\small
\begin{aligned}[t]
    \norm{\doublenoiselessx_{k+1} - \optx} 
        & \leq \frac{1}{1+\tau \mu_x} \Big((1+\tau \Lxx) \norm{x_k - \optx} + \tau \Lxy \norm{\noiselessy_{k+1} - \opty} \Big)\\
        &\leq \frac{1}{1+\tau \mu_x} 
            \begin{alignedat}[t]{2}
                \left(
                    \left(1+\tau \Lxx + \frac{\tau \sigma (1+\theta) \Lyx \Lxy}{1 + \sigma \mu_y}\right) \norm{x_k - \optx}\right. 
                    &+ \frac{\tau \sigma \theta \Lxy \Lyx}{1 + \sigma \mu_y} \norm{x_{k-1} - \optx}& \\
                     + \frac{\tau \Lxy (1 + \sigma (1+\theta)\Lyy)}{1 + \sigma \mu_y} \norm{y_k - \opty}& 
                     + \left. \frac{\tau \sigma \theta \Lxy \Lyy }{1 + \sigma \mu_y} \norm{y_{k-1} - \opty} \right),\\
            \end{alignedat}
\end{aligned}
}
\]
from which \sa{we deduce the following bound:}
\begin{equation*}
    \scalemath{0.9}{
    \begin{alignedat}[t]{1}
        \norm{\doublenoiselessy_{k+2} - \opty}  
            & \leq \frac{1}{1+\sigma \mu_y} 
                    \begin{alignedat}[t]{1}
                        \Big( 
                            \norm{\noiselessy_{k+1} - \opty} &
                            + \sigma (1+\theta) \Lyx \norm{\doublenoiselessx_{k+1} - \optx}
                            +  \sigma (1+\theta) \Lyy \norm{\noiselessy_{k+1} - \opty} \\
                        & + \sigma \theta \Lyx \norm{x_{k} - \optx} 
                              + \sigma \theta \Lyy \norm{y_{k} - \opty} \Big)\\
                    \end{alignedat}\\
            & \leq \frac{1}{1+\sigma \mu_y} 
                \begin{alignedat}[t]{1}
                     & \Bigg( 
                        \Bigg(
                            \begin{aligned}[t]
                                & (1 + \sigma (1+\theta) \Lyy) \frac{\sigma (1+\theta) \Lyx}{1 + \sigma \mu_y} \\
                                & + \sigma (1+\theta) \Lyx \frac{\left(1+\tau \Lxx + \frac{\tau \sigma (1+\theta) \Lyx \Lxy}{1 + \sigma \mu_y}\right)}{1+\tau \mu_x} + \sigma \theta \Lyx \Bigg) \norm{x_k - \optx}
                            \end{aligned} \\
                     & \quad + \left(\frac{(1 + \sigma (1+\theta) \Lyy)^2}{1+\sigma \mu_y} + \sigma (1+\theta) \Lyx \frac{\tau \Lxy (1 + \sigma (1+\theta)\Lyy)}{(1+\sigma \mu_y)(1+\tau \mu_x)} + \sigma \theta \Lyy \right) \norm{y_k - \opty}  \\  
                    & \quad +  \left((1 + \sigma (1+\theta) \Lyy) \frac{\theta \sigma \Lyx}{(1+\sigma \mu_y)} + \sigma (1+\theta) \frac{\tau \sigma \theta \Lxy \Lyx^2}{(1+\sigma \mu_y)(1+\tau \mu_x)}  \right) \norm{x_{k-1} - \optx} \\
                    & \quad + \left(\frac{(1 + \sigma (1+\theta) \Lyy) \sigma \theta \Lyy}{1+\sigma \mu_y} + \sigma (1+\theta) \frac{\tau \sigma \theta \Lxy \Lyx \Lyy}{(1+\sigma \mu_y)(1+\tau \mu_x)}  \right) \norm{y_{k-1} - \opty} \Bigg).\\
                \end{alignedat}
    \end{alignedat}
    }
\end{equation*}
{Combining the above bounds with Cauchy-Schwarz inequality implies \eqref{rso:eq:upperbound_hat_variables} and we conclude.}
\QEDB

%% file: A3_quadratic.tex
\subsection{Properties of SAPD~on the quadratic SP problem given in\texorpdfstring{~\eqref{rso:eq:quadratic_problem}}{4.1}
}\label{rso:sec:recalls_quadratics}

In this section, we briefly recall the \sa{discussion in}~\citep{zhang2021robust} regarding the \sa{convergence} behaviour of \sapdname~on the SP problem in~\eqref{rso:eq:quadratic_problem}.
Precisely, denoting $\sa{\tilde z_k} = [x_{k-1}, y_{k}]^\top$ and $\omega_k = \left[\omega_{k-1}^x \omega_{k-1}^y ; \omega_k^y\right]^{\top}$, the authors observe that \sa{$(\tilde z_k)_{k\geq 0}$} satisfies the recurrence relation $\tilde z_{k+1} = A \tilde z_k + B \omega_k$ where $A$ and $B$ are defined as 
\begin{equation}\label{rso:eq:a_lam_r_lam}
    \scalemath{0.8}{A \mgtwo{=}
        \left[\begin{array}{cc}
            \frac{1}{1+\tau \mu_x} I_d & \frac{-\tau}{\left(1+\tau \mu_x\right)} K \\
            \frac{1}{1+\sigma \mu_y}\left(\frac{\sigma(1+\theta)}{1+\tau \mu_x}-\sigma \theta\right) K & \frac{1}{1+\sigma \mu_y}\left(I_d-\frac{\tau \sigma(1+\theta)}{1+\tau \mu_x} K^2 \right)
            \end{array}\right], 
        \quad 
            B \mgtwo{=}
            \left[\begin{array}{ccc}
                    \frac{-\tau}{1+\tau \mu_x} I_d & 0_d & 0_d \\
                    \frac{-\tau \sigma(1+\theta)}{\left(1+\tau \mu_x\right)\left(1+\sigma \mu_y\right)} K & \frac{-\sigma \theta}{1+\sigma \mu_y} I_d & \frac{\sigma(1+\theta)}{1+\sigma \mu_y} I_d
            \end{array}\right]}.
\end{equation}
As a result, the covariance matrix \sa{$\tilde\Sigma_k$ of $\tilde z_k$} satisfies for all $k\geq 0$,
\begin{equation}\label{rso:eq:cov_recursion}
    \sa{\tilde\Sigma_{k+1}} = A \sa{\tilde\Sigma_{k}} A^\top + R,
\end{equation}
where $R = \frac{\delta^2}{d} B B^{\top}+A \mathbb{E}\left[\tilde z_k \omega_k^{\top}\right] B^{\top}+B \mathbb{E}\left[\omega_k \tilde z_k^{\top}\right] A^{\top}$. 
Using the independence assumptions on the $\gnoisex{k}$'s and $\gnoisey{k}$'s, elementary derivations lead to expressing $R$ as
\[
    \scalemath{0.8}{R = \frac{\delta^2}{d} \left[\begin{array}{cc}
    \frac{\tau^2}{\left(1+\tau \mu_x\right)^2} 
&   \left(\frac{\tau^2 \sigma(1+\theta)}{\left(1+\tau \mu_x\right)^2\left(1+\sigma \mu_y\right)}
        + \frac{\tau \sigma^2 \theta (1+\theta)}{\left(1+\sigma \mu_y\right)^2(1+\tau \mu_x)} \right) K \\
\left(\frac{\tau^2 \sigma(1+\theta)}{\left(1+\tau \mu_x\right)^2\left(1+\sigma \mu_y\right)}
        + \frac{\tau \sigma^2 \theta (1+\theta)}{\left(1+\sigma \mu_y\right)^2(1+\tau \mu_x)} \right) K &  
    \frac{\sigma^2(1+\theta)^2}{\left(1+\sigma \muy \right)^2}
        \left(
            \frac{\tau^2}{\left(1+\tau \mu_x\right)^2}
                + \frac{2 \tau \sigma \theta}{\left(1+\tau \mu_x\right)\left(1+\sigma \mu_y\right)}\right) K^2 
    + \frac{\sigma^2}{\left(1+\sigma \mu_y\right)^2} \left(1 + \frac{2 \theta(1+\theta)\sigma \mu_y}{1+ \sigma \muy}\right) I_d\\
\end{array}\right]}.
\]
Provided that the spectral radius $\rho(A)$ of $A$ is less than $1$, the sequence \sa{$(\tilde \Sigma_k)_{k\geq 0}$} converges to a matrix \sa{$\tilde \Sigma^{\infty}$} satisfying
\begin{equation}\label{rso:eq:Lyap_eq}
    \sa{\tilde \Sigma^{\infty}= A 
    \tilde \Sigma^{\infty} A^\top + R.}
\end{equation}
Leveraging the spectral theorem, it is shown in~\citep{zhang2021robust} that an orthogonal change of basis enables to reduce the $2d \times 2d$ Lyapunov equation to $d$ systems of the following form for each $\lambda\in\spectrum(K)$:  
\begin{equation}\label{rso:eq:Lyap_eq_reduced}
    \saa{\tilde \Sigma^{\infty, \lambda} = A^{\lambda} \tilde \Sigma^{\infty, \lambda} {A^{\lambda}}^\top + R^{\lambda},}
\end{equation}
\saa{such that $A^{\lambda}$ and $R^{\lambda}$ are $2\times 2$ matrices defined for each $\lambda\in\spectrum(K)$ as}
\[
    \begin{split}
        A^{\lambda} &\defineq \left[\begin{array}{cc}
    \frac{1}{1+\tau \mu_x}  & \frac{-\tau}{\left(1+\tau \mu_x\right)} \lambda \\
    \frac{1}{1+\sigma \mu_y}\left(\frac{\sigma(1+\theta)}{1+\tau \mu_x}-\sigma \theta\right) \lambda & \frac{1}{1+\sigma \mu_y}\left(I_d-\frac{\tau \sigma(1+\theta)}{1+\tau \mu_x} \lambda^2\right)
    \end{array}\right] \\
    R^{\lambda} &\defineq \delta^2 \scalemath{0.8}{\left[\begin{array}{cc}
        \frac{\tau^2}{\left(1+\tau \mu_x\right)^2} 
    &   \left(\frac{\tau^2 \sigma(1+\theta)}{\left(1+\tau \mu_x\right)^2\left(1+\sigma \mu_y\right)}
            + \frac{\tau \sigma^2 \theta (1+\theta)}{\left(1+\sigma \mu_y\right)^2(1+\tau \mu_x)} \right) \lambda \\
    \left(\frac{\tau^2 \sigma(1+\theta)}{\left(1+\tau \mu_x\right)^2\left(1+\sigma \mu_y\right)}
            + \frac{\tau \sigma^2 \theta (1+\theta)}{\left(1+\sigma \mu_y\right)^2(1+\tau \mu_x)} \right) \lambda &  
        \frac{\sigma^2(1+\theta)^2}{\left(1+\sigma \muy \right)^2}
            \left[
                \frac{\tau^2}{\left(1+\tau \mu_x\right)^2}
                    + \frac{2 \tau \sigma \theta}{\left(1+\tau \mu_x\right)\left(1+\sigma \mu_y\right)}\right] \lambda^2
        + \frac{\sigma^2}{\left(1+\sigma \mu_y\right)^2} \left(1 + \frac{2 \theta(1+\theta)\sigma \mu_y}{1+ \sigma \muy}\right)\\ 
    \end{array}\right]}   
    \end{split}
\]
{and $A$ is similar to the matrix $\diag(A^{\lambda_1},\cdots, A^{\lambda_d})$. Therefore, we have $\rho(A) = \max_{i=1,2,\dots,d}\rho(A^{\lambda_i})$}. 

\subsection{Proof of Theorem~\ref{rso:thm:quadratics}}\label{rso:sec:proof_thm_quadratics}
\label{sec:proof-lim-cov}

In this section, we solve the Lyapunov equations \eqref{rso:eq:Lyap_eq_reduced} analytically under the parameterization in~\eqref{rso:eq:cb_params}. 
Throughout, \mg{given $\lambda \in \spectrum(K)$, we introduce the quantity $\klam \defineq \frac{\lambda}{\sqrt{\mux \muy}}$ which is closely related to the condition number $\kappa = \max\{L_{xx}, L_{yx}, L_{yy}\}/\min\{\mu_y, \mu_x\}$. 
Indeed, for $\mu_x=\mu_y$, we have \saa{$\kappa = \max \{|\kappa_{\lambda}|:~ \lambda\in\spectrum(K)\}$}. \saa{For each $\lambda\in\spectrum(K)$, let} $\tilde\Sigma^{\infty, \lambda}$ be \saa{a} solution to \eqref{rso:eq:Lyap_eq_reduced}, i.e., $\tilde\Sigma^{\infty, \lambda}$ \saa{solves the following $2\times 2$ Lyapunov equation:}
\begin{equation}\label{rso:eq:Lyap_eq_reduced_lambda}
    \tilde \Sigma^{\infty, \lambda} = A^{\lambda} \tilde \Sigma^{\infty, \lambda} {A^{\lambda}}^\top + \Rlam.
\end{equation}
\saa{Furthermore,} such a solution is unique if $\rho(A^{\lambda})<1$ \citep{laub1990sensitivity,hassibi1999indefinite}.
The following result provides an explicit formula to $\tilde \Sigma^{\infty, \lambda}$ whose proof is deferred to Section \ref{proof-of-prop:limit_covariance_quadratic}.} 
\begin{prop}\label{rso:prop:limit_covariance_quadratic}
Under the Chambolle-Pock parameterization \saa{in~\eqref{rso:eq:cb_params},} for \mg{$\lambda \neq 0$} and \mg{$\theta \in \left(\frac{1}{\klam} \bigg(\sqrt{1 + {(\klam)}^2} - 1\bigg), 1\right)$}, \mg{we have $\rho(A^{\lambda})<1$,} and the unique solution $\Sigma^{\infty, \lambda}$ \saa{to the equation} in~\eqref{rso:eq:Lyap_eq_reduced_lambda} is given by 
        \[
            \tilde\Sigma^{\infty, \lambda} = \frac{\delta^2(1-\theta)}{d\lambda^2 P_c(\theta, \klam)} \left[\begin{array}{cc}
            \frac{\lambda^2}{\mux^2}\left(\tilde P_{1,1}^{(1)}(\theta, \klam) + \frac{\lambda^2}{\mu_y^2} \tilde P_{1,1}^{(2)}(\theta, \klam) \right)
                & \frac{\lambda}{\mu_x} \left(\tilde P_{1,2}^{(1)}(\theta, \kappa) + \frac{\lambda^2}{\mu_y^2} \tilde P_{1,2}^{(2)}(\theta, \klam)\right) \\
            \frac{\lambda}{\mu_x} \left(\tilde P_{1,2}^{(1)}(\theta, \klam) + \frac{\lambda^2}{\mu_y^2} \tilde P_{1,2}^{(2)}(\theta, \klam)\right) 
        \end{array}\right],
        \]
where  $P_c$ and $P_{q,\ell}^{(k)}$ for $q=1,2$ and $\ell=1,2$ are polynomials in $\theta$ and $\klam$, defined in \ylfourth{the top part of Table~\ref{rso:table:polynomials}}. 
\mg{Otherwise, for $\lambda=0$ and $\theta\in [0,1)$, we also have $\rho(\Alam)<1$ and the unique solution $\tilde \Sigma^{\infty, \lambda}$ of~\eqref{rso:eq:Lyap_eq_reduced_lambda} is given by} 
\[
        \tilde \Sigma^{\infty, 0} 
            = \frac{\delta^2}{d} \frac{(1-\theta)}{\mu_x^2 \mu_y^2(1+\theta)}\scalemath{0.8}{
                    \left[\begin{array}{cc}
                        \mu_y^2 & 0 \\
                        0 & \mu_x^2\left(1+2\left(1-\theta^2\right) \theta\right)
                    \end{array}\right]}.
\]   
\end{prop}

Theorem~\ref{rso:thm:quadratics} \mg{will then follow} directly from Proposition~\ref{rso:prop:limit_covariance_quadratic}. 
\begin{proof}[Proof of Theorem~\ref{rso:thm:quadratics}] 
Proposition~\ref{rso:prop:limit_covariance_quadratic} characterizes the asymptotic covariance matrix of $(x_{n-1}, y_n)$ in the limit as $n\to\infty$, which we will use to deduce the covariance matrix $\Sigma^{\infty}$ of $(x_n, y_n)$ in the limit as $n\to\infty$. 
First, recall from~\citep{zhang2021robust} that the orthogonal matrix leading to the reduced Lyapunov~\eqref{rso:eq:Lyap_eq_reduced} \mg{is given by} $Z = P V$ where $P$ is the permutation matrix associated to the permutation $\mathcal{P}$ of $\{1,\ldots, 2d \}$ defined as $\mathcal{P}(qd + r) = q + 2r-1\; \text{mod}\;[d]$, for all $q \in \{0,1\}, r \in \{1,\dots, d\}$,
and $V \defineq \diag(U, U) \in \mathbb{R}^{2d \times 2d}$ where $U$ describes an orthogonal basis for $K$ with $K = U \diag(\lambda_1, \dots, \lambda_d) U^\top$. 
Now, since $x_n = \frac{1}{1+\tau \mux}(x_{n-1} - \tau K y_n)$, we have $[x_n^\top , y_n^\top ]^\top = T [x_{n-1}^\top , y_n^\top ]^\top$ where
{\small
\[
    T \defineq
        \left[\begin{array}{cc}
                \frac{1}{1+\tau \mu_x} I & \frac{-\tau}{1+\tau \mu_x} K \\
                0 & I
            \end{array}\right].
\]}
Thus, $\Sigma^\infty = T \Sigma^\infty T^\top$, and noting that $T$ admits the block diagonalization $T = Z T^{(\Lambda)} Z^\top$ where $T^{(\Lambda)} = \diag(T^{(\lambda_1)}, \dots, T^{(\lambda_d)})$ and 
\[
    T^{(\lambda_i)} = 
        \left[\begin{array}{cc}
            \frac{1}{1+\tau \mu_x} & \frac{-\tau \lambda_i}{1+\tau \mu_x} \\
            0 & 1
            \end{array}\right] \quad \forall i \in\{1, \dots, d\},
\]
we obtain $\Sigma^\infty = Z T^{(\Lambda)} \tilde \Sigma^{\infty, \Lambda} (T^{(\Lambda)})^\top Z^\top$, where $\tilde \Sigma^{\infty, \Lambda} = \diag(\tilde \Sigma^{\infty, \lambda_1}, \dots, \tilde \Sigma^{\infty, \lambda_d})$.
Finally, we observe that $T^{(\Lambda)} \tilde \Sigma^{\infty, \Lambda} (T^{(\Lambda)})^\top = \diag(\Sigma^{\infty,, \lambda_1}, \dots, \Sigma^{\infty,, \lambda_d})$ where 
\[
        \Sigma^{\infty, \lambda_i} 
            \defineq \left[\begin{array}{cc}
            \theta^2 \tilde \Sigma^{\infty, \lambda_i}_{11}-2 \theta(1-\theta) \frac{\lambda_i}{\mu_x} \tilde \Sigma^{\infty, \lambda_i}_{12}+\frac{(1-\theta)^2 \lambda_i^2}{\mu_x^2} \tilde \Sigma^{\infty, \lambda_i}_{22} & \theta \tilde \Sigma^{\infty, \lambda_i}_{12}-\theta(1-\theta) \frac{\lambda_i}{\mu_x} \tilde \Sigma^{\infty, \lambda_i}_{22} \\
            \theta \tilde \Sigma^{\infty, \lambda_i}_{12}-\theta(1-\theta) \frac{\lambda_i}{\mu_x} \tilde \Sigma^{\infty, \lambda_i}_{22} & \tilde \Sigma^{\infty, \lambda_i}_{22}
            \end{array}\right].
\]
Plugging $\lambda=\lambda_i$ {into} the expression of $\tilde \Sigma^{\infty, \lambda}$ computed in Proposition~\ref{rso:prop:limit_covariance_quadratic}, we obtain $\Sigma^{\infty, 0} = \frac{\delta^2}{d} \frac{(1-\theta)}{\mu_x^2 \mu_y^2(1+\theta)} \diag \left[\theta^2 \mu_y^2,\;  \mu_x^2\left(1+2\left(1-\theta^2\right) \theta\right)\right]$, if $\lambda_i=0$; otherwise,
{\small
\[
        \Sigma^{\infty, \lambda_i} = \frac{(1-\theta) \delta^2}{d\lambda_i^2 P_c(\theta, \kappa)} 
            \left[\begin{array}{cc}
                \frac{\lambda_i^2}{\mux^2}\left(P_{1,1}^{(\infty, 1)}(\theta, \kappa) + \frac{\lambda_i^2}{\mu_y^2} P_{1,1}^{(\infty, 2)}(\theta, \kappa) \right)
                & \frac{\lambda}{\mu_x} \left(P_{1,2}^{(\infty, 1)}(\theta, \kappa) + \frac{\lambda_i^2}{\mu_y^2} P_{1,2}^{(\infty, 2)}(\theta, \kappa)\right) \\
                \frac{\lambda}{\mu_x} \left(P_{1,2}^{(\infty, 1)}(\theta, \kappa) + \frac{\lambda_i^2}{\mu_y^2} P_{1,2}^{(\infty, 2)}(\theta, \kappa)\right) 
                & P_{2,2}^{(\infty, 1)}(\theta, \kappa) + \frac{\lambda_i^2}{\mu_y^2} P_{2,2}^{(\infty, 2)}(\theta, \kappa),
\end{array}\right],
\]}%
where the polynomials $P_{i,j}^{(\infty, k)}$ and $P_c$ are given {explicitly} in {the bottom part of} Table~\ref{rso:table:polynomials} of Appendix \ref{rso:sec:constants}. {From the closed-form expressions of these polynomials, the fact that the elements of the matrix $\Sigma^{\infty,\lambda}$ scale with $(1-\theta)$ as $\theta \to 1$ can be checked in a straightforward manner.} 
\end{proof}
\subsection{\sa{Proof of Corollary~\ref{coro:cov-rate}}}
\label{sec:cov-rate-proof}
Let $V,J$ denote the Jordan decomposition of $A$. 
For $n\in \N$, let $\widecheck{\Sigma}_n\triangleq V^{-1} \Sigma_n \left(V^{-1}\right)^{\top}$, $\widecheck{\Sigma}^\infty\triangleq V^{-1} \Sigma^\infty \left(V^{-1}\right)^{\top}$, and $\widecheck{R}^i=V^{-1} R^{{}}\left(V^{-1}\right)^{\top}$. 
In view of the recursion~\eqref{rso:eq:cov_recursion}, we have $\widecheck{\Sigma}_{n+1} =J \widecheck{\Sigma}_{n} J+ \widecheck{R}$, and vectorizing again this recursion lead to $\vectorize(\widecheck{\Sigma}_{n+1}) = (J \otimes J) \vectorize(\widecheck{\Sigma}_{n}) + \widecheck{R}$, i.e., $\widecheck{\Sigma}_{n} = (J \otimes J)^{n-1}\; \widecheck{\Sigma}_{1} + \sum_{k=1}^{n-1} (J \otimes J)^{k-1} \vectorize (\widecheck{R})$.
Hence, noting that $\widecheck{\Sigma}^{\infty} = \sum_{k=0}^{\infty} (J \otimes J)^{k} \vectorize (\widecheck{R})$ we obtain
{\small
\[
\begin{split}
    \spectralradius{\Sigma_n - \Sigma^\infty} 
        &= \spectralradius{\widecheck{\Sigma}_n - \widecheck{\Sigma}^\infty} = \spectralradius{\vectorize(\widecheck{\Sigma}_n) - \vectorize( \widecheck{\Sigma}^\infty)} = \spectralradius{(J \otimes J)^{n-1}\; \widecheck{\Sigma}_{1} + \sum_{k=n-1}^{\infty} (J \otimes J)^{k-1} \vectorize (\widecheck{R})} \\
        &\leq \rho{(J \otimes J)}^{n-1} \; \spectralradius{\widecheck{\Sigma}_{1} + \widecheck{\Sigma}^\infty}, \\ 
\end{split}
\]}%
and the {claimed convergence rate} follows from observing that $\rho\left(J \otimes J\right) = \rho(A)^2$. {Note that here $\rho(A)<1$ because by Proposition \ref{rso:prop:limit_covariance_quadratic} we have $\rho(A^{\lambda_i}) < 1$ for every $i$ and $\rho(A)= \max_i \rho(A^{\lambda_i})$}.
\subsection{Proof of Theorem~\ref{rso:prop:tightness}}\label{rso:sec:proof_tightness}

We start with proving the lower bound, and then we will proceed to the upper bound.

\subsubsection{Lower bound}\label{sec:app:lowerbound}
In view of Theorem~\ref{rso:thm:quadratics}, $z_\infty$ follows a centered Gaussian distribution with covariance matrix $\Sigma^\infty$ as defined in~\eqref{eq:limit_cov_informal_non_zero}. 
Hence, let $X \sim \cN(0, I_d)$ be such that $\squarednormtwo{z_\infty} = X^\top \Sigma^\infty X$. We almost surely have $\squarednormtwo{z_n} \geq \squarednormtwo{X}\; \min \spectrum(\Sigma^\infty) \defineq \psi_1(p, \theta)$. 
By~\citep{inglot2010inequalities}, we have $Q_p(\squarednormtwo{z_\infty-\optz}) = Q_p(\squarednormtwo{z_\infty}) \geq 2d + 2 \log\left(1/(1-p))\right) - 5/2$, {where we used $\optz=(0,0)$}.
Thus, it suffices to show that $\min \spectrum(\Sigma^\infty) = \Theta{(1-\theta)}$ as $\theta \to 1$. 
Given the bloc decomposition of $\Sigma^\infty$ in~\eqref{eq:limit_cov_informal_non_zero}, we have 
{\small 
\[
\begin{split}
    \min \spectrum(\Sigma^\infty) &= \min_{i \in \{1, \ldots, d\}} \spectrum(\Sigma^{\infty, \lambda_i}) = \min_{i \in \{1, \ldots, d\}} \frac{1}{2} \left(\Sigma^{\infty,, \lambda_i}_{11} + \Sigma^{\infty,, \lambda_i}_{22} - \sqrt{\left(\Sigma^{\infty,, \lambda_i}_{11} - \Sigma^{\infty, \lambda_i}_{22} \right)^2 + 4 {\Sigma^\infty_{12}}^2}\right).\\
\end{split}
\]}
We will now show that for all $\lambda \in \spectrum(K)$, $\Sigma^{\infty,, \lambda}_{11} + \Sigma^{\infty,, \lambda}_{22} - \sqrt{\left(\Sigma^{\infty,, \lambda}_{11} - \Sigma^{\infty, \lambda}_{22} \right)^2 + 4 {\Sigma^\infty_{12}}^2} = \Theta(1-\theta)$, as $\theta \to 1$. If $0 \in \spectrum(\Sigma^\infty)$, given~\eqref{eq:limit_cov_informal_0}, we have 
\[
    \begin{split}
        \Sigma_{1,1}^{\infty, 0}+\Sigma_{2,2}^{\infty, 0}
            &- \sqrt{\left(\Sigma_{1,1}^{\infty, 0}-\Sigma_{2,2}^{\infty, 0}\right)^2+4 (\Sigma_{1,2}^{\infty, 0})^2} \\
            &= \frac{\delta^2}{d} \frac{(1-\theta)}{\mu_x^2 \mu_y^2(1+\theta)}
                \left(\theta^2 \mu_y^2+\mu_x^2\left(1+2\left(1-\theta^2\right) \theta\right)-\mid \theta_{\mu_y^2}^2-\mu_x^2\left(1+2\left(1-\theta^2\right) \theta \right) \mid\right) \\
            &= \frac{\delta^2}{d} \frac{2\theta(1-\theta)}{\mu_x^2 \mu_y^2(1+\theta)}
                \min\left(\mu_y^2, \mu_x^2\left(1+2\left(1-\theta^2\right)\right)\right) = \Theta(1-\theta),\\
    \end{split}
\]
where second equality follows from having $a+b - |a-b| = 2\min(a,b)$. 
If $\lambda \neq  0$ is in $\spectrum(K)$, in view of Table~\ref{rso:table:polynomials} {of Appendix \ref{rso:sec:constants}}, we have as $\theta \to 1$,
{\small
\[
\begin{array}{lll}
    P_{1,1}^{\infty, 1}(\theta, \kappa)=-16 \kappa^2+\littleOh(1-\theta), & P_{1,2}^{\infty, 1}(\theta, \kappa)=-8 \kappa^4+\littleOh(1-\theta), &
    P_{2,2}^{\infty, 1}(\theta, \kappa)=-8 \kappa^6+\littleOh(1-\theta), \\
    P_{1,1}^{\infty, 2}(\theta, \kappa)=-8 \kappa^2+\littleOh(1-\theta), & P_{1,2}^{\infty, 2}(\theta, \kappa)=8 \kappa^2+\littleOh(1-\theta), & P_{2,2}^{\infty, 2}(\theta, \kappa)=-16 \kappa^2+\littleOh(1-\theta), \\
\end{array}
\]}%
and $P_c(\theta, \kappa)=-32, \kappa^2+ \littleOh(1-\theta)$, so that 
{\small
\[
\begin{array}{ll}
    \Sigma_{1,1}^{\infty, \lambda} 
        = \frac{(1-\theta)\delta^2}{d (-32\kappa^2) \lambda^2} (-8\kappa^2)\left(\frac{\lambda^2}{\mux^2} + \frac{\lambda^4}{\mux^2 \muy^2}\right) + \littleOh(1-\theta), 
    & \Sigma_{1,2}^{\infty, \lambda} 
        =  \frac{(1-\theta)\delta^2}{d (-32\kappa^2) \lambda^2} (8\kappa^2) \left(\frac{\lambda^3}{\muy^2 \mux} - \frac{\lambda^3}{\mux^2 \muy}\right) + \littleOh(1-\theta), \\
    &\Sigma_{2,2}^{\infty, \lambda} 
        = \frac{(1-\theta)\delta^2}{d (-32\kappa^2) \lambda^2} (-8\kappa^2)\left(\frac{\lambda^2}{\muy^2} + \frac{\lambda^4}{\mux^2 \muy^2}\right) + \littleOh(1-\theta). \\        
\end{array}
\]}%
Hence, we deduce that 
{\small
\[
    \begin{split}
        \Sigma_{1,1}^{\infty, \lambda}+\Sigma_{2,2}^{\infty, \lambda}
        &- \sqrt{\left(\Sigma_{1,1}^{\infty, \lambda}-\Sigma_{2,2}^{\infty, \lambda}\right)^2+4 (\Sigma_{1,2}^{\infty, \lambda})^2} \\
        &= \frac{(1-\theta) \delta^2}{2d \mux^2 \muy^2} \left(\lambda^2 + \mux^2 + \muy^2 - \sqrt{(\mux^2 - \muy^2)^2 + \lambda^2 (\mux-\muy)^2} \right) + \bigOh(1-\theta), \\
    \end{split}
\]}%
and it suffices to show that 
{\small
    $\lambda^2 + \mux^2 + \muy^2 - \sqrt{(\mux^2 - \muy^2)^2 + \lambda^2 (\mux-\muy)^2}  > 0$.
Now given the identity $a-b = \frac{a^2 - b^2}{a+b}$ for $a+b\neq 0$, we have
{\small
\[
    \lambda^2 + \mux^2 + \muy^2 - \sqrt{(\mux^2 - \muy^2)^2 + \lambda^2 (\mux-\muy)^2}  = \frac{\lambda^2 (\mux-\muy)^2 + \lambda^4}{\lambda^2 + \mux^2 + \muy^2 + \sqrt{(\mux^2 - \muy^2)^2 + \lambda^2 (\mux-\muy)^2}} > 0,
\]}
which completes the proof.

\subsubsection{Upper bound}

\mg{The CP parametrization corresponds to choosing $\alpha = \frac{1}{2\sigma} - \sqrt{\theta} L_{yy}$ in the matrix inequality \citep[Cor. 1]{zhang2021robust}.} 
Under this parameterization, since $1-\alpha\sigma \geq 1/2$, we have ${\sa{\cD_n} \geq \frac{\theta}{4(1-\theta)} \left(\mu_x {\squarednormtwo{x_n-\optx}} + \muy {\squarednormtwo{y_n-\opty}} \right)}$ {and we have $\optz=(\optx,\opty)=(0,0)$}. \sa{Since $\{z_n\}_{n\geq 0}$ converges in distribution to $z_\infty$, \eqref{rso:eq:high_probability_bound_sapd} implies that} the $p$-quantile of $\norm{z_\infty}^2$ satisfies \sa{$Q_p(\norm{z_\infty}^2) \leq \psi_2(p, \theta)$ for any $p \in (0, 1)$, where} ${\psi_2(p, \theta) \triangleq  \frac{4(1-\theta)}{\theta \min\{\mux, \muy\}}\; 
        \left(
            \Xi^{(1)}_{\tau,\sigma,\theta} 
            + \Xi^{(2)}_{\tau,\sigma,\theta}
                \log\left(
                    \frac{1}{1-p}
                \right) 
        \right).}$
\yassine{Thus, the asymptotic property of our upper bound follows from Lemma~\ref{rso:lem:asymptotic_xi}.}

\subsection{Proof of Proposition~\ref{rso:prop:limit_covariance_quadratic}}\label{proof-of-prop:limit_covariance_quadratic}

We first note that under the parameterization~\eqref{rso:eq:cb_params}, the matrices $\Alam$ and $R^\lambda$ simplify to 
{\small 
\begin{equation}\label{rso:eq:a_lam_r_lam_under_cb}
    \begin{split}
        \Alam &= \left[\begin{array}{cc}
    \theta & -(1-\theta) \frac{\lambda}{\mu_x} \\
    (1-\theta) \theta^2 \frac{\lambda}{\mu_y} & \theta-(1-\theta)^2(1+\theta) \kappa^2
    \end{array}\right], \\
        R^\lambda &= \frac{\delta^2}{d} \frac{(1-\theta)^2}{\mu_x^2 \mu_y^2}\left[\begin{array}{cc}
    \mu_y^2 & \left(1-\theta^2\right)\left(\theta \mu_x+\mu_y\right) \lambda \\
    \left(1-\theta^2\right)\left(\theta \mu_x+\mu_y\right) \lambda & (1-\theta)^2(1+\theta)^2\left(1+2 \theta \frac{\mu_x}{\mu_y}\right) \lambda^2+\mu_x^2\left(1+2\left(1-\theta^2\right) \theta\right)
    \end{array}\right].\\
    \end{split}
\end{equation}
}
If $\lambda = 0$, then $\Alam = \text{Diag}(\theta, \theta)$. 
Hence, using the relation $\vectorize(ABC) = (C^\top \otimes A)\vectorize(B)$, we have 
\[
    \begin{split}
        \tilde \Sigma^{\infty, \lambda}=\Alam \tilde \Sigma^{\infty, \lambda} \Alam+R & \Leftrightarrow \operatorname{Vec}\left(\tilde \Sigma^{\infty, \lambda}\right)=\left(\Alam \otimes \Alam\right) \operatorname{Vec}\left(\tilde \Sigma^{\infty, \lambda}\right)+\operatorname{Vec}(R^\lambda) \\
        & \Leftrightarrow \operatorname{Vec}\left(\tilde \Sigma^{\infty, \lambda}\right)=\left(I-\Alam \otimes \Alam\right)^{-1} \operatorname{Vec}\left(R^\lambda\right).
    \end{split}
\]
Noting that $\left(I-\Alam \otimes \Alam\right)^{-1} = \text{Diag}(\frac{1}{1-\theta^2}, \frac{1}{1-\theta^2}, \frac{1}{1-\theta^2}, \frac{1}{1-\theta^2})$, we obtain $\Sigma^{\infty, 0} = \frac{1}{1-\theta^2} R^0$ \mg{for any $\theta\in [0,1)$}.
\medskip
\mg{It remains to consider the case when $\lambda\neq 0$. We first provide an eigenvalue decomposition to the matrix $\Alam$.}
\begin{lem}\label{rso:lem:airjordan}
    For any $\theta \in \left(\big(\sqrt{1+\klam^2} - 1\big)/\yassine{|\klam|}, 1\right)$ \mg{and $\lambda\neq 0$}, the matrix $\Alam$ introduced in~\eqref{rso:eq:Lyap_eq_reduced} admits the diagonalization \mg{$\Alam = V^\lambda  J^{\lambda}(\Vlam)^{-1}$} where 
    \begin{equation}\label{eq-Jlam}
        \Jlam = \left[\begin{array}{cc}
        \nuonelam & 0 \\
        0 & \nutwolam \\
        \end{array}\right], \quad
        \Vlam \defineq \left[\begin{array}{cc}
        - A_{1,2}^{{}} & - A_{1,2}^{{}} \\
        \theta - \nuonelam & \theta - \nutwolam \\
        \end{array}\right],
    \end{equation}
    with \mg{complex eigenvalues} $\nuonelam \defineq \frac{\left(2 \theta-(1-\theta)^2(1+\theta) \klam^2\right)+i \sqrt{|\Delta|}}{2}$, $\nutwolam \defineq \frac{\left(2 \theta-(1-\theta)^2(1+\theta) \klam^2\right)-i \sqrt{|\dislam|}}{2}$, and $\dislam = (1-\theta)^4(1+\theta)^2 \klam^4-4 \theta^2(1-\theta)^2 \klam^2$.
    \yassine{Moreover, in this case, $\dislam<0$ and $\rho(\Alam) < 1$.}
\end{lem}
\begin{proof} 
    Noting that $\text{Tr}(\Alam) = 2\theta - (1-\theta)^2(1+\theta) \klam^2$ and $\text{Det}(\Alam) = \theta^2 - (1-\theta)^2 \theta \klam^2$, the characteristic polynomial of $\Alam$ has for discriminant $\dislam= \text{Tr}(\Alam) - 4 \text{Det}(\Alam) = (1-\theta)^4(1+\theta)^2 \klam^4-4 \theta^2(1-\theta)^2 \klam^2$.
    Note also that \mg{$\klam \neq 0$ since $\lambda\neq 0$ by assumption, and}
    \[
    \begin{split}
        \dislam < 0 \iff (1-\theta^2)^2 \leq \frac{4\theta^2}{\klam^2} \iff (1-\theta^2) \leq \frac{2 \theta }{\yassine{|\klam|}} \iff \theta \geq \frac{1}{\yassine{|\klam|}} (\sqrt{1 + \klam^2} - 1),
    \end{split}
    \]
    and in such case, \mg{it is straightforward to check that $\Alam $ admits the two complex conjugate values $\nuonelam, \nutwolam$}. 
    Furthermore, \mg{observe that $A_{12}^{(\lambda}\neq 0$ as $\lambda \neq 0$ and $\theta <1$} and for $\nu \in \mathbb{C}$, $x,y\in \mathbb{C}$, 
    \[
        \begin{aligned}
        (\Alam- v I)\left[\begin{array}{l}
                        x \\
                        y
                    \end{array}\right]=0 
        & \Leftrightarrow
            \left\{\begin{array}{l}
                \left(\theta -v\right) x+A_{1,2}^{\lambda} y=0 \\
                \yassine{A_{2,1}^{\lambda} x+\left(A_{2,2}^{\lambda}-v\right) y}=0
            \end{array}\right. 
         \Leftrightarrow y=\frac{-\left(\theta -v\right)}{A_{1,2}^{\lambda}} x \\
        &\Leftrightarrow \quad(x, y) \in \operatorname{Span}\left(\begin{array}{c}
        1 \\
        -\frac{\left(\theta -v\right)}{A_{1,2}^{\lambda}}
        \end{array}\right)=\operatorname{Span}\left(\begin{array}{l}
        -A_{1,2}^{{}} \\
        \theta -v
        \end{array}\right).
        \end{aligned}
    \]
    \mg{Therefore, the columns of the $\Vlam$ matrix are in fact eigenvectors corresponding to the complex conjugate eigenvalues $\nuonelam$ and $\nutwolam$, and we conclude that the eigenvalue decomposition $\Alam = V^\lambda J^{\lambda}(\Vlam)^{-1}$ holds}. Finally, $\rho(\Alam)^2 =|\nuonelam |^2 = \text{Det}(\Alam) = \theta^2 - \theta (1-\theta)^2 \klam^2$ so that we have $\mg{\rho(\Alam)^2-1 =\text{Det}(\Alam) -1}= -(1-\theta)\left(1+\theta\left(1+\klam^2\right)-\theta^2 \klam^2\right),$ and \mg{$\rho(\Alam)^2=1$} if and only if $\theta \in \left\{1, \frac{1}{2} + \frac{1}{2\klam^2}  \pm \frac{1}{2\klam^2} \sqrt{(1+\klam^2)^2 + 4 \klam^2}\right\}$.
    Observing that $\sqrt{(1+\klam^2)^2 + 4 \klam^2} \geq 1 + \klam^2$, we deduce that $\frac{1}{2} + \frac{1}{2\klam^2} + \frac{1}{2\klam^2} \sqrt{(1+\klam^2)^2 + 4 \klam^2} > 1$ and $\frac{1}{2} + \frac{1}{2\klam^2} - \frac{1}{2\klam^2} \sqrt{(1+\klam^2)^2 + 4 \klam^2} < 0$. 
    Hence, we conclude $\rho(\Alam) < 1$ for any $\theta \in \left(\frac{1}{\klam} (\sqrt{1 + \klam^2} - 1), 1\right)$. 
\end{proof}
\mg{In the following lemma, we also provide basic identities satisfied by the eigenvalues $\nuonelam$ and $\nutwolam$ which will be key for the exact computation of $\tilde \Sigma^{\infty, \lambda}$. 
The proof of this lemma is omitted as it follows from straightforward calculations.}

\begin{lem}\label{rso:lem:polynomial_root}
    Let $\nuonelam, \nutwolam$, be the two complex conjugate eigenvalues of \mg{$\Alam$}, as specified in Lemma~\ref{rso:lem:airjordan}.
    Then,
    {\small \[ 
    \begin{split}
        \nuonelam \nutwolam &= \theta^2 - \theta (1-\theta)^2 \klam^2, \\
        \nuonelam + \nutwolam &= 2\theta - (1-\theta)^2(1+\theta) \klam^2, \\
        \nuonelam^2+\nutwolam^2 &= 2 \theta^2-2 \theta(1-\theta)^2(1+2 \theta) \klam^2+(1-\theta)^4(1+\theta)^2 \klam^4, \\
        \nuonelam^3+\nutwolam^3 &= \left(2 \theta-(1-\theta)^2(1+\theta) \klam^2\right)\left(\theta^2-\theta(1-\theta)^2(1+4 \theta) \klam^2+(1-\theta)^4(1+\theta)^2 \klam^4\right), \\
        \nuonelam^4+\nutwolam^4& =  \begin{aligned}[t]
                        & 2 \theta^4 -(1-\theta)^2 \klam^2 \theta^3(4+16 \theta) +(1-\theta)^4 \klam^4 \theta^2\left(6+24 \theta+20 \theta^2\right) \\
                        &- (1-\theta)^6 \klam^6 4 \theta\left(1+4 \theta+5 \theta^2+2 \theta^3\right) + (1-\theta)^8 \klam^8(1+\theta)^4,
                        \end{aligned}\\
    \end{split}        
    \]}
    where $\kappa_\lambda = \frac{\lambda}{\sqrt{\mux \muy}}$.
\end{lem}

\mg{The following lemma says that the solution $\tilde \Sigma^{\infty, \lambda}$ of~\eqref{rso:eq:Lyap_eq_reduced_lambda} can be computed by solving 4-dimensional linear equations.}

\begin{lem}\label{rso:lem:vectorization}
    For any \mg{$\lambda\neq 0$ and} $\theta \in \bigg(\frac{1}{\yassine{|\klam|}} (\sqrt{1 + \klam^2} - 1), 1\bigg)$, the solution $\tilde \Sigma^{\infty, \lambda}$ of~\eqref{rso:eq:Lyap_eq_reduced_lambda} satisfies 
    \begin{equation}\label{rso:eq:vecorizedsystem}
        \vectorize\left(\widecheck{\Sigma}^{\infty, \lambda}\right) = \left(I_{4} - \Jlam \otimes \Jlam \right)^{-1} \vectorize(\widecheck{R}^\lambda),
    \end{equation}
    where $ \widecheck{\Sigma}^{\infty, \lambda} \defineq (\Vlam)^{-1} \tilde \Sigma^{\infty, \lambda} \big((\Vlam)^{-1}\big)^{\top}$, $\widecheck{R}^\lambda  \defineq (\Vlam)^{-1} \Rlam \left((\Vlam)^{-1}\right)^\top$ and \mg{$\Jlam, \Vlam$ are the matrices arising in the eigenvalue decomposition of $\Alam$}, as in Lemma~\ref{rso:lem:airjordan}.
\end{lem}

\begin{proof} 
    Given Lemma~\ref{rso:lem:airjordan}, the equation \eqref{rso:eq:Lyap_eq_reduced_lambda} can be rewritten as $\tilde \Sigma^{\infty, \lambda} = \Vlam \Jlam (\Vlam)^{-1} \tilde \Sigma^{\infty, \lambda} ((\Vlam)^{-1})^\top (\Jlam)^\top (\Vlam)^\top + \Rlam$, 
    which amounts to
    \[
        (\Vlam)^{-1}\tilde \Sigma^{\infty, \lambda} ((\Vlam)^{-1})^\top = \Jlam (\Vlam)^{-1} \tilde \Sigma^{\infty, \lambda} ((\Vlam)^{-1})^\top (\Jlam)^\top  + (\Vlam)^{-1} \Rlam ((\Vlam)^{-1})^\top ,
    \]
    or equivalently $\widecheck{\Sigma}^{\infty, \lambda} =  \Jlam \widecheck{\Sigma}^{\infty, \lambda} (\Jlam)^\top + \widecheck{R}^{{}}$.
    Taking $\text{Vec}(\cdot)$ of both sides, and noting the relation $\text{Vec}(ABC) = (C^\top \otimes A) \text{Vec}(B)$, it suffices to show that the $4\times 4$ matrix $I_4 - \Jlam \otimes \Jlam$ is invertible. Observing that $\diag \left[\nuonelam^2, \nuonelam \nutwolam, \nuonelam \nutwolam, \nutwolam^2\right]$,
    \mg{it is sufficient to show that $\nuonelam \nutwolam \neq 1$ for  $\theta \in (\frac{1}{\yassine{|\klam|}} (\sqrt{1 + \klam^2} - 1), 1)$, and this directly follows from the proof of Lemma~\ref{rso:lem:airjordan}}. 
\end{proof} 

\mg{Equipped with the representation \eqref{rso:eq:vecorizedsystem}, we complete the proof of Proposition \ref{rso:prop:limit_covariance_quadratic} in three steps: (I) explicit computation of $\widecheck{R}^\lambda$ , (II) explicit computation of $\widecheck{\Sigma}^{\infty, \lambda}$ from $\widecheck{R}^\lambda$ based on \eqref{rso:eq:vecorizedsystem}, (III) explicit computation of $\tilde \Sigma^{\infty, \lambda}$ from  $\widecheck{\Sigma}^{\infty, \lambda}$ based on the relationship given in Lemma \ref{rso:lem:vectorization}.}
\begin{compactitem}[$\bullet$]
    \item [(I)]\textbf{Computation of $\widecheck{R}^{{}}$.}
        Using the Cramer rule, first observe that $V^{-1}$ satisfies
        \[
            (\Vlam)^{-1}=\frac{1}{A_{1,2}^{{}}\left(\nutwolam-\nuonelam\right)}\left[\begin{array}{ll}
            \left(\theta-\nutwolam\right) & +A_{1,2}^{{}} \\
            -\left(\theta-\nuonelam\right) & -A_{1,2}^{{}}
            \end{array}\right],
        \]
        from which we deduce
        \[
            \widecheck{R}^{(\lambda)} = (\Vlam)^{-1} \Rlam ({(\Vlam)^{-1}})^\top = 
            \frac{\delta^2(1-\theta)^2}{d (A_{1,2}^{\lambda})^2\left(\nuonelam-\nutwolam\right)^2 \mu_x^2 \mu_y^2} \left[\begin{array}{ll}
                Q_{1,1}^{{\lambda}}
                & Q_{1,2}^{{\lambda}} \\
                Q_{1,2}^{{\lambda}}
                & 
                Q_{2,2}^{{\lambda}}
            \end{array}\right],
        \]
        where 
        {\small \begin{equation}\label{rso:eq:def_matrix_Q}
        \begin{split}
             Q_{1,1}^{{\lambda}} &\defineq   \begin{aligned}[t]
                            & \left(\theta-\nutwolam\right)^2 \mu_y^2 \\
                            & + (A_{1,2}^{{\lambda}})^2 \left( 
                                (1-\theta)^2(1+\theta)^2\left(1+2 \theta \frac{\mu_x}{\mu_y}\right) \lambda^2
                                + \mu_x^2\left(1+2\left(1-\theta^2\right) \theta\right) \right) \\
                            & +2 A_{1,2}^{{\lambda}}\left(\theta-\nutwolam\right)\left(1-\theta^2\right)\left(\theta \mu_x+y_y\right) \lambda,  \\
                        \end{aligned} \\
           Q_{2,2}^{{\lambda}} &\defineq \begin{aligned}[t]
                        & \left(\theta-\nuonelam\right)^2 \mu_y^2 \\
                        & +(A_{1,2}^{{\lambda}})^2\left((1-\theta)^2(1+\theta)^2\left(1+2 \theta \frac{\mu_x}{\mu_y}\right) \lambda^2+\mu_x^2\left(1+2\left(1-\theta^2\right) \theta\right)\right) \\
                        & +2 A_{1,2}^{{\lambda}} \left(\theta-\nuonelam\right)\left(1-\theta^2\right)\left(\theta \mu_x+y_y\right) \lambda, \\
                    \end{aligned} \\
             Q_{1,2}^{{\lambda}} &\defineq \begin{aligned}[t] 
                            & -\left(\theta-\nutwolam\right)\left(\theta-\nuonelam\right) \mu_y^2 \\
                            & -(A_{1,2}^{{\lambda}})^2\left((1-\theta)^2(1+\theta)^2\left(1+2 \theta \frac{\mu_x}{\mu_y}\right) \lambda^2+\mu_x{ }^2\left(1+2\left(1-\theta^2\right) \theta\right)\right)\\
                            & +\left(1-\theta^2\right)\left(\theta \mu_x +\muy\right) \lambda A_{1,2}^{{\lambda}} \left(\nuonelam+\nutwolam-2 \theta\right).
                        \end{aligned} \\
        \end{split}
        \end{equation}}
        \item [(II)] \textbf{Computation of $\widecheck{\Sigma}^{\infty, \lambda}$.}
            From the definition of $J^\lambda$ given in \eqref{eq-Jlam}, observing that $\left(I\mg{_4} - \Jlam \otimes \Jlam \right)^{-1} = \diag \left[\frac{1}{1-\nuonelam^2}, \frac{1}{1-\nuonelam \nutwolam}, \frac{1}{1-\nuonelam \nutwolam}, \frac{1}{1-\nutwolam^2} \right]$, 
            \mg{we deduce from Lemma \ref{rso:lem:vectorization} that}
            {\small
            \[
                \widecheck{\Sigma}^{\infty, \lambda} =  \frac{\delta^2(1-\theta)^2}{d (A_{1,2}^{{\lambda}})^2\left(\nuonelam-\nu_2\right)^2 \mu_x^2 \mu_y^2} \left[\begin{array}{ll}
                        \frac{1}{1-\nuonelam^2}\; Q_{1,1}^{{\lambda}}
                    &   \frac{1}{1-\nuonelam \nu_2}\; Q_{1,2}^{{\lambda}} \\
                        \frac{1}{1-\nuonelam \nu_2}\; Q_{1,2}^{{\lambda}}
                    & 
                        \frac{1}{1-\nutwolam^2}\; Q_{2,2}^{{\lambda}}
                \end{array}\right].
            \]}
     \item [(III)] \textbf{Computation of ${\tilde \Sigma}^{\infty, \lambda}$.} 
            From Lemma \ref{rso:lem:vectorization}, we also have
            {\small
            \begin{equation}
                \tilde \Sigma^{\infty, \lambda} = \Vlam \widecheck{\Sigma}^{\infty, \lambda} (\Vlam)^\top = \frac{\delta^2(1-\theta)^2}{d (A_{1,2}^{{\lambda}})^2\left(\nuonelam-\nu_2\right)^2 \mu_x^2 \mu_y^2}  \left[\begin{array}{ll}
                        S_{1,1}^{{\lambda}}
                    &   S_{1,2}^{{\lambda}} \\
                        S_{1,2}^{{\lambda}}
                    & 
                        S_{2,2}^{{\lambda}} \\
                \end{array}\right],
                \label{eq-formula-Sigma-lambda}
            \end{equation}}
            with 
            {\small
            \begin{equation}\label{eq-S-ij}
            \begin{split}
                S_{1,1}^{{\lambda}} &\defineq (A_{1,2}^{{\lambda}})^2
                                            \left(
                                                \frac{Q_{1,1}^{{\lambda}}}{1-\nuonelam^2}
                                                +\frac{Q_{2,2}^{{\lambda}}}{1-\nutwolam^2}
                                                +2 \frac{Q_{1,2}^{{\lambda}}}{1-\nuonelam\nu_2}
                                            \right), \\
                S_{1,2}^{{\lambda}} &\defineq -A_{1,2}^{{\lambda}} 
                                            \left[
                                                \left(\theta-\nuonelam\right) \frac{Q_{1,1}^{{\lambda}} }{1-\nuonelam^2}
                                                +\left(\theta-\nu_2\right) \frac{Q_{2,2}^{{\lambda}} }{1-\nutwolam^2}
                                                +\left(2 \theta-\left(\nuonelam + \nu_2\right)\right) \frac{Q_{1,2}^{{\lambda}} }{1-\nuonelam\nu_2}
                                            \right], \\
                S_{2,2}^{{\lambda}} &\defineq \left(\theta-\nuonelam\right)^2 \frac{Q_{1,1}^{{\lambda}} }{1-\nuonelam^2}+\left(\theta-\nu_2\right)^2 \frac{Q_{2,2}^{{\lambda}} }{1-\nutwolam^2}+2\left(\theta-\nuonelam\right)\left(\theta-\nu_2\right) \frac{Q_{1,2}^{{\lambda}} }{1-\nuonelam\nu_2}. \\
            \end{split}
            \end{equation}}%
\end{compactitem}


While \eqref{eq-formula-Sigma-lambda} provides a formula for $\Sigma^{\lambda,\infty}$, the dependence of this formula to $\klam$ and $\theta$ are not very clear. 
We provide in Section~\ref{rso:sec:supp:quadratic} of this Appendix a simplification of the terms in \eqref{eq-formula-Sigma-lambda} in terms of their dependence to $\klam$ and $\theta$.
As a consequence, in view of~\eqref{eq-formula-Sigma-lambda}, we may write
\[
    \tilde \Sigma^{\infty, \lambda} = \frac{\delta^2(1-\theta)}{d\lambda^2 P_c(\theta, \klam)} \left(\begin{array}{cc}
        \frac{\lambda^2}{\mux^2}\left(\tilde P_{1,1}^{(1)}(\theta, \klam) + \frac{\lambda^2}{\mu_y^2} \tilde P_{1,1}^{(2)}(\theta, \klam) \right)
            & \frac{\lambda}{\mu_x} \left(\tilde P_{1,2}^{(1)}(\theta, \klam) + \frac{\lambda^2}{\mu_y^2} \tilde P_{1,2}^{(2)}(\theta, \klam)\right) \\
        \frac{\lambda}{\mu_x} \left(\tilde P_{1,2}^{(1)}(\theta, \klam) + \frac{\lambda^2}{\mu_y^2} \tilde P_{1,2}^{(2)}(\theta, \klam)\right) 
        & \tilde P_{2,2}^{(1)}(\theta, \klam) + \frac{\lambda^2}{\mu_y^2} \tilde P_{2,2}^{(2)}(\theta, \klam)
    \end{array}\right),
\]
where the $\tilde P_{q,\ell}^{(k)}$ and $P_c$ are polynomials in $\theta, \klam$, defined \ylfourth{in the top part of Table~\ref{rso:table:polynomials}}. 
This proves Proposition~\ref{rso:prop:limit_covariance_quadratic}.

%% file: A4_constants.tex
The convergence analysis of \sapdname~relies on a series of convex inequalities that we wrote in matrix form for compactness. 
All the constants arising in these inequalities (including those mentioned in the statement of 
Lemma~\ref{rso:prop:barbers_lemma} and Lemma\ref{rso:lem:upperbound_hat_variables}) are made explicit as follows in Table~\ref{rso:table:constants_C} and Table~\ref{rso:table:constants_A}. For convenience of the reader, in Section~\ref{rso:sec:supp:constants} of this Appendix, we also provide the expressions of the these constants under the CP parameterization~\eqref{rso:eq:cb_params} which is a particular class of parameters where our complexity results can be achieved. 
We finally detail in Table~\ref{rso:table:polynomials} the polynomials involved in the entries of the covariance matrix $\Sigma^{\infty, \lambda}$ given in Theorem~\ref{rso:thm:quadratics}. 

%
%
\begin{table}[ht]
    {\small
    \begin{tabular}{l}
    \hline \\[-1.ex]
            \vspace{1em} \makecell[l]{
                {$B^x = \frac{4 \tau}{1+\rho}+\frac{\sigma^2(1+\theta)^2 \Lyx^2}{\left(1+\sigma \muy\right)^2} \frac{\rho}{4\left\|A_0\right\|^2(1+\rho)} \frac{3 \tau^2}{\left(1+\tau \mux\right)^2}, \quad C^x= \frac{\tau}{1+\tau \mux}+\frac{\tau \sigma(1+2 \theta) \Lxy}{ {2} \left(1+\tau \mux\right)\left(1+\sigma \muy\right)},  
                    \quad C_{-1}^x = \frac{\tau \sigma \theta}{{2 \rho}(1+\sigma \muy)} \frac{(1+\theta) \Lyx}{1+\tau \mux},
                $}        
            } \\
            \vspace{1em} \makecell[l]{
                {$B^y= \frac{4(1+\theta)^2 \sigma}{(1+\rho)(1-\alpha \sigma)}
                            {
                            + \frac{3\left\|A_0\right\|^2(1+\rho) \theta^2}{\rho^3}
                            }{\scriptsize
                            \begin{aligned}[t]
                                & + \frac{\rho}{4\left\|A_0\right\|^2(1+\rho)} 
                                \frac{\left(1+\sigma(1+\theta) \Lyy\right)^2}{(1+\sigma \muy)^2} \frac{2 \sigma^2(1+\theta)^2}{\left(1+\sigma \muy\right)^2} \\
                                & + \frac{\sigma^2(1+\theta)^2 \Lyx^2}{\left(1+\sigma \muy\right)^2} \frac{\rho}{4\left\|A_0\right\|^2(1+\rho)} \frac{3 \tau^2 \sigma^2(1+\theta)^2 \Lxy^2}{\left(1+\tau \mux\right)^2\left(1+\sigma \muy\right)^2},
                            \end{aligned}}
                            $}
            }\\
            \vspace{1em} \makecell[l]{
                {$B_{-1}^y= 
                    \frac{4 \sigma \theta^2}{{\rho}(1+\rho)(1-\alpha \sigma)}
                    + \frac{\rho}{4\left\|A_0\right\|^2(1+\rho)} 
                        \frac{\left(1+\sigma(1+\theta) \Lyy\right)^2}{\left(1+\sigma \muy\right)^2} 
                        \frac{2 \sigma^2 \theta^2 \rho^{-1}}{\left(1+\sigma {\muy}\right)^2} 
                    + \frac{\sigma^2(1+\theta)^2 \Lyx^2}{\left(1+\sigma \muy\right)^2} \frac{\rho}{4\left\|A_0\right\|^2(1+\rho)} \frac{3 \tau^2 \sigma^2 \theta^2 \rho^{-1} \Lxy^2}{\left(1+\tau \mux\right)^2\left(1+\sigma \muy\right)^2},
                    $}
            }\\
            \vspace{1em} \makecell[l]{
                {$
                    C^y = \frac{\sigma(1+ \ylthird{2} \theta)(1+\ylthird{\theta})}{1+\sigma \muy}+\frac{\tau \sigma(1+\theta) \Lxy}{\ylthird{2} \left(1+\tau \mux\right)\left(1+\sigma \muy\right)}, \quad C_{-2}^y
                    = \frac{\sigma \theta^2}{{2\rho^2}(1+\sigma \muy)} \left(\frac{1+\sigma(1+\theta) \mathrm{L}_{\mathrm{yy}}}{1+\sigma \muy}+\frac{\tau \sigma(1+\theta) \mathrm{L}_{\mathrm{yx}} \mathrm{L}_{\mathrm{xy}}}{\left(1+\tau \mux\right)\left(1+\sigma \muy\right)}\right),
                    $}
            }\\
            \vspace{1em} \makecell[l]{
                {$C_{-1}^y
                    = \frac{\sigma \theta}{\rho(1+\sigma \muy)} 
                        \left( 
                            \ylthird{1 + 2} \theta 
                            + \frac{\tau}{\ylthird{2}(1+\tau \mux)} \left((1+\theta) \Lyx + \Lxy \right) 
                            + (1+ \ylthird{ \frac{3\theta}{2}})
                            \left(\frac{1+\sigma(1+\theta) \mathrm{L}_{\mathrm{yy}}}{1+\sigma \muy}+\frac{\tau \sigma(1+\theta) \mathrm{L}_{\mathrm{yx}} \mathrm{L}_{\mathrm{xy}}}{\left(1+\tau \mux\right)\left(1+\sigma \muy\right)}\right)
                        \right),
                    $}
            }\\
            \vspace{1em} \makecell[l]{
                {$\cQ_x \defineq B^x + C^x + C_{-1}^x, \quad \cQ_y = B^y + {B_{-1}^y} + C^y + C_{-1}^y + C_{-2}^y.
                    $}
            }\\
             \hline
    \end{tabular}}%
    \vspace{.5em}
    \caption{{\small Summary of the constants $B^x, C^x, C_{-1}^x, \cQ^x, B^y, B_{-1}^y, C^y, C_{-1}^y, C_{-2}^y$, and $\cQ^y$ used throughout the analysis.}}
    \label{rso:table:constants_C}
\end{table}
\vspace{-1em}
\begin{table}[ht]
{\small
    \begin{tabular}{l}
    \hline \\[-1.ex]
    \vspace{1em}
        \makecell[l]{$\scalemath{0.95}{
            \hat{A}_1 \defineq \begin{bmatrix}
                    \left(1+\tau \Lxx + \frac{\tau \sigma (1+\theta) \Lyx \Lxy}{1 + \sigma \muy}\right) \\  
                    \frac{\tau \Lxy (1 + \sigma (1+\theta)\Lyy)}{1 + \sigma \muy}\\
                    \yl{\sigma} \tau \theta \frac{\Lxy \Lyx }{(1 + \sigma \muy)}\\
                    \yl{\sigma} \tau \theta \frac{\Lxy \Lyy}{(1 + \sigma \muy)} \\
                \end{bmatrix}, 
            \quad 
                \hat{A}_2 \defineq \begin{bmatrix}
                    \sigma (1+\theta) \Lyx \\
                    1 + \sigma (1+\theta) \Lyy \\
                    \yl{\sigma} \theta \Lyx\\
                    \yl{\sigma} \theta \Lyy\\
                \end{bmatrix}}$,} \\
            \vspace{1em} 
            \makecell[l]{$\hat{A}_3 \defineq \scalemath{0.87}
            {\begin{bmatrix} 
                \left[\begin{aligned}[c]
                    \mg{C_{\sigma,\theta}} \sigma (1+\theta) (1+\tau \mux) \Lyx 
                    & + \sigma (1+\theta) \Lyx \Big(\sa{(}1+\tau \Lxx\sa{)} (1+\sigma \muy) + \tau \sigma (1+\theta) \Lyx \Lxy\Big) \\
                    & + \sigma \theta \Lyx (1+\tau \mux) (1+\sigma \muy)
                \end{aligned}\right] \\
                (1+\tau \mux)\mg{C_{\sigma,\theta}^2} 
                + \sigma (1+\theta) \Lyx \tau \Lxy \mg{C_{\sigma,\theta}} + \sigma \theta \Lyy (1+\tau \mux) (1 + \sigma \muy) \\
                \yl{\sigma} \Big( \mg{C_{\sigma,\theta}} \theta \Lyx (1+\tau \mux) + (1+\theta) \tau \sigma \theta \Lxy \Lyx^2 \Big) \\
                \yl{\sigma} \Big( \mg{C_{\sigma,\theta}} \theta \Lyy (1+\tau \mux) +  (1+\theta) \tau \sigma \theta \Lxy \Lyx \Lyy \Big)\\
            \end{bmatrix}},$} \\ 
            \vspace{1em} \makecell[l]{$A \defineq \begin{bmatrix}
                A_1^\top \\
                A_2^\top \\
                A_3^\top \\
            \end{bmatrix} \defineq \yl{\frac{\sqrt{1 + \rho}}{\sqrt{\rho}}} \begin{bmatrix}
        \frac{\ylsecond{4}}{1+\tau \mux} \hat{A}_1^\top \\
        \frac{\ylsecond{4\sqrt{2}}(1+\theta)}{1+\sigma \muy} \hat{A}_2^\top \\
        \frac{\ylsecond{4\sqrt{2}}\theta\sa{\rho^{-1}}}{(1+\sigma \muy)^2 (1+\tau \mux)} \hat{A}_3^\top \\
    \end{bmatrix} 
            \text{Diag}
                \begin{bmatrix}
                    \sqrt{2\rho \tau} \\ \sqrt{2\rho \sigma/(1-\alpha \sigma)} \\ \sqrt{2\rho \tau} \\ \sqrt{2\rho \sigma/(1-\alpha \sigma)}
                \end{bmatrix}, \quad C_{\sigma,\theta} \defineq 1 + \sigma (1+\theta) \Lyy
                $.} \\ \hline
    \end{tabular}}%
    \vspace{.5em}
    \caption{{\small Summary of the constants $A_1, A_2, A_3$ used throughout the analysis.}}
    \label{rso:table:constants_A}
\end{table}
%

%
%
\begin{table}[ht]
{
\begin{tabular}{l} 
\hline \\[-1.ex]
        \vspace{1em} \makecell[l]{
                \scalebox{0.75}{$
                    \tilde{P}_{1,1}^{(1)}(\theta, \kappa) 
                        = \begin{aligned}[t]
                            &-4 \klam^2 \theta^2\left(1+\theta\right)^2 \\
                            &+ \klam^4\left(
                                      \begin{array}{l}
                                         1+2 \theta-\theta^2 -8 \theta^3 \\
                                         -9 \theta^4 +6 \theta^5+\theta^6  \\
                                    \end{array}
                                    \right) \\
                            &+(1-\theta)^2 \klam^6\left(\theta+4 \theta^2+4 \theta^3-\theta^5\right)
                            \end{aligned}, 
                    \quad  
                    \tilde{P}_{1,1}^{(2)} (\theta, \kappa) 
                        = \begin{aligned}[t]
                            & -4 \klam^2 \theta^2\left(1+2 \theta-\theta^2-2 \theta^3+2 \theta^4\right) \\
                            & +(1-\theta)^2 \klam^4
                                \left(\begin{array}{l}
                                     1+4 \theta+4 \theta^2-6 \theta^3 \\
                                     -11 \theta^4+2 \theta^5+2 \theta^6 \\
                                \end{array}\right) \\
                            & +(1-\theta)^4 \klam^6 \theta(1+\theta)^2(1+2 \theta)
                        \end{aligned}
                ,$}
            }\\
        \vspace{1em} \makecell[l]{
                \scalebox{0.75}{$
                    \tilde{P}_{1,2}^{(1)}(\theta, \kappa) 
                        = \begin{aligned}[t]
                            & -4 \klam^4 \theta^2\left(1 + 2\theta - \theta^3\right) \\
                            & +(1-\theta) \klam^6
                                \left(\begin{array}{l}
                                     1+3 \theta+\theta^2-8 \theta^3 \\
                                     -11 \theta^4+\theta^5+\theta^6 \\ 
                                \end{array}\right) \\
                            & +(1-\theta)^3 \klam^8 \theta(1+\theta)^2(1+2 \theta)
                        \end{aligned}, \quad  
                    \tilde{P}_{1,2}^{(2)} (\theta, \kappa) 
                        = \begin{aligned}[t]
                                &\klam^2 4 \theta^4\left(1 + 2\theta - \theta^3\right) \\
                                &-(1-\theta) \klam^4 \theta^2
                                \left(\begin{array}{l}
                                     5+15 \theta+5 \theta^2-20 \theta^3  \\
                                     -11 \theta^4+9 \theta^5+\theta^6 
                                \end{array} \right) \\
                                &+(1-\theta)^3 \klam^6
                                    \left(\begin{array}{l}
                                         1+5 \theta+8 \theta^2-3 \theta^3  \\
                                         -21 \theta^4-14 \theta^5 \\
                                         +2 \theta^6+2 \theta^7 \\ 
                                    \end{array}\right) \\
                                &+(1-\theta)^{5} \klam^8 \theta(1+\theta)^3(1+2 \theta)
                            \end{aligned},$}
            }\\
        \vspace{1em} \makecell[l]{
                \scalebox{0.75}{$
                    \tilde{P}_{2,2}^{(1)}(\theta, \kappa) 
                        = \begin{aligned}[t]
                              & -4 \klam^6 \theta^2\left(1+2 \theta-\theta^2-2 \theta^3+2 \theta^4\right) \\
                              & +(1-\theta)^2 \klam^8
                                  \left(
                                      \begin{array}{l}
                                           1+4 \theta+4 \theta^2-6 \theta^3  \\
                                           -11 \theta^4+2 \theta^5+2 \theta^6
                                      \end{array}
                                      \right) \\
                              &+(1-\theta)^{4} \klam^{10} \theta(1+\theta)^2(1+2 \theta)
                          \end{aligned}, \quad  
                    \tilde{P}_{2,2}^{(2)} (\theta, \kappa) 
                        = \begin{aligned}[t]
                                & \klam^2 4 \theta^2\left(1+\theta\right)^2\left(-1-2 \theta+2 \theta^3\right) \\
                                & + \klam^4
                                    \left(\begin{array}{l}
                                         1+4 \theta+3 \theta^2 -20 \theta^3 \\
                                         -45 \theta^4-2 \theta^5 +53 \theta^6 \\
                                         +20 \theta^7 -20 \theta^8-2 \theta^9 \\ 
                                    \end{array}\right) \\
                                & +(1-\theta)^2 \klam^6 \theta 
                                    \left(
                                        \begin{array}{cc}
                                             3+14 \theta+20 \theta^2\\
                                             -8 \theta^3-47 \theta^4  \\
                                             -30 \theta^5+4 \theta^6+4 \theta^7 \\
                                        \end{array}
                                        \right) \\
                                & +(1-\theta)^{4} \klam^8 2 \theta^2(1+\theta)^3(1+2 \theta)
                            \end{aligned}
                ,$}
            }\\
        \hline\\
        \vspace{1em} \makecell[l]{
                \scalebox{0.75}{$P_{1,1}^{(\infty, 1)}(\theta, \kappa) = \begin{aligned}[t]
                    &  -4 \kappa^2 \theta^4 (1+\theta)^2 \\
                    &+\kappa^4 (\theta^2+10 \theta^3+7 \theta^4-24 \theta^5-17 \theta^6+14 \theta^7+\theta^8) \\
                    & -\kappa^6 (1-\theta)^2 \theta (2+10 \theta+9 \theta^2-24 \theta^3-34 \theta^4+10 \theta^5+3 \theta^6) \\
                    & +\kappa^8 (1-\theta)^4 (1+4 \theta+2 \theta^2-14 \theta^3-21 \theta^4-2 \theta^5+2 \theta^6) \\
                    & +\kappa^{10} (1-\theta)^6 \theta (1+\theta)^2 (1+2 \theta)
                \end{aligned}, \quad  
                P_{1,1}^{(\infty, 2)} (\theta, \kappa) = 
                \begin{aligned}[t]
                    &\kappa^2 \left(\begin{array}{l}
                         -4 \theta^2-8 \theta^3+4 \theta^4 \\
                         +8 \theta^5-8 \theta^6 \\                    
                    \end{array}\right) \\
                    &+ \kappa^4 (1-\theta)^2 
                        \left(
                            \begin{array}{l}
                                1+4 \theta+4 \theta^2-6 \theta^3    \\
                                -11 \theta^4+2 \theta^5+2 \theta^6
                            \end{array} \right) \\
                    &+ \kappa^6 (1-\theta)^4 \theta (1+\theta)^2 (1+2 \theta) \\
                \end{aligned},$}
            }\\
            \vspace{1em} \makecell[l]{
        \scalebox{0.75}{$P_{1,2}^{(\infty, 1)} (\theta, \kappa) = \begin{aligned}[t]
            &+4 \kappa^4 \theta^3 (-1-2 \theta+\theta^3) \\
            &+\kappa^6 (1-\theta) \theta (1+3 \theta+5 \theta^2-15 \theta^4-7 \theta^5+9 \theta^6) \\
            &-\kappa^8 (1-\theta)^3 \theta (1+3 \theta-11 \theta^3-13 \theta^4+2 \theta^5+2 \theta^6) \\
            &-\kappa^{10} (1-\theta)^5 \theta^2 (1+\theta)^2 (1+2 \theta) \\
        \end{aligned}, \quad 
        P_{1,2}^{(\infty, 2)} (\theta, \kappa) =
        \begin{aligned}[t]
            &\kappa^2 (4 \theta^3+12 \theta^4+8 \theta^5-12 \theta^6-16 \theta^7+4 \theta^8+8 \theta^9) \\
            &+\kappa^4 (1-\theta) \theta \left(\begin{array}{l}
                -1-4 \theta-8 \theta^2+5 \theta^3+40 \theta^4+22 \theta^5  \\
                -42 \theta^6-29 \theta^7+19 \theta^8+2 \theta^9
            \end{array} \right) \\
            &+\kappa^6 (1-\theta)^3 \left(\begin{array}{l}
                \theta+2 \theta^2-6 \theta^3-23 \theta^4-13 \theta^5+33 \theta^6  \\
                +32 \theta^7-2 \theta^8-4 \theta^9 \\
            \end{array}\right) \\
            &+\kappa^8 (1-2 \theta) (1-\theta)^5 \theta^2 (1+\theta)^3 (1+2 \theta) \\
        \end{aligned},$}} \\
        \vspace{1em}\makecell[l]{
            \scalebox{0.75}{$P_{2,2}^{\infty, 1}(\theta, \kappa) = \begin{aligned}[t]
                &-4 \kappa^6 \theta^2\left(1+2 \theta-\theta^2-2 \theta^3+2 \theta^4\right) \\
                &+(1-\theta)^2 \kappa^8
                    \left(
                        \begin{array}{l}
                             1+4 \theta+4 \theta^2-6 \theta^3  \\
                             -11 \theta^4+2 \theta^5+2 \theta^6
                        \end{array}
                        \right) \\
                &+(1-\theta)^{4} \kappa^{10} \theta(1+\theta)^2(1+2 \theta)
            \end{aligned}, \quad 
            P_{2,2}^{\infty, 2}(\theta, \kappa) = \begin{aligned}[t]
                    &\kappa^2 4 \theta^2\left(1+\theta\right)^2\left(-1-2 \theta+2 \theta^3\right) \\
                    &+ \kappa^4
                        \left(\begin{array}{l}
                             1+4 \theta+3 \theta^2 -20 \theta^3 -45 \theta^4-2 \theta^5 +53 \theta^6 +20 \theta^7 -20 \theta^8-2 \theta^9 \\ 
                        \end{array}\right) \\
                    &+(1-\theta)^2 \kappa^6 \theta 
                        \left(
                            \begin{array}{cc}
                                 3+14 \theta+20 \theta^2 -8 \theta^3 -47 \theta^4   -30 \theta^5+4 \theta^6+4 \theta^7 \\
                            \end{array}
                            \right) \\
                    &+(1-\theta)^{4} \kappa^8 2 \theta^2(1+\theta)^3(1+2 \theta)
                \end{aligned},$}
        } \\
        \vspace{1em}\makecell[l]{
            \scalebox{0.8}{$P_{c}(\theta, \kappa) = \begin{aligned}[t]
                    &- 4 \kappa^2 \theta^2\left(1+\theta\right)^3 + \kappa^4\left(1 + 3\theta + \theta^2 - 17 \theta^3 - 33 \theta^4 - 3 \theta^5 + 15 \theta^6 + \theta^7\right) \\
                    &+(1-\theta) \kappa^6 \theta\left(3+14 \theta+13 \theta^2-24 \theta^3-35 \theta^4+10 \theta^5+3 \theta^6\right) \\
                    &+(1-\theta)^3 \kappa^8\left(-1-4 \theta-2 \theta^2+14 \theta^3+21 \theta^4+2 \theta^5-2 \theta^6\right) \\
                    &-(1-\theta)^{5} \kappa^{10} \theta(1+\theta)^2(1+2 \theta) \\
                \end{aligned}.$}
        }\\
            \hline
\end{tabular}}
\vspace{.5em}
\caption{\label{rso:table:polynomials} Polynomials involved in the description of the equilibrium covariance matrices $\tilde \Sigma^{\infty}$ of $\lim_{n\to \infty} [x_{n-1}, y_n]$ (top) and $\Sigma^{\infty}$ of $\lim_{n\to \infty} [x_{n}, y_n]$ (bottom).}
\end{table}

%% file: B_online_material.tex

\makeatletter
\patchcmd{\@sect}{\@Alph\c@section}{\@arabic\c@section}{}{}
\makeatother


\section{Supplementary derivations for the quadratic case}\label{rso:sec:supp:quadratic}
\input{B1_supp_quadratic}

\section{Proof of Corollary~\ref{rso:thm:complexity_sapd_result}\label{sec-complexity-proof}}
\input{B2_complexity}

\section{Tables of constants under the CP parameterization}\label{rso:sec:supp:constants}
\input{B3_constants}

%% file: B1_supp_quadratic.tex
\mgtwo{In this section, we outline how the constants $S_{1,1}^{{\lambda}}$, $S_{2,2}^{{\lambda}}$ and $S_{1,2}^{{\lambda}}$ that are defined by \eqref{eq-S-ij} in the proof of Proposition \ref{rso:prop:limit_covariance_quadratic} can be simplified and reorganized under a common denominator that is related to the polynomial $P_c(\theta, \klam)$ given in Table \ref{rso:table:polynomials} of Appendix \ref{rso:sec:constants}. This requires tedious but otherwise straightforward computations as follows:}

\paragraph{(I) Simplification of $S_{1,1}^{{\lambda}}$.}

\mg{Given the expressions \eqref{rso:eq:def_matrix_Q}} and \eqref{eq-S-ij}, we have
\[
\begin{split}
    S_{1,1}^{{\lambda}} = & 
        \frac{(1-\theta)^2 \lambda^2}{\mu_x^2} \frac{1}{\left(1-\nuonelam^2\right)\left(1-\nutwolam^2\right)\left(1-\nuonelam \nu_2\right)} \times \\
        &   \scalemath{0.8}{\left(\begin{array}{c}
                \left(1-\nutwolam^2\right)\left(1-\nuonelam \nu_2\right) 
                    \left(\begin{array}{l}
                        \left(\theta -\nu_2\right)^2 \mu_y^2 \\
                        + (A_{1,2}^{{\lambda}})^2\left((1-\theta)^2(1+\theta)^2\left(1+2 \theta \frac{\mu_x}{\mu_y}\right) \lambda^2+\mu_{x}^2\left(1+2\left(1-\theta^2\right) \theta\right)\right) \\
                        +2 A_{1,2}^{{\lambda}} \left(\theta -\nu_2\right)\left(1-\theta^2\right)\left(\theta \mu_x+\mu_y\right) \lambda \\
                    \end{array}\right) \\
                + 2\left(1-\nuonelam^2\right)\left(1-\nutwolam^2\right)  
                    \left(\begin{array}{l} 
                        -\left(\theta -\nu_2\right)\left(\theta -\nuonelam\right) \mu_y^2 \\
                        - (A_{1,2}^{{\lambda}})^2\left((1-\theta)^2(1+\theta)^2\left(1+2 \theta \frac{\mu_x}{\mu_y}\right) \lambda^2+\mu_{x}^2\left(1+2\left(1-\theta^2\right) \theta\right)\right) \\
                        +\left(1-\theta^2\right)\left(\theta \mu_x + \muy \right) \lambda A_{1,2}^{{\lambda}} \left(\nuonelam+\nutwolam-2 \theta \right)
                    \end{array}\right) \\
                + \left(1-\nuonelam \nutwolam\right)\left(1-\nuonelam^2\right) 
                    \left(\begin{array}{l}
                        \left(\theta -\nuonelam\right)^2 \mu_y^2 \\
                        + (A_{1,2}^{{\lambda}})^2\left((1-\theta)^2(1+\theta)^2\left(1+2 \theta \frac{\mu_x}{\mu_y}\right) \lambda^2+\mu_{x}^2\left(1+2\left(1-\theta^2\right) \theta\right)\right) \\
                        +2 A_{1,2}^{{\lambda}} \left(\theta -\nuonelam\right)\left(1-\theta^2\right)\left(\theta \mu_x+\mu_y\right) \lambda \\
                    \end{array}\right) \\
            \end{array}\right)}, \\
\end{split}
\]
\ylfourth{where $A^\lambda$, defined in~\eqref{rso:eq:a_lam_r_lam}, admits the expression~\eqref{rso:eq:a_lam_r_lam_under_cb} under the CP parameterization, and $\nu_{1, \lambda}, \nu_{2, \lambda}$ are its eigenvalues, provided Lemma~\ref{rso:lem:airjordan}.}
First note that 
\[
    \begin{split}
        & \mu_y^2 \bigg(
                \begin{aligned}[t]
                    & \left(1-\nuonelam \nutwolam\right)
                    \left(\left(\theta -\nutwolam\right)^2\left(1-\nutwolam^2\right)+\left(\theta -\nuonelam\right)^2\left(1-\nuonelam^2\right)\right) \\
                    & -2\left(1-\nuonelam^2\right)\left(1-\nutwolam^2\right)\left(\theta -\nuonelam\right)\left(\theta -\nutwolam\right) \bigg)
                \end{aligned} \\
        & = \mu_y^2 \left(
                \begin{array}{l}
                    \left(\nuonelam^2+\nutwolam^2\right)\left(1+\theta^2\right)-\left(\nuonelam^4+(\nutwolam)^4\right) \\
                    + \nuonelam \nutwolam \left(\begin{array}{l}
                            \left(\nuonelam^2+\nutwolam^2\right) \left(1+\theta^2\right) + \left(\nuonelam^4+(\nutwolam)^4\right) 
                            -2 \theta\left(\nuonelam^3+(\nutwolam)^3\right)  \\
                            -2 (1+\theta^2)
                            -2 \nuonelam \nutwolam  \theta^2
                            +2 \theta \nuonelam \nutwolam \left(\nuonelam+\nutwolam\right)
                            -2\left(\nuonelam \nutwolam\right)^2 \\
                    \end{array}\right)
                \end{array} \right)
             \\
        &=  \mu_y^2 \left(
                \begin{array}{l}
                 -4(1-\theta)^2 \theta^2\left(1-\theta^2\right)^2 \klam^2 \\
                 +(1-\theta)^4\left(1+2 \theta-\theta^2-16 \theta^3-17 \theta^4+14 \theta^5+9 \theta^6\right) \klam^4 \\
                 +(1-\theta)^6 \theta\left(3+14 \theta+16 \theta^2-12 \theta^3-23 \theta^4-6 \theta^5\right) \klam^6 \\
                 +(1-\theta)^8(1+\theta)^2\left(-1-2 \theta+2 \theta^2+8 \theta^3+\theta^4\right) \klam^8 \\
                 -(1-\theta)^{10} \theta(1+\theta)^4 \klam^{10}
                 \end{array} \right),
    \end{split}
\]
where the last line can be deduced from Lemma~\ref{rso:lem:polynomial_root}. 
Second, we observe that
\[
    \begin{split}
        \left(1-\nutwolam^2\right)& \left(1-\nuonelam \nutwolam\right)-2\left(1-\nuonelam^2\right)\left(1-\nutwolam^2\right)+\left(1-\nuonelam \nutwolam\right)\left(1-\nuonelam^2\right) \\
        & = \left(\begin{array}{l}
            (1-\theta)^2 \klam^2\left(2 \theta-2 \theta(1+2 \theta)-2 \theta^3(1+2 \theta)+2 \theta^3\right) \\
    +(1-\theta)^4 \klam^4\left(\left(1+\theta^2\right)(1+\theta)^2+2 \theta^2(1+2 \theta)-2 \theta^2\right) \\
            -(1-\theta)^6 \klam^6 \theta(1+\theta)^2 \\
    \end{array}\right),
    \end{split}
\]
and 
\begin{equation}\label{rso:eq:quad:exp_factor_1}
    \begin{split}
        (A_{1,2}^{{\lambda}})^2 
            &\left((1-\theta)^2(1+\theta)^2\left(1+2 \theta \frac{\mu_x}{\mu_y}\right) \lambda^2+\mu_x^2\left(1+2\left(1-\theta^2\right) \theta\right)
            \right) \\
            &= \mu_y^2\left((1-\theta)^4 \klam^4(1+\theta)^2+\left(2 \theta(1-\theta)^4 \klam^2(1+\theta)^2+(1-\theta)^2(1+2(1-\theta^2) \theta)\right) \frac{\lambda^2}{\mu_y^2}\right), \\     
    \end{split}
\end{equation}
and finally, noting that 
\begin{equation}\label{rso:eq:quad:exp_factor_2}
    A_{1,2}^{{\lambda}} \left(1-\theta^2\right)(\theta \mu_x+ \mu_y) \lambda = - \mu_y^2\; \left((1-\theta)^2(1+\theta) \theta \frac{\lambda^2}{\mu_y^2}+(1-\theta)^2(1+\theta)  \klam^2 \right),
\end{equation}
and using again Lemma~\ref{rso:lem:polynomial_root}, we obtain
{\small \[
    \begin{split}
        &\left(\begin{array}{l}
            \left(1-\nutwolam^2\right)\left(1-\nuonelam \nutwolam\right) \quad 2 A_{1,2}^{{\lambda}} \left(\theta -\nutwolam\right)\left(1-\theta^2\right)\left(\theta \mu_x+\mu_y\right) \lambda \\
             +2\left(1-\nuonelam^2\right)\left(1-\nutwolam^2\right)\left(1-\theta^2\right)\left(\theta \mu_x+ \mu_y\right) \lambda A_{1,2}^{{\lambda}} \left(\nuonelam+\nutwolam-2 \theta \right)   \\
             +\left(1-\nuonelam \nutwolam\right)\left(1-\nuonelam^2\right) \quad 2 A_{1,2}^{{\lambda}} \left(\theta -\nuonelam\right)\left(1-\theta^2\right)\left(\theta \mu_x+ \mu_y\right) \lambda
                \\
        \end{array}\right) \\
        &= 2 A_{1,2}^{{\lambda}} \left(1-\theta^2\right)(\theta \mu_x+ \mu_y) \lambda 
            \left(\begin{array}{l}
                \theta   \left(\nuonelam^2 + \nutwolam^2\right) - 2 \theta \nuonelam \nutwolam + \theta   \nuonelam \nutwolam  \left(\nuonelam^2 + \nutwolam^2\right) \\
                - \nuonelam \nutwolam \left(\nuonelam^3 + (\nutwolam)^3\right)- 2 \theta   \left(\nuonelam \nutwolam\right)^2 + \left(\nuonelam \nutwolam\right)^2   \left(\nuonelam+\nutwolam\right)
            \end{array}\right) \\
        &= \mu_y^2\; \left(2(1-\theta)^2(1+\theta) \theta \frac{\lambda^2}{\mu_y^2}+(1-\theta)^2(1+\theta) 2 \klam^2 \right)  \left(\begin{array}{l}
                - 4 \theta^3 (1-\theta)^3 \klam^2 (1+\theta) \\
                + \klam^4 (1-\theta)^3 \theta (1+\theta - 2 \theta^2-10 \theta^3 + 5 \theta^4 + 5 \theta^5) \\
                +\klam^6 (1-\theta)^6 \theta^2 (2+ 9 \theta+8 \theta^2+\theta^3) \\
                -\klam^8 (1-\theta)^8 \theta (1+\theta)^3 \\
            \end{array}\right).
    \end{split} 
\]}
Hence, grouping together the terms which have a $\lambda^2/\muy^2$ factor and those which do not, we obtain
{\small
\begin{equation}\label{rso:eq:simplification_s11}
    \begin{split}
        & S_{1,1}^{{\lambda}} = 
            \frac{(1-\theta)^6 \lambda^2 \muy^2}{\mu_x^2 \left(1-\nuonelam^2\right)\left(1-\nutwolam^2\right)\left(1-\nuonelam \nutwolam\right)} \times \\
            &\left(
                \underbrace{
                    \left(\begin{array}{l}
                    -4 \klam^2 \theta^2\left(1+\theta\right)^2 \\
                    + \klam^4\left(\begin{array}{l}
                         1+2 \theta-\theta^2 -8 \theta^3 \\
                         -9 \theta^4 +6 \theta^5+\theta^6  \\
                    \end{array}\right) \\
                    +(1-\theta)^2 \klam^6\left(\theta+4 \theta^2+4 \theta^3-\theta^5\right)
                    \end{array}\right)
                }_{\tilde P_{1,1}^{(1)}(\theta, \klam) \defineq} 
            + \frac{\lambda^2}{\mu_y^2} 
                \underbrace{\left(\begin{array}{l}
                    -4 \klam^2 \theta^2\left(1+2 \theta-\theta^2-2 \theta^3+2 \theta^4\right) \\
                    +(1-\theta)^2 \klam^4
                        \left(\begin{array}{l}
                             1+4 \theta+4 \theta^2-6 \theta^3 \\
                             -11 \theta^4+2 \theta^5+2 \theta^6 \\
                        \end{array}\right) \\
                    +(1-\theta)^4 \klam^6 \theta(1+\theta)^2(1+2 \theta)
                \end{array}\right)}_{\tilde P_{1,1}^{(2)}(\theta, \klam) \defineq} 
            \right). \\
    \end{split}
\end{equation}}
\mg{where $\tilde P_{1,1}^{(1)}$ and $\tilde P_{1,1}^{(2)}$ are polynomials in $\theta$ and $\klam$}.

\paragraph{(II) Simplification of $S_{2,2}^{{\lambda}}$.}

Similarly, from \eqref{rso:eq:def_matrix_Q} and \eqref{eq-S-ij}, we have
{\small \[
\begin{split}
    S_{2,2}^{{\lambda}} = & 
        \frac{1}{\left(1-\nuonelam^2\right)\left(1-\nutwolam^2\right)\left(1-\nuonelam \nutwolam\right)} \times \\
        &   \scalemath{0.8}{\left(\begin{array}{c}
                \muy^2 
                    \left(\begin{array}{l}
                        \left(\theta-\nuonelam\right)^2\left(1-\nutwolam^2\right)\left(1-\nuonelam \nutwolam\right)\left(\theta-\nutwolam\right)^2 \\
                        \left(\theta-\nutwolam\right)^2\left(1-\nuonelam \nutwolam\right)\left(1-\nuonelam^2\right)\left(\theta-\nuonelam\right)^2 \\
                        -2\left(\theta-\nuonelam\right)\left(\theta-\nutwolam\right)\left(1-\nuonelam^2\right)\left(1-\nutwolam^2\right)\left(\theta-\nutwolam\right)\left(\theta-\nuonelam\right) \\
                    \end{array}\right) \\
                + (A_{1,2}^{{\lambda}})^2\left((1-\theta)^2(1+\theta)^2\left(1+2 \theta \frac{\mu_x}{\mu_y}\right) \lambda^2+\mu_x^2\left(1+2\left(1-\theta^2\right) \theta\right)\right)  
                    \left(\begin{array}{l} 
                        \left(\theta-\nuonelam\right)^2\left(1-\nutwolam^2\right)\left(1-\nuonelam \nutwolam\right)\\
                        +\left(\theta-\nutwolam\right)^2\left(1-\nuonelam \nutwolam\right)\left(1-\nuonelam^2\right) \\
                        -2\left(\theta-\nuonelam\right)\left(\theta-\nutwolam\right)\left(1-\nuonelam^2\right)\left(1-\nutwolam^2\right) \\
                    \end{array}\right) \\
                + A_{1,2}^{{\lambda}} \left(1-\theta^2\right)\left(\theta \mu_x+\mu_y\right) \lambda 
                    \left(\begin{array}{l}
                        2\left(\theta -\nutwolam\right) \quad\left(\theta-\nuonelam\right)^2\left(1-\nutwolam^2\right)\left(1-\nuonelam \nutwolam\right) \\
                        +2\left(\theta -\nuonelam\right)\left(\theta-\nutwolam\right)^2\left(1-\nuonelam \nutwolam\right)\left(1-\nuonelam^2\right) \\
                        +\left(\nuonelam+\nutwolam-2 \theta \right) 2\left(\theta-\nuonelam\right)\left(\theta-\nutwolam\right)\left(1-\nuonelam^2\right)\left(1-\nutwolam^2\right) \\
                    \end{array}\right) \\
            \end{array}\right)}, \\
\end{split}
\]}
which, combined with~\eqref{rso:eq:quad:exp_factor_1} and~\eqref{rso:eq:quad:exp_factor_2}, gives
{\small \[
\begin{split}
    S_{2,2}^{{\lambda}} = & 
        \frac{\muy^2}{\left(1-\nuonelam^2\right)\left(1-\nutwolam^2\right)\left(1-\nuonelam \nutwolam\right)} \times \\
        &   \scalemath{0.8}{\left(\begin{array}{c}
                 \left(\begin{array}{l}
                        \left(\theta-\nuonelam\right)^2\left(\theta-\nutwolam\right)^2 \left(\left(1-\nuonelam \nutwolam\right)\left(2-\left(\nuonelam^2+\nutwolam^2\right)\right)-2\left(1-\nuonelam^2\right)\left(1-\nutwolam^2\right)\right) \\
                    \end{array}\right) \\
                +  
                \left(\begin{array}{l}
                        (1-\theta)^4 \klam^4(1-\theta)^2 \\
                        + \frac{\lambda^2}{\mu_y^2} \left(2 \theta(1-\theta)^4 \klam^2(1+\theta)^2+(1-\theta)^2(1+2(1-\theta) \theta)\right)       
                \end{array} \right)
                    \left(\begin{array}{l} 
                        \left(\theta-\nuonelam\right)^2\left(1-\nutwolam^2\right)\left(1-\nuonelam \nutwolam\right)\\
                        +\left(\theta-\nutwolam\right)^2\left(1-\nuonelam \nutwolam\right)\left(1-\nuonelam^2\right) \\
                        -2\left(\theta-\nuonelam\right)\left(\theta-\nutwolam\right)\left(1-\nuonelam^2\right)\left(1-\nutwolam^2\right) \\
                    \end{array}\right) \\
                - \left((1-\theta)^2(1+\theta) \theta \frac{\lambda^2}{\mu_y^2}+(1-\theta)^2(1+\theta) \klam^2 \right)
                    \left(\begin{array}{l}
                        2\left(\theta-\nutwolam\right) \quad\left(\theta-\nuonelam\right)^2\left(1-\nutwolam^2\right)\left(1-\nuonelam \nutwolam\right) \\
                        +2\left(\theta-\nuonelam\right)\left(\theta-\nutwolam\right)^2\left(1-\nuonelam \nutwolam\right)\left(1-\nuonelam^2\right) \\
                        +\left(\nuonelam+\nutwolam-2 \theta\right) 2\left(\theta-\nuonelam\right)\left(\theta-\nutwolam\right)\left(1-\nuonelam^2\right)\left(1-\nutwolam^2\right) \\
                    \end{array}\right) \\
            \end{array}\right)}. \\
\end{split} 
\]}
Finally, using Lemma~\ref{rso:lem:polynomial_root}, and regrouping terms with a $\lambda^2/\muy^2$ factor, we obtain
{\small
\begin{equation}\label{rso:eq:simplification_s22}
    \begin{split}
        S_{2,2}^{{\lambda}} = &
            \frac{\muy^2 (1-\theta)^6}{\left(1-\nuonelam^2\right)\left(1-\nutwolam^2\right)\left(1-\nuonelam \nutwolam\right)} \times \\
            & \scalemath{0.8}{\left(
                \underbrace{\left(\begin{array}{l}
                    -4 \klam^6 \theta^2\left(1+2 \theta-\theta^2-2 \theta^3+2 \theta^4\right) \\
                    +(1-\theta)^2 \klam^8
                        \left(
                            \begin{array}{l}
                                 1+4 \theta+4 \theta^2-6 \theta^3  \\
                                 -11 \theta^4+2 \theta^5+2 \theta^6
                            \end{array}
                            \right) \\
                    +(1-\theta)^{4} \klam^{10} \theta(1+\theta)^2(1+2 \theta)
                \end{array}\right)}_{\tilde P_{2,2}^{(1)}(\theta, \klam) \defineq}
                +\frac{\lambda^2}{\mu^2}
                    \underbrace{\left(\begin{array}{l}
                        \klam^2 4 \theta^2\left(1+\theta\right)^2\left(-1-2 \theta+2 \theta^3\right) \\
                        + \klam^4
                            \left(\begin{array}{l}
                                 1+4 \theta+3 \theta^2 -20 \theta^3 \\
                                 -45 \theta^4-2 \theta^5 +53 \theta^6 \\
                                 +20 \theta^7 -20 \theta^8-2 \theta^9 \\ 
                            \end{array}\right) \\
                        +(1-\theta)^2 \klam^6 \theta 
                            \left(
                                \begin{array}{cc}
                                     3+14 \theta+20 \theta^2\\
                                     -8 \theta^3-47 \theta^4  \\
                                     -30 \theta^5+4 \theta^6+4 \theta^7 \\
                                \end{array}
                                \right) \\
                        +(1-\theta)^{4} \klam^8 2 \theta^2(1+\theta)^3(1+2 \theta)
                    \end{array}\right)}_{\tilde P_{2,2}^{(2)}(\theta, \klam) \defineq} \right).} \\
    \end{split}
\end{equation}}
%

\paragraph{(III) Simplification of $S_{1,2}^\lambda$.} By going through similar steps, we obtain

%
{\small
\begin{equation}\label{rso:eq:simplification_s12}
\begin{split}
    S_{1,2}^{{\lambda}} = & \frac{\mu_y^2}{\mu_x} \frac{(1-\theta)^6 \lambda}{\left(1-\nuonelam^2\right)\left(1-\nutwolam^2\right)\left(1-\nuonelam \nutwolam\right)} \\
    & \scalemath{0.8}{\left(
        \underbrace{\left(\begin{array}{l}
            -4 \klam^4 \theta^2\left(1 + 2\theta - \theta^3\right) \\
            +(1-\theta) \klam^6
                \left(\begin{array}{l}
                     1+3 \theta+\theta^2-8 \theta^3 \\
                     -11 \theta^4+\theta^5+\theta^6 \\ 
                \end{array}\right) \\
            +(1-\theta)^3 \klam^8 \theta(1+\theta)^2(1+2 \theta)
        \end{array}\right)}_{\tilde P_{1,2}^{(1)}(\theta, \klam) \defineq}
        +\frac{\lambda^2}{\mu^2}
        \underbrace{\left(\begin{array}{c}
            \klam^2 4 \theta^4\left(1 + 2\theta - \theta^3\right) \\
            -(1-\theta) \klam^4 \theta^2
                \left(\begin{array}{l}
                     5+15 \theta+5 \theta^2-20 \theta^3  \\
                     -11 \theta^4+9 \theta^5+\theta^6 
                \end{array} \right) \\
            +(1-\theta)^3 \klam^6
                \left(\begin{array}{l}
                     1+5 \theta+8 \theta^2-3 \theta^3  \\
                     -21 \theta^4-14 \theta^5 \\
                     +2 \theta^6+2 \theta^7 \\ 
                \end{array}\right) \\
            +(1-\theta)^{5} \klam^8 \theta(1+\theta)^3(1+2 \theta)
        \end{array}\right)}_{\tilde P_{1,2}^{(2)}(\theta, \klam) \defineq} 
        \right).}
\end{split}
\end{equation}}
Finally, we notice that if we write  $S_{i,j}^{{\lambda}}$ for $i=1,2$ and $j=1,2$ in a common denominator, the following term would arise:
{\small 
\begin{equation}\label{rso:eq:simplification_common_factor}
\begin{split}
\tiny
    & \frac{(1-\theta)^2 \muy^2}{ (A_{1,2}^{{\lambda}})^2\left(\nuonelam-\nutwolam\right)^2 \mu_x^2 \mu_y^2}\frac{1}{\left(1-\nuonelam^2\right)\left(1-\nutwolam^2\right)\left(1-\nuonelam \nutwolam\right)} \\
    &= \frac{\frac{1}{ \lambda^2} }{\left(\nuonelam-\nutwolam\right)^2 \left(1-\nuonelam^2\right)\left(1-\nutwolam^2\right)\left(1-\nuonelam \nutwolam\right)} \\
    &= \frac{1}{ \lambda^2 (1-\theta)^5} 
            \frac{1}{
                \underbrace{\left(\begin{array}{l}
                    - 4 \klam^2 \theta^2\left(1+\theta\right)^3 \\
                    + \klam^4\left(1 + 3\theta + \theta^2 - 17 \theta^3 - 33 \theta^4 - 3 \theta^5 + 15 \theta^6 + \theta^7\right) \\
                    +(1-\theta) \klam^6 \theta\left(3+14 \theta+13 \theta^2-24 \theta^3-35 \theta^4+10 \theta^5+3 \theta^6\right) \\
                    +(1-\theta)^3 \klam^8\left(-1-4 \theta-2 \theta^2+14 \theta^3+21 \theta^4+2 \theta^5-2 \theta^6\right) \\
                    -(1-\theta)^{5} \klam^{10} \theta(1+\theta)^2(1+2 \theta)
                \end{array}\right)}_{P_c(\theta, \klam) \defineq}
    }.
\end{split}
\end{equation}}
%

%% file: B2_complexity.tex
\mgtwo{Our choice of parameters is a special case of the CP parameterization derived in \citep[Corollary 1]{zhang2021robust}.} Let us first assume that $\Lyy > 0$. 
Then, by~\citep[Corollary 1]{zhang2021robust}, $(\rho, \tau, \sigma, \theta, \alpha)$, with $\rho=\theta$ and $\alpha = \frac{1}{2\sigma} - \sqrt{\theta} \Lyy$ is a solution of the matrix inequality~\eqref{rso:eq:matrix_inequality_params}. 
\mgtwo{Thus, our Theorem~\ref{rso:thm:high_probability_bound_sapd}  is applicable}. 
In particular, noting that $(1-\alpha \sigma)/ (2\sigma) \geq 1 / (4\sigma)$, we have
\[
\begin{split}
    \mathcal{D}_n = \frac{1}{2 \tau}\left\|x_n-\optx\right\|^2+\frac{1-\alpha \sigma}{2 \sigma}\left\|y_n-\opty\right\|^2 
        & \geq \frac{1}{4 \tau}\left\|x_n-\optx\right\|^2+\frac{1}{4 \sigma}\left\|y_n-\opty\right\|^2 \\
        & \geq \frac{1}{4} \frac{\theta}{(1-\theta)}\left(\mu_x\left\|x_n-\optx\right\|^2+\mu_y\left\|y_n-\opty\right\|^2\right).
\end{split}
\]
Therefore, Theorem~\ref{rso:thm:high_probability_bound_sapd} \ylsecond{gives}
\[
    \begin{split}
        Q_p\bigg(\mu_x\left\|x_n- \optx \right\|^2 &+ \muy \|y_n- \opty \|^2\bigg) \\
        & \leq 
        \frac{4 (1-\theta)}{\theta} 
            \left(
                \left(\frac{1+\theta}{2}\right)^n \left(\cC_{\tau, \sigma , \theta}D_{\tau, \sigma}  + \Xi_{\tau, \sigma, \theta}^{(1)} \right) 
                + \Xi_{\tau, \sigma, \theta}^{(1)}
                + \Xi_{\tau, \sigma, \theta}^{(2)} \log\left(\frac{1}{1-p}\right)
            \right).
    \end{split}
\]
Let us fix $\varepsilon > 0$. To show the claimed iteration complexity bound, it suffices to show that the right hand-side is at most $\varepsilon$ for the given choice of parameters. 
For this purpose, sufficient conditions are
\begin{align}
     \frac{(1-\theta)}{\theta} \left(\frac{1+\theta}{2}\right)^n  C_{\tau, \sigma, \theta} D_{\tau, \sigma} 
        &\leq \frac{\varepsilon}{16}, \label{rso:eq:complexity_bias_term}\\ 
     \ylsecond{\frac{(1-\theta)}{\theta}}\left(\frac{1+\theta}{2}\right)^n \Xi_{\tau, \sigma, \theta}^{(1)} 
        &\leq \frac{\varepsilon}{16}, \label{rso:eq:complexity_xi1}\\
    \frac{(1-\theta)}{\theta} \Xi_{\tau, \sigma, \theta}^{(2)} 
        &\leq \frac{\varepsilon}{16}, \label{rso:eq:complexity_xi2} \\ 
    \frac{(1-\theta)}{\theta} \Xi_{\tau, \sigma, \theta}^{(3)} \log \left(\frac{1}{1-p}\right) 
        &\leq \frac{\varepsilon}{16}. \label{rso:eq:complexity_xi3} 
\end{align}
\mgtwo{where we recall that 
\begin{equation}\label{eq-constants-in-x-y}
    \Xi_{\tau, \sigma, \theta}^{(i)} = \Xi_{\tau, \sigma, \theta}^{(i, x)} \proxyx^2 + \Xi_{\tau, \sigma, \theta}^{(i, y)} \proxyy^2, \quad \mbox{for} \quad i=1,2,3.
\end{equation}
}
\mgtwo{In the remainder of this proof, we show that each of the above conditions \eqref{rso:eq:complexity_bias_term}-\eqref{rso:eq:complexity_xi3} will be satisfied if $\theta$ is set according to the following two conditions:
\begin{eqnarray}\label{ineq-theta-lb-complexity} 
    \theta 
        &\geq& \mgtwo{\bar{\theta}_x =} 1 - \frac{\mux\varepsilon}{(\left(c_1^x + c_2^x \frac{\Lxy}{\Lyx} +  c_3^x \frac{\Lxy^2}{\Lyx^2}\right) \proxyx^2 (1 + \log(1/(1-p)))},\label{ineq-theta-lb-complexity-x}\\
    \theta 
        &\geq& \mgtwo{\bar{\theta}_y =} 1 - \frac{\muy \varepsilon}{\left(c_1^y + c_2^y \frac{\Lxy}{\Lyx} + + c_3^y \frac{\Lxy^2}{\Lyx^2}\right) \proxyy^2 (1 + \log(1/(1-p)))},\label{ineq-theta-lb-complexity-y}
\end{eqnarray}
for some universal constants $c_1^{x}, c_2^{x}, c_3^{x}, c_1^{y}, c_2^{y}, c_3^{y} \geq 0$ that are large enough. 
Then, we will use these lower bounds on $\theta$ to characterize the number of steps $n$ required to reach $\varepsilon$-accuracy in high probability. 
We next consider each of the four conditions \eqref{rso:eq:complexity_bias_term}-\eqref{rso:eq:complexity_xi3} separately and argue that our choice of stepsize according to \eqref{ineq-theta-lb-complexity-x}- \eqref{ineq-theta-lb-complexity-y} will satisfy each condition.  
}

\paragraph{(I) Satisfying Condition~\eqref{rso:eq:complexity_xi2}.} Based on \eqref{eq-constants-in-x-y}, in order to satisfy \eqref{rso:eq:complexity_xi2} it suffices that $\frac{(1-\theta)}{\theta}\Xi_{\tau, \sigma, \theta}^{(x,2)} \proxyx^2 \leq \frac{\varepsilon}{32}$ and $\frac{(1-\theta)}{\theta}\Xi_{\tau, \sigma, \theta}^{(y,2)} \proxyy^2 \leq \frac{\varepsilon}{32}$. 
Noting that \ylsecond{$\Xi_{\tau, \sigma, \theta}^{(x,2)} = \frac{64}{(1-\theta)} \ylprime{\cQ_x}$}, and using that $\theta \in [1/2, 1]$, we have in view of Lemma~\ref{rso:lem:partial_complexity_bounds},
\[
    \Xi_{\tau, \sigma, \theta}^{(x, 2)} 
        \leq \frac{64}{\mux} \left[ a_1 + a_2 \frac{\Lxy}{\Lyx}\right].
\]
Thus, to ensure that $\frac{(1-\theta)}{\theta}\Xi_{\tau, \sigma, \theta}^{(x,2)} \proxyx^2 \leq \frac{\varepsilon}{32}$, \ylsecond{it suffices to set
\begin{equation}\label{rso:eq:constraint_barbar1}
    \theta \geq  1 - \frac{\mux}{
\left(c_1^x + c_2^x\; \frac{\Lxy}{\Lyx} +  c_3^x\; \frac{\Lxy^2}{\Lyx^2}\right)} \frac{\varepsilon}{\proxyx^2},
\end{equation}
with universal constants $c_1^x, c_2^x>0$ that are large enough and for any constant $c_3^x \geq 0$.
            }
Through similar derivations, one can ensure that $\frac{(1-\theta)}{\theta}\Xi_{\tau, \sigma, \theta}^{(y,2)} \proxyy^2 \leq \frac{\varepsilon}{32}$ can be guaranteed via
\ylsecond{\begin{equation}\label{rso:eq:constraint_barbar2}
    \theta \geq 1 - \frac{\muy}{\left(c_1^y +  c_2^y \frac{\Lxy}{\Lyx} + c_3^y \frac{\Lxy^2}{\Lyx^2}\right)} \frac{\varepsilon}{\proxyy^2}.
\end{equation}
for universal positive constants $c_1^y, c_2^y , c_3^y$ that are large enough. 
We conclude that the choice of stepsize according to \eqref{ineq-theta-lb-complexity-x} and \eqref{ineq-theta-lb-complexity-y} with large enough constants $c_1^x, c_2^x , c_3^x, c_1^y, c_2^y , c_3^y$ will satisfy Condition~\eqref{rso:eq:complexity_xi2}.  
\newline}

\paragraph{(II) Satisfying Condition~\eqref{rso:eq:complexity_xi1}.} Observe now that the constraint~\eqref{rso:eq:complexity_xi1} is satisfied if we ensure the following two conditions:
\[
\begin{split}
    \frac{(1-\theta)}{\theta} \Xi_{\tau, \sigma, \theta}^{(x, 1)} &\leq \frac{\varepsilon}{32 \ylsecond{\proxyx^2}}, \\ 
    \frac{(1-\theta)}{\theta} \Xi_{\tau, \sigma, \theta}^{(y, 1)} &\leq \frac{\varepsilon}{32\ylsecond{\proxyy^2}}.
\end{split}
\]
Among the last two constraints, \mgtwo{using the definitions of $\Xi_{\tau, \sigma, \theta}^{(x, 1)}$ given in \eqref{def-constants-main-thm}}, the first constraint will be satisfied whenever
\[
\begin{split}
    \frac{16 (1-\theta)^2}{\theta \mux} \leq \frac{\varepsilon}{64 \proxyx^2} \quad \mbox{and} \quad
    \frac{(1-\theta)}{\theta} 32 \cQ_x \leq \frac{\varepsilon}{64 \proxyx^2}, \\
\end{split}
\]
and using that $\theta \in [\frac{1}{2}, 1)$, it can be checked that 
both constrains are valid under \ylfourth{assumptions of the same form as~\eqref{rso:eq:constraint_barbar1}}.
Similarly, one may observe that $\frac{(1-\theta)}{\theta} \Xi_{\tau, \sigma, \theta}^{(y, 1)} \leq \frac{\varepsilon}{32}$ is also valid whenever \ylfourth{an assumption of the same form as}~\eqref{rso:eq:constraint_barbar2} is satisfied. 
Therefore, the choice of stepsize according to \eqref{ineq-theta-lb-complexity-x} and \eqref{ineq-theta-lb-complexity-y} with large enough constants $c_1^x, c_2^x , c_3^x, c_1^y, c_2^y , c_3^y$ will satisfy Condition~\eqref{rso:eq:complexity_xi2}. \newline

\paragraph{(III) Satisfying Condition~\eqref{rso:eq:complexity_xi3}.} For Condition~\eqref{rso:eq:complexity_xi3} to hold, it suffices that
\[
\begin{split}
    \gamma_x \leq \frac{\varepsilon}{\ylsecond{256}\; \proxyx^2 \log\left(\frac{1}{1-p}\right)}, \quad 
    \gamma_y \leq \frac{\varepsilon}{\ylsecond{256}\; \proxyy^2 \log\left(\frac{1}{1-p}\right)}. \\
\end{split}
\]
From the definitions of $\gamma_x$ and $\gamma_y$ given in \eqref{eqn-gammas}, we see that the bound on $\gamma_x$ will hold provided that we have
\[
\begin{split}
    \frac{2(1-\theta)}{\mux} &\leq \frac{\varepsilon}{3 \times 256\; \proxyx^2 \log\left(\frac{1}{1-p}\right)}, \\
    16 \cQ_x &\leq \frac{\varepsilon}{3 \times 256\; \proxyx^2 \log\left(\frac{1}{1-p}\right)}, \\
    4 \squarednormtwo{A_1} &\leq \frac{\varepsilon}{3 \times 256\; \proxyx^2 \log\left(\frac{1}{1-p}\right)}, \\
\end{split}
\]
whereas the bound on $\gamma_y$ will hold whenever 
\[
    \begin{aligned}
        \left(\frac{\ylthird{4}(1-\theta)}{\mu_y}\right) 
            &\leq \frac{\varepsilon}{\ylthird{4 \times } \times 256~\proxyy^2 \log \left(\frac{1}{1-\rho}\right)}, \\
        16 Q_y 
            &\leq \frac{\varepsilon}{\ylthird{4 } \times 256~\proxyy^2 \log \left(\frac{1}{1-\rho}\right)}, \\
        4 \left\|A_2\right\|^2 
            &\leq \frac{\varepsilon}{\ylthird{4} \times 256~ \proxyy^2 \log \left(\frac{1}{1-\rho}\right)}, \\
        4 \left\|A_3\right\|^2 
            &\leq \frac{\varepsilon}{\ylthird{4} \times 256 ~ \proxyy^2 \log \left(\frac{1}{1-p}\right)}.
\end{aligned} 
\]
Using Lemma~\ref{rso:lem:partial_complexity_bounds}, after straightforward computations reveals that it suffices to set the stepsize according to \eqref{ineq-theta-lb-complexity-x}- \eqref{ineq-theta-lb-complexity-y} as long as the universal constants $c_1^{x}, c_2^{x}, c_3^{x}, c_1^{y}, c_2^{y}, c_3^{y} \geq 0$ are chosen to be large enough.

\paragraph{(IV) Satisfying Condition~\eqref{rso:eq:complexity_bias_term}.} Finally, consider the last constraint~\eqref{rso:eq:complexity_bias_term}, which is
{\small
\begin{align*}
    &(1-\theta)
    \left(\frac{1+\theta}{2}\right)^n  
    \left(
        4 + \frac{1}{4(1-\theta)} \left(\mux \squarednormtwo{A_1} + \frac{\muy}{2} \left(\squarednormtwo{A_2} + \squarednormtwo{A_3}\right)\right)
    \right) \\
    &\qquad\qquad\qquad\qquad\qquad\qquad\times 
    \frac{1}{2(1-\theta)}
    \left(
        \mux \squarednormtwo{x_0 - \optx} + \muy \squarednormtwo{y_0 - \opty}  
    \right)
        \leq  \frac{\varepsilon}{16}.
\end{align*}}
%
By Lemma~\ref{rso:lem:partial_complexity_bounds}, we have 
\[
\begin{aligned}
    & \frac{\mu_x\left\|A_1\right\|^2}{\ylsecond{4}(1-\theta)} 
        \leq \frac{1}{4} (a_1 + a_3 \frac{\Lxy^2}{\Lyx^2}), \\
    & \frac{\mu_y\left\|A_2\right\|^2}{\ylsecond{4}(1-\theta)} 
        \leq \frac{1}{4} a_1, \\
    & \frac{\mu_y\left\|A_3\right\|^2}{\ylsecond{4}(1-\theta)} 
        \leq \frac{1}{4}(a_1 + a_3 \frac{\Lxy^2}{\Lyx^2}).
\end{aligned}
\]
\ylsecond{Hence, noting that $\theta \in [1/2, 1]$, it suffices to have
\[
\begin{split}
    \left(\frac{1+\theta}{2}\right)^n \left(32 + 6a_1 + 4a_3 \frac{\Lxy^2}{\Lyx^2}\right)\left(\mu_x\left\|x_0-x^*\right\|^2+\mu_y\left\|y_0-y^*\right\|^2\right) \leq \varepsilon. 
\end{split}
\]
}
Thus, it suffices to set the number of iterations to
\[
    n \geq \frac{-1}{\log\left(1 - \frac{1-\theta}{2}\right)} \log\left(\frac{\left( 32 + 6a_1 +  4a_3 \frac{\Lxy^2}{\Lyx^2} \right) \left(\mu_x\left\|x_0-x^*\right\|^2+\mu_y\left\|y_0-y^*\right\|^2\right)}{\varepsilon}\right), 
\]
and noting that $ \frac{-1}{\log(1 - \frac{1-\theta}{2})} \leq \frac{2}{(1-\theta)}$, we can set $\theta = \max(\bar{\theta}_1, \bar{\theta}_2, \bar{\bar{\theta}}_x, \bar{\bar{\theta}}_y)$. The iteration complexity bound~\eqref{intro-complexity} finally follows from the expressions of $\bar{\bar{\theta}}_x$ and $\bar{\bar{\theta}}_y$ given respectively in~\eqref{ineq-theta-lb-complexity-x} and~\eqref{ineq-theta-lb-complexity-x}, and the identities
\[
     \left(1-\bar{\theta}_1\right)^{-1}=\frac{1}{2}\left(\frac{\Lxx}{\mu_x}+1\right)+\sqrt{\frac{1}{4}\left(\frac{\Lxx}{\mu_x}+1\right)^2+\frac{2 \Lyx^2}{\beta \mu_x \mu_y}}, \quad\left(1-\bar{\theta}_2\right)^{-1}=\frac{1}{2}+\sqrt{\frac{1}{4}+\frac{16 \Lyy^2}{(1-\beta)^2 \mu_y^2}}.
\]
%
The case $L_{yy}=0$ can be treated in the same manner where the only difference will be on how the constant $\bar{\theta}_2$ is selected, the details are omitted for brevity.

\subsection{Supplementary Lemmas}\label{sec:asymp}

%
\begin{lem}\label{rso:lem:asymptotic_xi}    
    Under CP parameterization \saa{in \eqref{rso:eq:cb_params},} 
    $\Xi_{\tau, \sigma, \theta}^{(1)} = \Theta(1)$ and $\Xi_{\tau, \sigma, \theta}^{(2)} =\Theta(1)$ as $\theta \to 1$.
\end{lem}
\begin{proof}
    Given the expressions of $\Xi_{\tau, \sigma, \theta}^{(1)}$ and $\Xi_{\tau, \sigma, \theta}^{(1)}$ it suffices to show that $\squarednormtwo{A_1}$, $\squarednormtwo{A_2}$, $\squarednormtwo{A_3}$, $\cQ_x$, and $\cQ_{y}$ are $\bigOh(1-\theta)$ as $\theta \rightarrow 1$, which directly follows from their expression in the CP parameterization, as laid down in Table~\ref{rso:table:cb_A} and Table~\ref{rso:table:cb_Q}.
\end{proof}
\begin{lem}\label{rso:lem:beta_bounds}
    \ylsecond{
        \mgtwo{Consider the choice of \sapdname~parameters in} Theorem~\ref{rso:thm:complexity_sapd_result}. We have
    \[
    \begin{aligned}[t]
        (1-\theta) \frac{\Lxx}{\mux} &\leq 2,
            \quad &(1-\theta) \frac{\Lxy}{\mux} \leq  \frac{\Lxy}{\Lyx}, \\
        (1-\theta) \frac{\Lxy}{\muy} &\leq \frac{\Lxy}{\Lyx},
            \quad & (1-\theta) \frac{\Lyx}{\mux} \leq 1, \\
    \end{aligned}
    \]
    \ylfourth{ and
        \[
            (1-\theta)^2 \frac{\Lyx^2}{\mux \muy} \leq \frac{1}{4}, \quad (1-\theta)^2 \frac{\Lxy \Lyx}{\mux \muy} \leq \ylthird{\frac{1}{4}} \frac{\Lxy}{\Lyx}, \quad (1-\theta)^2 \frac{\Lxy^2}{\mux \muy} \leq \frac{1}{4} \frac{\Lxy^2}{\Lyx^2}.
        \]
    }
    In the case $\Lyy > 0$, we also have $(1-\theta) \Lyy / \muy \leq \frac{1}{4}$.}
\end{lem}
\begin{proof}
    The bound $(1-\theta) \frac{\Lxx}{\mux} \leq 2$ follows from observing that
    \[
        \frac{1}{1-\theta_1} = \frac{1}{2}\left(\frac{\Lxx}{\mu_x}+1\right)+\sqrt{\frac{1}{4}\left(\frac{\Lxx}{\mu_x}+1\right)^2+\frac{2 \Lyx^2}{\beta \mu_x \mu_y}} \geq \frac{\Lxx}{2\mux}, 
    \]
    together with condition~\eqref{rso:eq:cb_param_complexity}.
    \ylsecond{By the subadditivity of $t \mapsto \sqrt{t}$, and observing that the definitions of $\theta_1$ and $\theta_2$ in Corollary \ref{rso:thm:complexity_sapd_result}, together with the condition~\eqref{rso:eq:cb_param_complexity} imply that \ylsecond{$(1-\theta) \leq \frac{\sqrt{\beta \mux \muy/2}}{\Lyx}$, \ylfourth{and, in the case $\Lyy > 0$,} $(1-\theta) \leq \frac{(1-\beta) \muy}{4 \Lyy}$}. Hence, for any $\beta \in [0,1)$, we have
    \[
    \begin{aligned}[t]
        (1-\theta) \frac{\Lxx}{\mux} &\leq 2,
            \quad &(1-\theta) \frac{\Lxy}{\muy} \leq \sqrt{\frac{\beta \mux}{2 \muy}} \frac{\Lxy}{\Lyx}, \\
        (1-\theta) \frac{\Lyy}{\muy} &\leq \frac{1-\beta}{4},
            \quad & (1-\theta) \frac{\Lxy}{\mux} \leq \sqrt{\frac{\beta \muy}{2 \mux}} \frac{\Lxy}{\Lyx}, \\
        (1-\theta) \frac{\Lyx}{\muy} &\leq \sqrt{\frac{\beta \mux}{2 \muy}}, 
            \quad & (1-\theta) \frac{\Lyx}{\mux} \leq \sqrt{\frac{\beta \muy}{2 \mux}}, \\
        (1-\theta)^2 \frac{\Lyx^2}{\mux \muy} &\leq \frac{\beta}{2},
            \quad & (1-\theta)^2 \frac{\Lxy \Lyx}{\mux \muy} \leq \frac{\beta}{2} \frac{\Lxy}{\Lyx}.
    \end{aligned}
    \]
    The result follows from setting $\beta = \min \{1/2, \mux/\mu_y, \muy/\mu_x\}$.} 
\end{proof}
\begin{lem}\label{rso:lem:partial_complexity_bounds}
    For any $\tau, \sigma, \theta$ satisfying the parameterization given in~\eqref{rso:eq:cb_param_complexity}, setting $\beta = \min \{\frac{1}{2}, \frac{\mux}{\muy}, \frac{\muy}{\mux}\}$, there exists positive universal constants $a_1, a_2, a_3$ such that we have
    \[
        \cQ_x \leq \frac{1-\theta}{\mux} \left( a_1 + a_2 \frac{\Lxy}{\Lyx} \right)
            , \quad \cQ_y \leq \frac{1-\theta}{\muy} \left( 
                           a_1
                            +a_2 \frac{\Lxy}{\Lyx}
                            + a_3 \frac{\Lxy^2}{\Lyx^2}
                        \right),
    \]
    and 
    \[
        \squarednormtwo{A_1} \leq \frac{1-\theta}{\mux}  \left(a_1+  a_3 \frac{\Lxy^2}{\Lyx^2}\right), \quad
        \squarednormtwo{A_2} \leq \frac{1-\theta}{\muy} a_1, \quad
        \squarednormtwo{A_3} \leq \frac{1-\theta}{\muy}  \left(a_1 + a_3 \frac{\Lxy^2}{\Lyx^2}\right) .
    \]
\end{lem}
\begin{proof}
Under this choice of parameters, we first observe that the formulas for the constants $\cQ_x, \cQ_y$ and the vectors $A_1, A_2,A_3$ simplify, where the simplified expressions are provided in Table~\ref{rso:table:cb_A} and Table~\ref{rso:table:cb_Q}. 
The proof consists of straightforward algebraic computations and a repeated application of Lemma \ref{rso:lem:beta_bounds} to simplify the bounds. 
\ylfourth{We focus here on the case $\Lyy > 0$, and the result seamlessly extends to the case $L_{yy}=0$, with similar bounds.} 

\ylsecond{In what follows, we systematically apply the \mgtwo{trivial} bounds $\theta \leq 1$, $(1 + \theta) \leq 2$, and $\theta^{-1} \leq 2$ without explicitly mentioning at each time.
\ylthird{First, we observe from the expression of $A_0$ explicitly given in \ylfourth{Lemma~\ref{rso:lem:upperbound_hathaty}} that 
\[
    \squarednormtwo{A_0} \geq \frac{1}{(1+\sigma \muy)^2} \frac{2\theta \sigma}{1-\alpha \sigma} (1 + \sigma (1+\theta) \Lyy)^2 \geq \frac{2\theta \sigma}{(1+\sigma \muy)^2} = \frac{2\theta^2 (1-\theta)}{\muy} \geq \frac{1-\theta}{2 \muy}.
\]
}
Hence, in view of Table~\ref{rso:table:cb_Q}, first note that the CP parameterization gives }
\[
\begin{split}
    B^x 
        &\leq \frac{1-\theta}{\mux}
            \left( 8 
            +  \ylthird{\frac{3}{2}} \frac{1-\theta}{\muy \squarednormtwo{A_0}} (1-\theta)^2 \frac{\Lyx^2}{\mux \muy} \right) 
         \leq \frac{1-\theta}{\mux}
            \left( 8 
            +  3 (1-\theta)^2 \frac{\Lyx^2}{\mux \muy} \right).  \\
\end{split}
\]
and applying Lemma~\ref{rso:lem:beta_bounds} gives
\[
    B^x \leq \frac{1-\theta}{\mux} (8+\frac{3}{4}).
\]
Similarly,  using the expressions for $C^x$ and $C_{-1}^x$ from Table~\ref{rso:table:cb_Q} we have 
\[
\begin{split}
    C^x &\leq \frac{1-\theta}{\mux} \left(1 + \frac{{\color{black}{3}}}{2} (1-\theta) \frac{\Lxy}{\muy}\right)  \leq \frac{1-\theta}{\mux} \left(1 + \frac{{\color{black}{3}}}{2} \frac{\Lxy}{\Lyx} \right), \\
    C_{-1}^x &\leq \frac{1-\theta}{\mux} (1-\theta) \frac{\Lyx}{\muy} \leq \frac{1-\theta}{\mux}.\\
\end{split}
\]
Hence, combining these bounds and using Lemma~\ref{rso:lem:beta_bounds}, we obtain
\[
    \cQ_x \mgtwo{= B^x + C^x +  C_{-1}^x}  \leq \frac{1-\theta}{\mux} \left(a_1+ a_2 \frac{\Lxy}{\Lyx} \right).
\]
provided that $a_1 \geq 10 + \frac{3}{4}$ and $a_2 \geq \frac{3}{2}$.
We upperbound $B^y$ using the same rationale. 
We first note that $1-\alpha\sigma = \frac{1}{2} + \frac{(1-\theta) \Lyy}{\sqrt{\theta} \muy} \geq \frac{1}{2}$, and $C_{\sigma, \theta} \leq 1 + 4 (1-\theta) \frac{\Lyy}{\muy} \leq 2$. Thus $C_{\sigma, \theta}^2 \leq 4$, and we have
\[
        \squarednormtwo{A_0} 
            \leq \frac{1-\theta}{\muy} 
                \left(
                    8 (1-\theta)^2 \frac{\Lyx^2}{\mux \muy} 
                    + 4 C_{\tau, \sigma}^2  
                    + 2 (1-\theta)^2 \frac{\Lyx^2}{{\color{black}\mux \muy}} 
                    + 4 (1-\theta)^2 \frac{\Lyy^2}{\muy^2}
                \right)
             \leq \frac{1-\theta}{\muy} a_4   
\]
for some universal constant $a_4>0$ where we used Lemma \ref{rso:lem:beta_bounds}.
Hence, $B^y$ can be bounded as
\[
    B^y \leq \frac{1-\theta}{\muy} 
        \left(
            a_5
            + 12 \frac{\Lyx^2 \Lxy^2}{\mux^2 \muy^2} 
        \right)
        \leq \frac{1-\theta}{\muy} 
        \left(
            a_5 + \frac{3}{4} \frac{\Lxy^2}{\Lyx^2}
        \right).\\            
\]
for some constant $a_5>0$, using Lemma \ref{rso:lem:beta_bounds} again.
Similarly, we have the bounds
\[
    \begin{split}
        B_{-1}^y \leq 
            \frac{1-\theta}{\muy}
            \left(
                12 
                + \frac{3}{16} \frac{\Lxy^2}{\Lyx^2}
            \right), \quad 
        C_y \leq \frac{1-\theta}{\muy} 
                \left(6 + \frac{\Lxy}{\Lyx} \right), \quad 
        C_{-2}^y \leq \frac{1-\theta}{\muy} 
            \left(
                \frac{3}{4} + \frac{1}{4} \frac{\Lxy}{\Lyx}
            \right), \\
    \end{split}
\]
Analogous bounds for $C_{-1}^y$ can also be obtained. Combining all these estimates, 
%
\[
    Q_y = B^y + B_y^{-1} + C_y + C_{-1}^y + C_{-2}^y \leq \frac{1-\theta}{\muy} 
            \left( 
                a_1
                +a_2 \frac{\Lxy}{\Lyx}
                +a_3 \frac{\Lxy^2}{\Lyx^2}
            \right) 
\]
as long as $a_1, a_2, a_3$ are positive (universal) constants that are large enough.
The other bounds for $\|A_1\|^2, \|A_2\|^2, \|A_3\|^2$ are obtained using tedious but similar arguments, based on their expression under the CP parameterization provided in Table~\ref{rso:table:cb_A}.
\end{proof}

%% file: B3_constants.tex
\mgtwo{The constants $A_0, A_1, A_2, A_3$ provided in Lemmma~\ref{rso:lem:upperbound_hathaty} and Table~\ref{rso:table:constants_A}} simplify under the CP parameterization specified in~\eqref{rso:eq:cb_params}. Under this particular choice of parameters, the constants $\cQ_x$ and $\cQ_y$ given in Table~\ref{rso:table:constants_C} of the main paper can also be simplified; the details are provided in the following tables.
\ylfourth{These expressions are particularly useful for the derivation of our complexity result, stated in Corollary \ref{rso:thm:complexity_sapd_result}}.
%
%
\begin{table}[ht]
{
\begin{tabular}{l} 
\hline \\[-1.ex]
        \vspace{1em}
            \makecell[l]{\scalebox{0.9}{\ylthird{$
                \makecell[l]{\scalebox{0.9}{\ylfourth{
                $H \defineq \text{Diag} \left[
                        \sqrt{\frac{2}{\mux}},\;
                        \frac{\sqrt{2/\muy}}{\sqrt{\frac{1}{2} + \frac{(1-\theta) \Lyy}{\sqrt{\theta} \muy}}},\;
                        \sqrt{\frac{2}{\mux}},\;
                        \frac{\sqrt{2/\muy}}{\sqrt{\frac{1}{2} + \frac{(1-\theta) \Lyy}{\sqrt{\theta} \muy}}}
                        \right], \quad A_1= 
                    \ylthird{\sqrt{1-\theta}\; 4 \sqrt{\theta(1+\theta)} \;
                    \ylfourth{H}
                    \left[\begin{array}{l}
                        \begin{aligned}[t]
                                    \bigg( 1 &+\frac{1}{\theta}(1-\theta) \frac{L_{x x}}{\mux} \\
                                      &+\frac{(1+\theta)}{\theta}(1-\theta)^2 \frac{\Lxy \Lyx}{\mux \muy}\bigg)
                                \end{aligned}
                            \\
                        (1-\theta) \frac{\Lxy}{\mux} C_{\sigma, \theta} \\
                        (1-\theta)^2 \frac{\Lxy \Lyx}{\mux \muy}  \\
                        (1-\theta)^2 \frac{\Lxy \Lyy}{\mux \muy}\\
                    \end{array}\right]}$
                }}} $}}} \\
        \vspace{1em}
            \makecell[l]{\scalebox{0.8}{
                $A_2 = 
                    \ylthird{\sqrt{1-\theta} 4 \sqrt{2} \sqrt{\theta(1+\theta)} (1+\theta)
                    \; \ylfourth{H}
                    \left[\begin{array}{l}
                            \frac{(1+\theta)}{\theta} (1-\theta) \frac{\Lyx}{\muy} \\
                            C_{\sigma, \theta} \\
                            (1-\theta) \frac{\Lyx}{\muy} \\
                            (1-\theta) \frac{\Lyy}{\muy} \\
                    \end{array}\right]},$}}\\
        \vspace{1em}
                \makecell[l]{\scalebox{0.9}{$
                A_3 = 
                    \ylthird{\sqrt{1-\theta} \sqrt{1+\theta} 4 \sqrt{2} \theta^{1/2}
                    \; \ylfourth{H}
                    \left[\begin{array}{l}
                        \bigg( 
                            \begin{aligned}[t]
                                & \left(1 + (1+\theta) +  {\color{black}(1+\theta)} \right) (1-\theta) \frac{\Lyx}{\muy} 
                                   + \frac{(1+\theta)^2}{\theta} (1-\theta)^2 \frac{\Lyx \Lyy}{\muy^2} \\
                                & + \frac{1+\theta}{\theta} (1-\theta)^2 \frac{\Lxx \Lyx}{\mux \muy}
                                  + \frac{(1+\theta)^2}{\theta} (1-\theta)^3 \frac{\Lyx^2 \Lxy}{\mux \muy^2} \bigg) \\
                            \end{aligned}                            
                            \\
                        \bigg(
                                \begin{aligned}[t]
                                     \theta 
                                &+ (1 + 2 (1+\theta)) (1-\theta) \frac{\Lyy}{\muy} 
                                + \frac{(1+\theta)^2}{\theta} (1-\theta)^2 \frac{\Lyy^2}{\muy^2} \\
                                    & + (1-\theta)^2 (1+\theta) \frac{\Lxy \Lyx}{\mux \muy} 
                                +  (1-\theta)^3 \frac{(1+\theta)^2}{\theta} \frac{\Lyx \Lxy \Lyy}{\mux \muy^2} \bigg) 
                                \end{aligned}
                            \\
                        \left(\theta(1-\theta) \frac{L_{y x}}{\mu_y}+(1+\theta)(1-\theta)^2 \frac{L_{y x} L_{y y}}{\mu_y^2}+(1+\theta)(1-\theta)^3 \frac{L_{x y} L_{yx}^2}{\mu_x \mu_y^2}\right)
                            \\
                        \left(\theta(1-\theta) \frac{L_{y y}}{\mu_y}+(1+\theta)(1-\theta)^2 \frac{L_{y y}^2}{\mu_y^2}+(1+\theta)(1-\theta)^3 \frac{L_{x y} L_{yx} \Lyy}{\mu_x \mu_y^2}\right)
                            \\
                    \end{array}\right]}.$}} \\
            \hline
\end{tabular}}
\vspace{.5em}
\caption{\label{rso:table:cb_A} Simplifications of $A_1, A_2, A_3$ from Table~\ref{rso:table:constants_A} under the CP parameterization in~\eqref{rso:eq:cb_params}.} 
\end{table}
%
%
\begin{table}[ht]
{
\begin{tabular}{l} 
\hline \\[-1.ex]        
    \vspace{1em}\makecell[l]{$
            A_0 = \sqrt{\frac{1-\theta}{\muy}} \left[\begin{array}{l}
                    \sqrt{2}(1+\theta) (1-\theta) \frac{\Lyx}{\sqrt{\mux \muy} } \\
                    \frac{\sqrt{2} \theta}{\sqrt{\frac{1}{2}+\frac{(1-\theta) \Lyy}{\sqrt{\theta} \muy}}} \bigg(1 + \frac{1+\theta}{\theta} (1-\theta) \frac{\Lyy}{\muy} \bigg) \\                        
                    \\
                    \sqrt{2} \ylsecond{\theta} (1-\theta) \frac{\Lyx}{\ylsecond{\sqrt{\mux \muy}}} \\
                    \frac{\sqrt{2} \theta}{\sqrt{\frac{1}{2}+\frac{(1-\theta) \Lyy}{\sqrt{\theta} \muy}}} (1-\theta) \frac{\Lyy}{\muy}
                \end{array}\right]$
    } \\
    \vspace{1em} \makecell[l]{
        $\ylthird{\mgtwo{C_{\sigma, \theta} = \left(1+\frac{(1-\theta)(1+\theta)}{\theta} \frac{\Lyy}{\muy}\right)}, \quad B^x = \frac{1-\theta}{\mux} \left(\frac{4}{\theta (1+\theta)} + \frac{1}{\squarednormtwo{A_0}} \frac{3 \theta (1+\theta) (1-\theta)}{4} (1-\theta)^2 \frac{ \Lyx^2}{\mux \muy^2}\right)},$} \\
        \makecell[l]{$
            \scalemath{0.85}{
                \ylthird{
                    B^y=
                        \frac{1-\theta}{\muy}
                        \bigg(
                            \frac{4(1+\theta)}{\theta \left(\frac{1}{2} + \frac{(1-\theta) \Lyy}{\sqrt{\theta} \muy}\right)} 
                            + \frac{3 \squarednormtwo{A_0}(1+\theta) \muy}{ \theta (1-\theta)} 
                            \begin{aligned}[t]
                                & + \frac{\theta^3 (1+\theta)(1-\theta)}{2 \muy \squarednormtwo{A_0}} 
                                \left(
                                    1 + \frac{2 (1-\theta) (1+\theta)}{\theta \muy} \Lyy + \frac{(1-\theta)^2}{\theta^2 \muy^2} (1+\theta)^2 \Lyy^2
                                \right) \\
                                & + \frac{3 \theta (1+\theta)^3}{4 \squarednormtwo{A_0}} \frac{(1-\theta)^5}{\muy} \frac{\Lyx^2 \Lxy^2}{\mux^2 \muy^2} \bigg),
                            \end{aligned}     
                    }
                }
            $} \\ 
        \makecell[l]{$
            \scalemath{0.85}{
                \ylthird{
                    B_{-1}^y =
                        \frac{1-\theta}{\muy} \bigg(
                            \begin{aligned}[t]
                                \frac{4}{(1+\theta) (1-\alpha \sigma)}
                                &+ \frac{\theta^4 (1-\theta)}{2 \squarednormtwo{A_0} (1+\theta) \muy} \left(
                                    1 + \frac{2 (1-\theta) (1+\theta)}{\theta \muy} \Lyy + \frac{(1-\theta)^2}{\theta^2 \muy^2} (1+\theta)^2 \Lyy^2
                                \right) \\
                                &+ \frac{1-\theta}{\muy} \frac{3(1+\theta) \theta^2}{4 \squarednormtwo{A_0}} (1-\theta)^4 \frac{\Lyx^2 \Lxy^2}{\mux^2 \muy^2} \bigg), \\
                            \end{aligned}                        
                }
            }
            $} \\ 
        \makecell[l]{$
            \scalemath{0.85}{
                \ylthird{C^x = \frac{1-\theta}{\mux} \left(1 + \frac{1+2\theta}{2} (1-\theta) \frac{\Lxy}{\muy}\right)\quad, 
                C_{-1}^x = \frac{1-\theta}{\mux}
                    \left(
                        \frac{1+\theta}{2} (1-\theta) \frac{\Lyx}{\muy}
                    \right), \quad
                    C^y= \frac{1-\theta}{\muy} \left((1+2\theta)(1+\theta) + \frac{1+\theta}{2} (1-\theta) \frac{\Lxy}{\mux}\right),
                }
            }
            $} \\ 
        \makecell[l]{$
            \scalemath{0.85}{
                \ylthird{
                    C_{-1}^y = 
                        \frac{(1-\theta)}{\muy}\left(1+ \ylthird{2\theta} 
                        + \frac{1+\theta}{\ylthird{2}} (1-\theta) \frac{\Lyx}{\mux}
                        + \frac{1}{\ylthird{2}} (1-\theta) \frac{\Lxy}{\mux}
                        +(1+ \frac{\ylthird{3 \theta}}{2} )\left(\theta+\frac{(1-\theta) (1+\theta) \Lyy}{\muy}+\frac{(1-\theta)^2(1+\theta) \Lyx \Lxy}{\mux \muy}\right)\right),                       
                }
            }
            $} \\ 
        \makecell[l]{$
            \scalemath{0.85}{
                \ylthird{
                    C_{-2}^y = 
                        \frac{(1-\theta)}{\muy}
                            \left(
                                \frac{\theta}{\ylthird{2}} 
                                +  \frac{(1+\theta)}{\ylthird{2}} (1-\theta) \frac{\Lyy}{\muy}
                                + \frac{(1+\theta)}{\ylthird{2}} (1-\theta)^2 \frac{\Lyx \Lxy}{\mux \muy}
                            \right),  
                    }
            }
            $} \\
           \makecell[l]{
           \yl{$\mgtwo{\cQ_x = B^x + C^x + C_{-1}^x, \quad \cQ_y = B^y + \ylsecond{B_{-1}^y} + C^y + C_{-1}^y + C_{-2}^y}$}
        }\\[1.ex]
\hline
\end{tabular}}
\vspace{.5em}
\caption{\label{rso:table:cb_Q} Simplifications of $\cQ_x, \cQ_y$, and $A_0$ under the CP parameterization in~\eqref{rso:eq:cb_params}.}
\end{table}